\documentclass[a4paper,11pt]{amsart}

\makeatletter
\providecommand{\@LN}[2]{}
\makeatother
\usepackage{quiver}
\usepackage{GabLob}
\usepackage[margin=3cm]{geometry}
\newcommand{\brdgrp}{\mathcal{B}}
\newcommand{\symgrp}{\mathcal{S}}
 \usepackage{xcolor}

 \def\black{\color{black}}

\date{\today}

\makeindex
\title[A finite approach to multicategories]{\textbf{A finite approach to representable multicategories and related structures}}
\author{Gabriele Lobbia}
\thanks{The author is deeply grateful to John Bourke for bringing up the problem and for the significant contribution, suggestions and feedback. 
The majority of this research formed part of the author's Phd project, supported by an EPSRC Scholarship, under the very helpful supervision of Nicola Gambino.}

\address{University of Leeds, 
Leeds, United Kingdom \newline
Masaryk University, 
Department of Mathematics and Statistics, Masaryk University, Kotl\'a\v rsk\'a 2, Brno 61137, Czech Republic}
\email{lobbia.math@gmail.com}

\begin{document}
\begin{abstract}
It is known that monoidal categories have a finite definition, whereas multicategories have an infinite (albeit finitary) definition. Since monoidal categories correspond to representable multicategories, it goes without saying that representable multicategories should also admit a finite description. With this in mind, we give a new finite definition of a structure called a short multicategory, which only has multimaps of dimension at most four, and show that under certain representability conditions short multicategories correspond to various flavours of representable multicategories. This is done in both the classical and skew settings.
\end{abstract}

\maketitle 
\tableofcontents


\section*{Introduction}

In the zoo of categorical structures, there are three closely related ones:
\begin{itemize}
\item monoidal categories \cite{Maclane-monoidal}, which involve tensor products $A \otimes B$ and a unit~$I$;
\item closed categories \cite{EilenbergS:cloc}, which involve an internal hom $[A,B]$ and unit~$I$;
\item multicategories \cite{Lambek-multicategories}, which involve multimorphisms $A_1,\ldots,A_n \to B$ for all $n \in \mathbb N$.
\end{itemize}
It is well known that there are various correspondences between different flavours of these notions \cite{Hermida2000Representable, Manzyuk2012Closed, LackBourke:skew,CruttShul:gen-framw-multi,StatLev:univ-prop-impure-progr,HERMIDA20017}. For instance, representable multicategories are equivalent to monoidal categories \cite{Hermida2000Representable}. Then, if we want to weaken the representability condition, for example considering only \emph{left representable} multicategories, we end up in the world of \emph{skew} monoidal categories. In \cite{LackBourke:skew} we can find the details of various equivalences between (skew) multicategories and skew monoidal categories. 

Hence, we get different equivalent concepts, each of which has various pros and cons.
\begin{itemize}
\item Monoidal categories are fairly straightforward to work in --- for instance, it is easy to write down the definition of a monoid in a monoidal category.  Another advantage is that while the definition is finite, they admit a coherence theorem --- all diagrams commute \cite{Maclane-monoidal}.  A disadvantage is that in practise, the tensor product is often constructed using colimits and so sometimes could be hard to describe explicitly. 
\item Closed categories have several advantages. Again they have a finite definition and a coherence theorem, though this is of a more complex nature \cite{KelMac:coherence-closed-cat, Soloviev1997-SOLPOA}.  Another advantage is that the internal homs are often constructed using limits, and so easy to describe explicitly --- see, for instance, the internal hom of vector spaces. The disadvantage is that the axiomatics of closed categories involve iterated contravariance, and this makes it quite hard to parse diagrams in a closed category.
\item In a multicategory, the multimaps can often be described directly --- see, for instance, multilinear maps of vector spaces --- and this avoids the potentially complicated constructions of tensor products and internal homs using colimits and limits.  A disadvantage is that the definition, is infinite (though finitary) in nature, and this sometimes makes it difficult to describe examples in full detail.
\end{itemize}

Since monoidal categories admit a definition involving finite data and finite axioms, it is natural to wonder if the same is possible for multicategories. Our goal in the present paper is to describe a finite approach to the kinds of multicategory that arise in practise --- these include representable and closed multicategories --- with the goal of making examples of such notions easier to construct.  We do this by introducing a structure called a \emph{short multicategory}, which is not itself a multicategory, since it indeed only has multimaps of dimension at most $4$. One of our main results shows that representable short multicategories are equivalent to representable multicategories, so providing a finite description of the latter.  Moreover, we adapt all of these results to the setting of skew multicategories and skew monoidal categories described in \cite{LackBourke:skew}. Our results make it easier to construct examples of (skew) multicategorical structures in practice. For instance, they made significantly more manageable the definition of the skew structures on the category of Gray-categories described in \cite{BourkeLob:SkewApp}.\footnote{We remark that in \cite{BourkeLob:SkewApp} \emph{4-ary (skew) multicategories} are used. These have an underlying short (skew) multicategorical structure, but the latter involves less data and axioms. For this reason in this paper we focus on the \emph{short} notion. }


\subsection*{Main Results}
The main contribution of this work is to provide equivalences between different flavours of short (skew) multicategories and (skew) multicategories. \black 
In particular, we consider the following cases:
\begin{itemize}
\item Theorem~\ref{thm:rep-mult} provides an equivalence between representable multicategories and representable short multicategories. We prove this as a consequence of the more general Theorem~\ref{thm:fin-equiv}, which deals with left representable short multicategories. 

\item Theorem~\ref{thm:closed-lr-multi} and Theorem~\ref{thm:cl-rep-equiv} show the equivalences in the closed left representable and closed representable case.

\item Then, Theorem~\ref{them:sk-fin-equiv} proves the left representable skew case.

\item Theorem~\ref{thm:skew-left-closed-equiv} is about the left representable closed skew case.

\item \black
Finally, Theorem~\ref{thm:brd-sk-equiv} provides an equivalence between braided/symmetric left representable short skew multicategories and braided/symmetric skew monoidal categories. From this it follows that braidings on a left representable skew multicategory have a finite presentation (Corollary~\ref{cor:brd-bij-corr}).  We conclude with Theorem~\ref{thm:brd-equiv}, which proves the braided/symmetric result for short multicategories. 
\black 
\end{itemize}

We also show that these equivalences are compatible with the ones given in \cite{LackBourke:skew,BouLack:skew-braid,Hermida2000Representable} for different flavours of multicategory and monoidal category.

\subsection*{Overview}
\black 
In Section~\ref{sec:clas-mult} we review the definition of a multicategory, before giving a slight reformulation of it better suited for our later use. We also recall some important notions for multicategories, such as representability and closedness.

In Section~\ref{sec:fin-mult} we use the reformulation given in Section~\ref{sec:clas-mult} to define \emph{short multicategories}.  We then define the notions of representability and closedness in the context of short multicategories. 

In Section~\ref{sec:skew-notions} we give an overview on skew monoidal categories and skew multicategories, including the notions of left representability, closedness and braiding/symmetry in this context. 

Section~\ref{sec:short-vs-mult} provides various equivalences between different flavour of short multicategories and skew monoidal categories. 

We conclude the paper in Section~\ref{sec:short-skew-mult} introducing short skew multicategories and describe analogues of the results in Section~\ref{sec:short-vs-mult} appropriate to the skew setting \black and further considering the braided case. \black 

\black For the interested reader, we leave to Appendix~\ref{sec:closed-case} a discussion on the \emph{just} closed case (not left representable), which leads to similar equivalences as well. Since most of the examples in the literature are also left representable, we left the treatment of this particular case to the appendix.
In Appendix~\ref{app:nat-conditions} we write explicitly some naturality conditions. 

%
%
%
\black 

\section{Classical Multicategories}
\label{sec:clas-mult}

In this section we will recall the definitions of multicategories and morphisms between them. To begin with, a \textbf{multicategory} $\mlc$ consists of:
\begin{itemize}
\item a collection of objects;
\item for each (possibly empty) list $a_1,...,a_n$ of objects and object $b$, a set $\mlc_n(a_1,...,a_n;b)$;
\item for each object $a$ an element $1_a\in\mlc_1(a;a)$.
\end{itemize}

The elements of the set $\mlc_n(a_1,...,a_n;b)$ are called $n$-ary multimaps, with domain the list $a_1,\ldots,a_n$ and codomain $b$, whilst $1_a$ plays the role of the identity unary morphism. We sometimes write $\overline{a}$ for the list, and then $\mlc_n(\overline{a};b)$ for the set of multimaps.

Substitution in a multicategory can be encoded in two ways. The best known one involves substitutions into all positions simultaneously. In this case, substitution is encoded by functions of the form
\begin{align*}
\mlc_n(b_1,\ldots,b_n;c)\times\prod_{i=1}^n\mlc_{k_i}(\overline{a_i};b_i) &\longrightarrow\mlc_K(\overline{a_1},\ldots,\overline{a_n};c) \\
(g,f_1,\ldots,f_n)&\longmapsto g \circ (f_1,\ldots,f_n)
\end{align*}
where $K={\sum_{i=1}^nk_i}$. For such substitutions, there is a straightforward associativity axiom --- see, for instance, Definition 2.1.1 of \cite{higherop} --- and two identity axioms, which at $g \in \mlc_n(a_1,\ldots,a_n;b)$ are captured by the two equations $1_b \circ (g)=g=g \circ (1_{a_1},\ldots,1_{a_n})$.

The original definition of multicategory, due to Lambek \cite{Lambek-multicategories}, instead involved substitutions into a single position, and these are encoded by functions of the following form:
\begin{align*}
-\circ_i-\colon \mlc_n(\overline{b};c) \times \mlc_m(\overline{a};b_i) &\to \mlc_{n+m-1}(\overline{b}_{<i},\overline{a},\overline{b}_{>i};c) \\
(g,f) &\mapsto g \circ_i f
\end{align*}
where $\overline{b}_{<i}$ and $\overline{b}_{>i}$ denote the sublists of $\overline{b}$ in indices strictly less than or greater than $i$, respectively. To encode associativity of the $\circ_i$-type substitutions, one requires the following two collections of equations (the first referred to as associative law and the second as commutative law in \cite{Lambek-multicategories})
\begin{align*}
h \circ_i(g\circ_j f) = (h\circ_i g)\circ_{j+i-1}f  \hspace{0.5cm} &\textnormal{for} \hspace{0.5cm} 1 \leq i \leq m, 1 \leq j \leq n   
\\
(h\circ_i g)\circ_{n+j-1}f=(h \circ_j f) \circ_{i} g \hspace{0.5cm} &\textnormal{for} \hspace{0.5cm} 1 \leq i<j \leq m. 
\end{align*}
Finally, there are the two identity axioms which at $g \in \mlc_n(a_1,\ldots,a_n;b)$ are captured by the equations $1_b \circ_1 g = g = g\circ_i 1_{a_i}$. 

Given a multicategory with $\circ$-type substitutions, the corresponding $\circ_i$ is defined by 
$$g \circ_i f = g \circ (1,\ldots,1,f,1,\ldots,1)$$
where $f$ is substituted in the $i$'th position.  Given a multicategory with $\circ_i$-type substitutions, the corresponding $\circ$ is defined by
$$g\circ (f_1,\ldots,f_n) = (\ldots((g\circ_1 f_1) \circ_{k_{1}+1} f_2)\ldots \circ_{k_{1} + \ldots k_{n-1} +1} \circ f_n$$

Each multicategory $\mlc$ has an underlying category $U\mlc$ with the same objects, and morphisms the unary ones, so that one can consider a multicategory $\mlc$ as a category $U\mlc$ equipped with additional structure. Thinking of a multicategory $\mlc$ as a category equipped with $\circ_i$-type substitution, we obtain the following reformulations, which will be our starting point in which follows.  It is closely related to Proposition 3.4 of \cite{LackBourke:skew}.

\begin{prop}\label{prop:classical}
A multicategory $\mlc$ is equivalently specified by:
\begin{itemize}
\item a category $\catc$;
\item for $n \in \mathbb N$ a functor $\mlc_n(-;-)\colon(\catc^{n})^{op} \times \catc \to \Set$ such that, when $n=1$, we have $\mlc_1(-;-)=\catc(-,-)\colon\catc^{op} \times \catc \to \Set$; 
\item substitution functions, for $i \in \{1,\ldots,n\}$, 
$$\circ_i\colon \mlc_n(\overline{b};c) \times \mlc_m(\overline{a};b_i) \to \mlc_{n+m-1}(\overline{b}_{<i},\overline{a},\overline{b}_{>i};c),$$
which are natural in each variable $a_1,\ldots, a_m,b_1,\ldots b_{i-1},b_{i+1},\ldots,b_n,c$, dinatural\footnote{We leave the precise diagrams relative to these (di)naturality requirements in Appendix~\ref{app:nat-dinat-axioms}.} in $b_i$ and satisfying the same axioms
\begin{align}
h \circ_i(g\circ_j f) = (h\circ_i g)\circ_{j+i-1}f  \hspace{0.5cm} &\textnormal{for} \hspace{0.5cm} 1 \leq i \leq m, 1 \leq j \leq n   \label{ax:classic-ass-line}\\
(h\circ_i g)\circ_{n+j-1}f=(h \circ_j f) \circ_{i} g \hspace{0.5cm} &\textnormal{for} \hspace{0.5cm} 1 \leq i<j \leq m \label{eq:classic-ass-not-line} 
\end{align}
as before. 
Moreover, 
we require that, for any unary maps $p\colon a'_i\to a_i$ and $q\colon b\to b'$ and $n$-ary map $f\colon\overline{a}\to b$,\footnote{This condition says that the action of $\mlc_n(-;-)$ must be substitution with unary morphisms.}
\begin{equation}
	\label{circ_i-comp-with-fct-act}
	q\circ_1f=\mlc_n(\overline{a};q)(f)\quad \text{and}\quad f\circ_i p=\mlc_n(\overline{a}_{<i},p,\overline{a}_{>i};b)(f). 
\end{equation}
\black 
In this way, $U\mlc=\catc$. 
\end{itemize}
\end{prop}
\begin{proof}
Given a structure as above, we can form a multicategory with objects those of $\catc$, sets of multimaps $\mlc_n(\overline{a};b)$, identities $1_a \in \catc(a,a)=\mlc_1(a;a)$ and substitution functions $\circ_i$ as above.
Finally we notice that, by \eqref{circ_i-comp-with-fct-act}, the identity axioms for the multicategory are encoded by the fact that the functor $\mlc_n(-;-)\colon(\catc^{n})^{op} \times \catc \to \Set$ preserves identities. \black  

In the opposite direction, given a multicategory $\mlc$ with $\circ_i$-type operations, let $\catc$ be its underlying category of unary morphisms.  We must define a functor
$$\mlc_n(-;-)\colon(\catc^{n})^{op} \times \catc \to \Set$$
 sending $(\overline{a};b)$ to $\mlc_n(\overline{a};b)$ on objects in such a way that the $\circ_i$ substitutions are natural in the sense described above, and such that $\mlc_1(-;-)=\catc(-,-)$.  In fact, the requirement~for 
$$\circ_i\colon \mlc_1({b};c) \times \mlc_n(\overline{a};b) \to \mlc_{n}(\overline{a};c)$$
to be dinatural in $b$ (and identity) forces us to define $\mlc_n(\overline{a};f) = f \circ_1 -$.  Naturality of the $\circ_i$ also ensure naturality of the associated $\circ$ operations, and in particular naturality of
$$\circ\colon\mlc_n(a_1,\ldots,a_n;b) \times \mlc_1(c_1,a_1) \times \ldots \mlc_1(c_n;a_n) \to \mlc_n(c_1,\ldots,c_n;b)$$ in $a_1,\ldots,a_n$ forces us similarly to define $\mlc_n(f_1,\ldots,f_n;b) = - \circ (f_1,\ldots,f_n)$.  With this definition of $\mlc_n(-;-)$ on morphisms, associativity and the identity axioms of substitutions implies that it is a functor and that the substitution maps are natural in each variable, and satisfy $\mlc_1(-;-)=\catc(-,-)$.

These two constructions are inverse.
\end{proof}

\begin{rmk}
	Alternatively, one could also start with a collection of objects $\catc_0$, sets $\mlc_{n}(\overline{a};b)$ and substitutions $\circ_i$ as above and derive from this data the functor structure on $\mlc_{n}(-;-)$. 
	Indeed, as seen in the proposition above, the action of $\mlc_{n}(-;-)$ on morphisms is forced to be pre/post-composition. 
	\black 
\end{rmk}

\begin{rmk}
	Another possible change to the description in Proposition~\ref{prop:classical} is to replace condition \eqref{circ_i-comp-with-fct-act} with the identity axioms, for any $n$-ary multimap $f$ with $n\geq 2$, 
	\begin{center}
		$1\circ_1f=f=f\circ_i1$ for $1 \leq i \leq n$. 
	\end{center}
\end{rmk}

Naturally, there is a notion of morphism between multicategories, which we call here \emph{multifunctor}. From now on, when talking about multicategories we will mean in the sense of Proposition~\ref{prop:classical}.

\begin{defn}
Let $\mlc$ and $\md$ two multicategories. A \emph{multifunctor} is a functor $F\colon\catc\to\catd$ together with natural families
$$F_n\colon\mlc_n(\overline{a};b)\to\md_n(F\overline{a};Fb)$$
for any $n\in\mathbb{N}$, such that when $n=1$, then $F_1$ is the functor action. These families must commute with all substitution operators $\circ_i$. 
\end{defn}

Multicategories and multifunctors form a category $\multi$. \black

\subsection{Representability}

An important notion for multicategories is the one of representability \cite{Hermida2000Representable}. Firstly, a \textbf{$\mathbf{n}$-ary map classifier} for $\overline{a}=a_1,\ldots,a_n$ consists of a representation of $\mlc_n(a_1,\ldots,a_n;-):\catc \to \Set$ -- in other words, a multimap $$\theta_{\overline{a}}\colon a_1,\ldots,a_n \to m(a_1,\ldots,a_n)$$ for which the induced function $-\circ \theta_{\overline{a}}\colon \mlc_1(m(a_1,\ldots,a_n);b) \to \mlc_n(a_1,\ldots,a_n;b)$ is a bijection for all $b$.  We sometimes refer to such a multimap as a universal multimap and write $m\overline{a}$ for $m(a_1,\ldots,a_n)$. A $n$-ary map classifier is said to be \emph{left universal} if, moreover, the induced function 
$$-\circ_1 \theta_{\overline{a}}\colon\mlc_{1+r}(m\overline{a},\overline{y};d) \to \mlc_{n+r}(a_1,\ldots,a_n,\overline{y};d)$$
is a bijection for any $\overline{y}$ of length $r$.

\begin{defn}[\cite{Hermida2000Representable, LackBourke:skew}]
\label{def:lr-rep-multi}
Let $\mlc$ be a multicategory.
\begin{itemize}
\item $\mlc$ is said to be \textbf{weakly representable} when each of the functors $\mlc_n(\overline{a};-)\colon\catc\to\Set$ is representable, i.e. if it has all $n$-ary map classifiers $\theta_{\overline{a}}$. 

\item $\mlc$ is said to be \textbf{left representable} if it is weakly representable and all $\theta_{\overline{a}}$ are left universal. 

\item $\mlc$ is said to be \textbf{representable} if it is weakly representable and substitution with universal $n$-multimaps $\theta_{\overline{a}}$ induces bijections
$$\mlc_{k+1}(\overline{x},m\overline{a},\overline{y}; b)\to\mlc_{k+n}(\overline{x},\overline{a},\overline{y}; b)$$
for $\overline{x}$ and $\overline{y}$ tuples of appropriate length. 
\end{itemize}
\end{defn}

We will denote with $\multilr$ and $\multirep$ the full subcategories of $\multi$ with objects, respectively, left representable multicategories and representable multicategories. 

\subsection{Closedness} 

Another important notion for multicategories is the one of closedness.

\begin{defn}
\label{def:multi-closed}
A multicategory $\mlc$ is said to be \textbf{closed} if for all pair of objects $b$ and $c$ there exists an object $[b,c]$ and binary map $e_{b,c}\colon [b,c],b \to c$ for which the induced function 
$$e_{b,c}\circ_1-\colon\mlc_n(\overline{x};[b,c]) \to \mlc_{n+1}(\overline{x},b;c)$$ 
is a bijection, for any tuple $\overline{x}$ of length $n$.  
\end{defn} 

We will denote with $\multilrcl$ the full subcategory of $\multi$ with objects left representable closed multicategories. 

\black

\section{Short Multicategories}
\label{sec:fin-mult}

In this section, we will present a finite definition of certain multicategory-like structures, which we call short multicategories.  Later on, under further assumptions, we will show that they are equivalent to known types of multicategory.  We will take Proposition~\ref{prop:classical} as the grounds for our definition.


A \textbf{short multicategory} consists, to begin with, of a category $\catc$ together with:
\begin{itemize}
\item For $n \leq 4$ a functor $\mlc_n(-;-)\colon(\catc^{n})^{op} \times \catc \to \Set$ such that, when $n=1$, we have $\mlc_1(-;-)=\catc(-,-):\catc^{op} \times \catc \to \Set$.
\end{itemize}

\begin{rmk}
\label{rmk:no-need-for-unary-circ_i}
This says that for $n \leq 4$ we have sets $\mlc_n(x_1,\ldots,x_n;y)$ of $n$-ary multimaps (where the unary morphisms are those of $\catc$) and $n$-ary multimaps can be precomposed and postcomposed by unary ones in a compatible manner. We sometimes refer to these compatibilities as \emph{profunctoriality of $n$-ary multimaps.}
For these pre/post-composition we will use the following notation, for any unary maps $p\colon a'_i\to a_i$ and $q\colon b\to b'$, and $n$-ary map $f\colon\overline{a}\to b$, 
\begin{center}
	$q\circ f:=\mlc_n(\overline{a};q)(f)$ and $f\circ_i p:=\mlc_n(\overline{a}_{<i},p,\overline{a}_{>i};b)(f)$. 
\end{center}
Clearly, with these definitions, since functors preserve the identity, we get the following identity equations (for any $n$-ary map $f$ and any $i=1,\ldots,n$).
$$1\circ f=f=f\circ_i1$$
\black 

\end{rmk}

Furthermore, we require substitution functions 
$$\circ_i\colon \mlc_n(\overline{b};c) \times \mlc_m(\overline{a};b_i) \to \mlc_{n+m-1}(b_{<i},\overline{a},b_{>i};c)$$
for $i \in \{1,\ldots,n\}$ which are natural in each variable $a_1,\ldots, a_m,b_1,\ldots,b_{i-1},b_{i+1},\ldots,b_n,c,$ and dinatural in $b_i$ where:
\begin{itemize}
\item $n=2,3$, $m=2$ (substitution of binary into binary and ternary);
\item  $n=2$, $m=3$  (substitution of ternary into binary);
\item  $n=2,3$, $m=0$ (substitution of nullary into binary and ternary).
\end{itemize}

In the context of multimaps $f,g$ and $h$ of arity $2,n$ and $p$ respectively, one can consider associativity and commutativity equations, and identity axioms of the form: 
\begin{align}
f \circ_i(g\circ_j h) = (f\circ_i g)\circ_{j+i-1}h  \hspace{0.5cm} \textnormal{for} \hspace{0.5cm} 1\leq i\leq 2, 1 \leq j \leq n  \label{eq:ass-line}\\
(f\circ_1 g)\circ_{n+1}h=(f \circ_2 h) \circ_{1} g \hspace{0.5cm}  \hspace{1cm} \label{eq:ass-not-line} 
\end{align}
These are particular cases of the equations (\ref{ax:classic-ass-line},\ref{eq:classic-ass-not-line}) in Section~\ref{sec:clas-mult}. We require equations~(\ref{eq:ass-line},\ref{eq:ass-not-line}) \black in the following cases:\black 
\begin{itemize}
\item[(a)] $n=p=2$;
\item[(b)] $n = 2$, $p=0$;
\item[(c)] only for (\ref{eq:ass-not-line}), $n = 0$, $p=2$;
\item[(d)] only for (\ref{eq:ass-not-line}), $n=p = 0$.
\end{itemize}
Let us explain these equations in a more digestible form, using diagrams.
\begin{itemize}
\item (\ref{eq:ass-line}.a) corresponds to the four equations 
\begin{equation*}
\label{ax:fin-BBB}
(f\circ_i g)\circ_{i-j+1}h=f \circ_i(g\circ_j h)
\end{equation*}
where $1 \leq i,j \leq 2$ with $f,g$ and $h$ binary.  These amount to the fact that certain string diagrams are well-defined.  For instance, if we set $i=1$ and $j=2$, we get that the two possible interpretations of the following string diagram are the same.
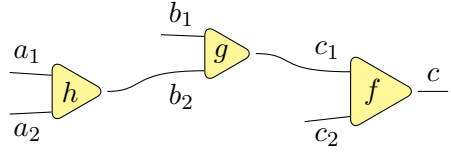
\begin{center}
\begin{tikzpicture}[triangle/.style = {fill=yellow!50, regular polygon, regular polygon sides=3,rounded corners}]
\path 
	(0,0.5) node [triangle,draw,shape border rotate=-90,inner sep=1pt,label=135:$a_1$,label=230:$a_2$] (b) {$h$} 
	(2,1) node [triangle,draw,shape border rotate=-90,inner sep=1pt,label=135:$b_1$,label=230:$b_2$] (a) {$g$}
	(4,0.5) node [triangle,draw,shape border rotate=-90,inner sep=1pt, label=135:$c_1$,label=230:$c_2$] (c) {$f$};

\draw [-] (-0.8,0.2) to (b.228);
\draw [-] (-0.8,0.75) to (b.139);
\draw [-] (3.1,0.1) to (c.228);
\draw [-] (a) .. controls +(right:1cm) and +(left:1cm).. (c.139);
\draw [-] (b) .. controls +(right:1cm) and +(left:1cm).. (a.225);
\draw [-] (c) to node [above] {$c$} (5,0.5);
\draw [-] (1.2,1.25) to (a.135);
\end{tikzpicture}  
\end{center}
\item (\ref{eq:ass-not-line}.a) corresponds to the equation 
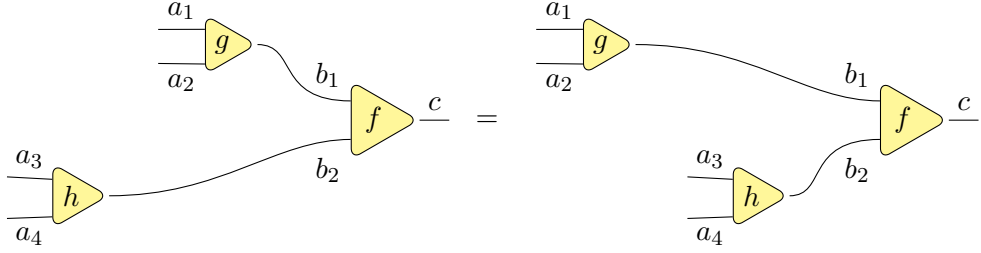
\begin{equation*}
\label{ax:fin-A}
\begin{gathered}
\begin{tikzpicture}[triangle/.style = {fill=yellow!50, regular polygon, regular polygon sides=3,rounded corners}]
\path (0,0) node [triangle,draw,shape border rotate=-90,inner sep=1.1pt] (a) {$h$}
	(2,2) node [triangle,draw,shape border rotate=-90,inner sep=1.1pt] (b) {$g$} 
	(4,1) node [triangle,draw,shape border rotate=-90,inner sep=1.1pt,label=135:$b_1$,label=230:$b_2$] (c) {$f$} 
	(9,0) node [triangle,draw,shape border rotate=-90,inner sep=1.1pt] (a') {$h$}
	(7,2) node [triangle,draw,shape border rotate=-90,inner sep=1.1pt] (b') {$g$} 
	(11,1) node [triangle,draw,shape border rotate=-90,inner sep=1.1pt, label=135:$b_1$,label=230:$b_2$] (c') {$f$};
	
\draw [-] (a) .. controls +(right:2cm) and +(left:1cm).. (c.220);
\draw [-] (b) .. controls +(right:1cm) and +(left:1cm).. (c.140);
\draw [-] (a') .. controls +(right:1cm) and +(left:1cm).. (c'.220);
\draw [-] (b') .. controls +(right:2cm) and +(left:1cm).. (c'.140);
	
\draw [-] (c) to node [above] {$c$} (5,1);
\draw [-] (1.15,2.2) to node [above] {$a_1$} (b.139);
\draw [-] (1.15,1.75) to node [below] {$a_2$} (b.228);
\draw [-] (-.85,0.25) to node [above] {$a_3$} (a.139);
\draw [-] (-.85,-0.3) to node [below] {$a_4$} (a.228);

\node () at (5.5,1) {$=$};

	\draw [-] (c') to node [above] {$c$} (12,1);
\draw [-] (6.15,2.2) to node [above] {$a_1$} (b'.139);
\draw [-] (6.15,1.75) to node [below] {$a_2$} (b'.228);
\draw [-] (8.15,0.25) to node [above] {$a_3$} (a'.139);
\draw [-] (8.15,-0.3) to node [below] {$a_4$} (a'.228);
\end{tikzpicture}
\end{gathered}
\end{equation*}
\item (\ref{eq:ass-line}.b) correspond to the four equations 
\begin{equation*}
\label{ax:fin-AU}
(f\circ_ig)\circ_{i-j+1}h=f\circ_i(g\circ_j h)
\end{equation*}
with $f,g$ binary, $h$ nullary and $1 \leq i,j \leq 2$.  For instance, if we set $i=2$ and $j=1$ it says that the following string diagram is well-defined:
\begin{center}
\begin{tikzpicture}[triangle/.style = {fill=yellow!50, regular polygon, regular polygon sides=3,rounded corners}]
\path 
	(0,0.5) node [triangle,draw,shape border rotate=-90,inner sep=1pt] (b) {$h$} 
	(2,0) node [triangle,draw,shape border rotate=-90,inner sep=1pt,label=135:$a_1$,label=230:$a_2$] (a) {$g$}
	(4,0.5) node [triangle,draw,shape border rotate=-90,inner sep=1pt,label=135:$b_1$,label=230:$b_2$] (c) {$f$};

\draw [-] (3.1,0.8) to (c.139);
\draw [-] (a) .. controls +(right:1cm) and +(left:1cm).. (c.220);
\draw [-] (b) .. controls +(right:1cm) and +(left:1cm).. (a.140);
\draw [-] (c) to node [above] {$c$} (5,0.5);
\draw [-] (1.2,-0.25) to (a.225);
\end{tikzpicture} .
\end{center}
\item (\ref{eq:ass-not-line}.b) is the equation
\begin{equation*}
\label{ax:fin-R}
\begin{gathered}
\begin{tikzpicture}[triangle/.style = {fill=yellow!50, regular polygon, regular polygon sides=3,rounded corners}]
\path (0,0.5) node [triangle,draw,shape border rotate=-90,inner sep=1pt] (a) {$h$}
	(2,2) node [triangle,draw,shape border rotate=-90,inner sep=1pt] (b) {$g$} 
	(4,1) node [triangle,draw,shape border rotate=-90, inner sep=1pt, label=135:$b_1$,label=230:$b_2$] (c) {$f$} 
	(9,0.5) node [triangle,draw,shape border rotate=-90,inner sep=1pt] (a') {$h$}
	(7,2) node [triangle,draw,shape border rotate=-90,inner sep=1pt] (b') {$g$} 
	(11,1) node [triangle,draw,shape border rotate=-90, inner sep=1pt, label=135:$b_1$,label=230:$b_2$] (c') {$f$};
	
\draw [-] (a) .. controls +(right:2cm) and +(left:1cm).. (c.220);
\draw [-] (b) .. controls +(right:1cm) and +(left:1cm).. (c.140);
\draw [-] (a') .. controls +(right:1cm) and +(left:1cm).. (c'.220);
\draw [-] (b') .. controls +(right:2cm) and +(left:1cm).. (c'.140);
	
\draw [-] (c) to node [above] {$c$} (5,1);
\draw [-] (1.15,2.2) to node [above] {$a_1$} (b.139);
\draw [-] (1.15,1.75) to node [below] {$a_2$} (b.228);

\node () at (5.5,1) {$=$};

	\draw [-] (c') to node [above] {$c$} (12,1);
\draw [-] (6.15,2.2) to node [above] {$a_1$} (b'.139);
\draw [-] (6.15,1.75) to node [below] {$a_2$} (b'.228);
\end{tikzpicture}
\end{gathered}
\end{equation*}

\item (\ref{eq:ass-not-line}.c) is the equation
\begin{equation*}
\label{ax:fin-L}
\begin{gathered}
\begin{tikzpicture}[triangle/.style = {fill=yellow!50, regular polygon, regular polygon sides=3,rounded corners}]
\path (0,0) node [triangle,draw,shape border rotate=-90,inner sep=1pt] (a) {$h$}
	(2,1.5) node [triangle,draw,shape border rotate=-90,inner sep=1pt] (b) {$g$} 
	(4,1) node [triangle,draw,shape border rotate=-90, inner sep=1pt, label=135:$b_1$,label=230:$b_2$] (c) {$f$} 
	(9,0) node [triangle,draw,shape border rotate=-90,inner sep=1pt] (a') {$h$}
	(7,1.5) node [triangle,draw,shape border rotate=-90,inner sep=1pt] (b') {$g$} 
	(11,1) node [triangle,draw,shape border rotate=-90,  inner sep=1pt, label=135:$b_1$,label=230:$b_2$] (c') {$f$};
	
\draw [-] (a) .. controls +(right:2cm) and +(left:1cm).. (c.220);
\draw [-] (b) .. controls +(right:1cm) and +(left:1cm).. (c.140);
\draw [-] (a') .. controls +(right:1cm) and +(left:1cm).. (c'.220);
\draw [-] (b') .. controls +(right:2cm) and +(left:1cm).. (c'.140);
	
\draw [-] (c) to node [above] {$c$} (5,1);
\draw [-] (-.85,0.25) to node [above] {$a_3$} (a.139);
\draw [-] (-.85,-0.3) to node [below] {$a_4$} (a.228);

\node () at (5.5,1) {$=$};

	\draw [-] (c') to node [above] {$c$} (12,1);
\draw [-] (8.15,0.25) to node [above] {$a_3$} (a'.139);
\draw [-] (8.15,-0.3) to node [below] {$a_4$} (a'.228);
\end{tikzpicture}
\end{gathered}
\end{equation*}

\item (\ref{eq:ass-not-line}.d) is the equation 
\begin{equation*}
\label{ax:fin-U}
\begin{gathered}
\begin{tikzpicture}[triangle/.style = {fill=yellow!50, regular polygon, regular polygon sides=3,rounded corners}]
\path (0,0.5) node [triangle,draw,shape border rotate=-90,inner sep=1pt] (a) {$h$}
	(2,1.5) node [triangle,draw,shape border rotate=-90,inner sep=1pt] (b) {$g$} 
	(4,1) node [triangle,draw,shape border rotate=-90,inner sep=1pt,label=135:$b_1$,label=230:$b_2$] (c) {$f$} 
	(9,0.5) node [triangle,draw,shape border rotate=-90,inner sep=1pt] (a') {$h$}
	(7,1.5) node [triangle,draw,shape border rotate=-90,inner sep=1pt] (b') {$g$} 
	(11,1) node [triangle,draw,shape border rotate=-90,inner sep=1pt,label=135:$b_1$,label=230:$b_2$] (c') {$f$};
	
\draw [-] (a) .. controls +(right:2cm) and +(left:1cm).. (c.220);
\draw [-] (b) .. controls +(right:1cm) and +(left:1cm).. (c.140);
\draw [-] (a') .. controls +(right:1cm) and +(left:1cm).. (c'.220);
\draw [-] (b') .. controls +(right:2cm) and +(left:1cm).. (c'.140);
	
\draw [-] (c) to node [above] {$c$} (5,1);

\node () at (5.5,1) {$=$};

	\draw [-] (c') to node [above] {$c$} (12,1);
\end{tikzpicture}
\end{gathered}
\end{equation*}

\end{itemize}

\begin{notation*}
From now on, when we will have a functor $F\colon\catc\to\catd$ and a list $\overline{a}$ of objects $a_1,\ldots,a_n$ in $\catc$, then we will write $F\overline{a}$ for the list of objects $Fa_1,\ldots,Fa_n$ in $\catd$. 
\end{notation*}

\begin{defn}
\label{def:map-fin-mult}
Let $\mlc$ and $\md$ two short multicategories. A \emph{morphism of short multicategories} is a functor $F\colon\catc\to\catd$ together with natural families
$$F_i\colon\mlc_i(\overline{a};b)\to\md_i(F\overline{a};Fb)$$
for any $0\leq i\leq 4$, with $i=1$ the functor action. These families must commute with all substitution operators $\circ_i$. 
\end{defn}

Short multicategories and their morphisms form a category $\fmulti$. Naturally, there is a forgetful functor $U\colon\multi\to\fmulti$ which takes a multicategory $\mlc$ and \emph{forgets} all the structure involving $n$-ary multimaps with $n\geq 4$. In particular, this functor forgets all substitutions which have as a result any $n$-ary multimaps with $n\geq 4$. For instance, it will not consider the substitution of ternary maps into ternary maps, since it gives out 5-ary multimaps.

\subsection{Representability for Short Multicategories}

We can define a $n$-ary map classifier for $\overline{a}$ also in $\fmulti$ as a representation of $\mlc_n(\overline{a};-)\colon\catc\to\Set$, i.e. a multimap 
$$\theta_{\overline{a}}\colon a_1,\ldots,a_n\to m\overline{a}$$
for which the induced function $-\circ_1\theta_{\overline{a}}\colon\mlc_1(m\overline{a};b)\to\mlc_n(\overline{a};b)$ is a bijection for all $b$.  

Then, a binary map classifier is said to be \emph{left universal} if, moreover, the induced function 
$$-\circ_1 \theta_{a,b}\colon\mlc_n(\,m(a,b),\overline{x};d) \to \mlc_{n+1}(a,b,\overline{x};d)$$
is a bijection for $n=2,3$ and $\overline{x}$ a tuple of the appropriate length.  Similarly a nullary map classifier $u \in \mlc_0(\diamond;i)$ is said to be left universal if, moreover, the function
 $$-\circ_1 u\colon\mlc_{1+n}(i,\overline{x};d) \to \mlc_n(\overline{x};d)$$
 is a bijection for $n=1,2$ and $\overline{x}$ a tuple of the appropriate length. We remark that here we consider only $n=1,2$ and not $n=3$ because in the definition of short multicategory we have only substitution of nullary into binary and ternary.

We will denote a binary multimap classifier as below left and the nullary map classifier as below right. 
From now on we might use the notation $m(a,b)=ab$. 
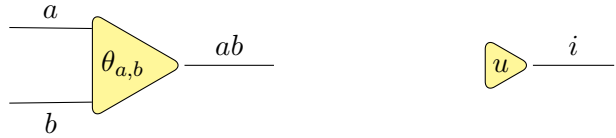
\begin{displaymath}
\begin{tikzpicture}[triangle/.style = {fill=yellow!50, regular polygon, regular polygon sides=3,rounded corners}]
\path (2,2) node [triangle,draw,shape border rotate=-90, inner sep=0pt] (m) {$\theta_{a,b}$};

\draw [-] (0.5,2.5) to node [above] {$a$} (m.130);
\draw [-] (0.5,1.5) to node [below] {$b$} (m.230);
\draw [-] (m) to node [above] {$ab$} (4,2);

\path (7,2) node [triangle,draw,shape border rotate=-90, inner sep=1pt] (m) {$u$};

\draw [-] (m) to node [above] {$i$} (8.5,2);
\end{tikzpicture}
%
\end{displaymath}

\begin{prop}\label{prop:universal}
Let $\mlc$ be a short multicategory with all binary map classifiers and nullary map classifiers. If, moreover, these classifiers are left universal, then the multimaps 

\begin{equation}\label{eq:univ}
\begin{gathered}
\begin{tikzpicture}[triangle/.style = {fill=yellow!50, regular polygon, regular polygon sides=3,rounded corners}]

\path 
	(7,1.35) node [triangle,draw,shape border rotate=-90,inner sep=0pt] (b') {$\theta_{a,b}$} 
	(9,1) node [triangle,draw,shape border rotate=-90,inner sep=-1.5pt,label=135:$ab$,label=230:$c$] (c') {$\theta_{ab, c}$};

\draw [-] (6.15,1.7) to node [above] {$a$} (b'.140);
\draw [-] (6.15,1) to node [below] {$b$} (b'.220);
\draw [-] (b') to (c'.140);
\draw [-] (7.9,0.65) to node [below] {$$} (c'.220);
\draw [-] (c') to node [above] {$(ab)c$} (11,1);
\end{tikzpicture}
\end{gathered}
\end{equation}

and

\begin{equation}
\label{eq:univ2}
\begin{gathered}
	\scalebox{1}{
\begin{tikzpicture}[triangle/.style = {fill=yellow!50, regular polygon, regular polygon sides=3,rounded corners}]

\path 
	(7,1.35) node [triangle,draw,shape border rotate=-90,inner sep=0pt] (b') {$\theta_{a,b}$} 
	(9,1) node [triangle,draw,shape border rotate=-90,inner sep=-1.5pt,label=135:$ab$,label=230:$c$] (c') {$\theta_{ab, c}$}
       	(11.7,0.65) node [triangle,draw,shape border rotate=-90,inner sep=-1.5pt,label=135:$(ab)c$,label=230:$d$] (d') {$\theta_{(ab)c, d}$};

\draw [-] (6.15,1.7) to node [above] {$a$} (b'.140);
\draw [-] (6.15,1) to node [below] {$b$} (b'.220);
\draw [-] (b') to (c'.140);
\draw [-] (7.9,0.65) to node [below] {$$} (c'.220);
\draw [-] (c') to node [above] {$$} (11,1);
\draw [-] (10.15,0.0) to node [below] {$$} (d'.223);
\draw [-] (d') to node [above] {$((ab)c)d$} (14.5,0.65);
\end{tikzpicture}}
\end{gathered}
\end{equation}

are $3$-ary and $4$-ary map classifiers and the following one a unary map classifier.

\begin{equation}\label{eq:univ2}
\begin{gathered}
\begin{tikzpicture}[triangle/.style = {fill=yellow!50, regular polygon, regular polygon sides=3,rounded corners}]
\path
	(2,1.5) node [triangle,draw,shape border rotate=-90,inner sep=1pt] (b) {$u$} 
	(4,1) node [triangle,draw,shape border rotate=-90,label=135:$i$,label=230:$a$] (ab) {$\theta$};

\draw [-] (3.2,0.6) to (ab.227);
\draw [-] (b) .. controls +(right:1cm) and +(left:1cm).. (ab.140);
	
\draw [-] (ab) to node [above] {$ia$} (5.5,1);

\end{tikzpicture}
\end{gathered}
\end{equation}

\end{prop}
\begin{proof}
Left universality implies that each component of the composite maps
\begin{equation*}\label{eq:local1}
\xymatrix{
\mlc_1((ab)c;d) \ar[rr]^{-\circ\theta_{ab,c}} & & \mlc_2(ab,c;d) \ar[rr]^{-\circ_1\theta_{a,b}} & & \mlc_3(a,b,c;d)}
\end{equation*}
\begin{equation*}\label{eq:local2}
\xymatrix{
\mlc_1(((ab)c)d;e) \ar[rr]^{-\circ\theta_{(ab)c,d}}  & & \mlc_2((ab)c,d;e) \ar[rr]^{-\circ_1\theta_{ab,c}} & & \mlc_3(ab,c,d;e) \ar[rr]^{-\circ_1\theta_{a,b}} & & \mlc_4(a,b,c,d;e) \\
}
\end{equation*}
and
\begin{equation*}\label{eq:local3}
\xymatrix{
\mlc_1(ia;b) \ar[rr]^{-\circ\theta_{i,a}} & & \mlc_2(i,a;b) \ar[rr]^{-\circ_1u} & & \mlc_1(a,b)}
\end{equation*}
is a bijection; it follows that the composites are bijections, which says exactly that the three claimed multimaps are universal.
\end{proof}

\begin{notation*}
	Let $\mlc$ be a short multicategory with all left universal binary map classifiers. 
	From now on, we might use the notation
	\begin{center}
		$\theta_{a,b,c}:=\theta_{ab,c}\circ_1\theta_{a,b}$, see \eqref{eq:univ}, and $\theta_{a,b,c,d}:=\theta_{(ab)c,d}\circ_1\theta_{ab,c}\circ_1\theta_{a,b}$, see \eqref{eq:univ2},
	\end{center} 
	for the 3-ary and 4-ary map classifiers found in the proposition above.  
\end{notation*}

Now, following the style of Definition~\ref{def:lr-rep-multi}, we will define the notion of representability for short multicategories. We will denote with $\mid\overline{x}\mid$ the length of a list $\overline{x}$. 

\begin{defn}\label{def:lr-rep-short.multi}
Let $\mlc$ be a short multicategory. 
\begin{itemize}
\item $\mlc$ is said to be \textbf{left representable} if it admits left universal nullary and binary map classifiers. 
\item $\mlc$ is said to be \textbf{representable} if it admits nullary and binary map classifiers such that the induced maps are bijections
\begin{center}
$-\circ_ju\colon\mlc_n(\overline{x},i,\overline{y};z)\to\mlc_{n-1}(\overline{x},\overline{y};z)$ \hspace{0.5cm} for $1\leq n\leq 3$\\ \vspace{0.1cm}
$-\circ_j\theta_{a,b}\colon\mlc_n(\overline{x},ab,\overline{y};z)\to\mlc_{n+1}(\overline{x},a,b,\overline{y};z)$ \hspace{0.5cm} for $1\leq n\leq 3$ 
\end{center}
where $0\leq\mid\overline{x}\mid,\mid\overline{y}\mid\leq n-1$ and $j=\mid\overline{x}\mid+1$. 
\end{itemize}
\end{defn}

We will denote by $\fleftrep$ and $\fmulti_{rep}$ the full subcategories of $\fmulti$ with objects left representable/representable short multicategories. Naturally, the forgetful functor $U\colon\multi\to\fmulti$ restricts to forgetful functors
\begin{center}
$U_{lr}\colon\multilr\to\fleftrep$ \\
$U_{rep}\colon\multirep\to\fmulti_{rep}$.
\end{center} 

\begin{notation*}
Let $\mlc$ be a short multicategory with a left universal binary classifier. Then we will use $(-)'\colon\mlc_n(\overline{a};b)\to\mlc_{n-1}(a_1a_2,a_3,\ldots,a_n;b)$ for the inverse of $-\circ_1\theta_{a_1,a_2}$: in other words, for any $n$-multimap $f$, $f'$ is the unique $(n-1)$-multimap such that $f'\circ_1\theta=f$.
\end{notation*}  


\begin{lemma}
\label{lemma:char-morph-left-repr}
Let $\mlc$ and $\md$ left representable short multicategories. A morphism $F\colon\mlc\to\md$ is uniquely specified by:
\begin{itemize}
\item A functor $F\colon\catc\to\catd$. 
\item Natural families $F_i\colon\mlc_i(\overline{a};b)\to\md_i(F\overline{a};Fb)$ for $i=0, 2$ commuting with the substitutions 
\begin{equation}
\label{eq:lemma-morph-nullary}
\begin{gathered}
\begin{tikzpicture}[triangle/.style = {fill=yellow!50, regular polygon, regular polygon sides=3,rounded corners}]
\path
	(2,1.5) node [triangle,draw,shape border rotate=-90,inner sep=1pt] (b) {$u$} 
	(4,1) node [triangle,draw,shape border rotate=-90,label=135:$a$,label=230:$b$] (ab) {$f$};

\draw [-] (3.2,0.55) to (ab.225);
\draw [-] (b) .. controls +(right:1cm) and +(left:1cm).. (ab.140);
	
\draw [-] (ab) to node [above] {$c$} (5.5,1);

\end{tikzpicture}
\hspace{2cm}
\begin{tikzpicture}[triangle/.style = {fill=yellow!50, regular polygon, regular polygon sides=3,rounded corners}]
\path
	(2,.4) node [triangle,draw,shape border rotate=-90,inner sep=1pt] (b) {$v$} 
	(4,1) node [triangle,draw,shape border rotate=-90,label=135:$a$,label=230:$b$] (ab) {$f$};

\draw [-] (3.2,1.4) to (ab.140);
\draw [-] (b) .. controls +(right:1cm) and +(left:1cm).. (ab.227);
	
\draw [-] (ab) to node [above] {$c$} (5.5,1);

\end{tikzpicture}
\end{gathered}
\end{equation}
and such that if we define, for any ternary map $h\in\mlc_3(\overline{a};b)$, $F_3h:=F_2h'\circ_1F_2\theta$, then $F_3$ also commutes with 
\begin{equation}
\label{eq:lemma-morph-bin}
\begin{gathered}
\begin{tikzpicture}[triangle/.style = {fill=yellow!50, regular polygon, regular polygon sides=3,rounded corners}]
\path 
	(2,0) node [triangle,draw,shape border rotate=-90,inner sep=0pt,label=135:$a$,label=230:$b$] (a) {$f$}
	(4,0.5) node [triangle,draw,shape border rotate=-90,label=135:$x$,label=230:$y$] (c) {$g$};

\draw [-] (3.1,0.8) to (c.139);
\draw [-] (a) .. controls +(right:1cm) and +(left:1cm).. (c.220);
\draw [-] (1.2,.25) to (a.140);
\draw [-] (c) to node [above] {$c$} (5,0.5);
\draw [-] (1.2,-0.25) to (a.225);
\end{tikzpicture} .
\end{gathered}
\end{equation}
\end{itemize}
\end{lemma}

\begin{proof}
\black The strategy to check that $F$ commutes with all substitutions is similar in every case, which is to use left representability and reduce each case to a lower dimensional one. We will show only the calculations for the case of the substitution of a binary map into the first variable of another binary map and redirect the interested reader to the proof of \cite[Lemma~4.2.5]{LobbiaThesis} for the remaining explicit calculations.  \black 
Let us now check that $F$ preserves substitution of binary maps into binarys maps.
Consider $g\colon x,c\to d$ and $f\colon a,b\to x$ binary maps. 
By left representability, we know that $f=f'\circ\theta_{a,b}$. 
\begin{align*}
&F_3(g\circ_1f) 
& &\\
&= F_2(g\circ_1f')\circ_1F_2(\theta_{a,b})
& (\textrm{definition of}\,F_3) 
&
\\ 
&= [F_2(g)\circ_1F_1(f')]\circ_1F_2(\theta_{a,b})
& (\textrm{naturality of }\,F_2) 
&
\\ 
&= F_2(g)\circ_1[F_1(f')\circ_1F_2(\theta_{a,b})]
& (\textrm{dinaturality of sub. of bin. into bin. in}\,\md) 
&
\\ 
&= F_2(g)\circ_1F_2(f'\circ\theta_{a,b})
& (\textrm{naturality of }\,F_2) 
&
\\ 
&= F_2(g)\circ_1F_2(f)
& (\textrm{definition of}\,f').
& \qedhere
\end{align*} 
\end{proof}

\subsection{Closedness for Short Multicategories}

We can adapt Definition~\ref{def:multi-closed} to short multicategories with the following.

\begin{defn}
\label{def:fin-closed}
A short multicategory is said to be \textbf{closed} if for all $b,c$ there exists an object $[b,c]$ and binary map $e_{b,c}\colon [b,c],b \to c$ for which the induced function 
$$e_{b,c}\circ_1-\colon\mlc_n(\overline{x};[b,c]) \to \mlc_{n+1}(\overline{x},b;c)$$
is a bijection, for $n=0,1,2,3$. 
\end{defn}

We will denote with $\fmulti^{cl}_{lr}$ the full subcategory of $\fmulti$ with objects left representable closed short multicategories. Naturally, the forgetful functor $U\colon\multi\to\fmulti$ restricts to a forgetful functor
\begin{center}
$U_{lr}^{cl}\colon\multilrcl\to\fleftrep^{cl}$.
\end{center} The next proposition gives a characterisation of closed short multicategories which are also left representable.  
 
\begin{prop}\label{prop:left-iff-adj}
A closed short multicategory is left representable if and only if it has nullary map classifier and each $[b,-]$ has a left adjoint. 
\end{prop}

\begin{proof}
If it is left representable and closed then the natural bijections
\begin{equation*}
\mlc_1(ab;c)\cong \mlc_2(a,b;c) \cong \mlc_1(a;[b,c])
\end{equation*}
show that $-b \dashv [b,-]$.  Conversely, if $[b,-]$ has left adjoint $-b$, then we have natural bijections 
\begin{equation*}
\catc(ab,c) \cong  \mlc_1(a;[b,c])  \cong \mlc_2(a,b;c)
\end{equation*}
and, by Yoneda, the composite is of the form $- \circ_1 \theta_{a,b}$ for a binary map classifier $\theta_{a,b}\colon$~$a,b \to$~$ab$.  It remains to show that this and the nullary map classifer are left universal. For the binary map classifier, we must show that
$- \circ \theta_{a,b}\colon\mlc_{n+1}(ab,\overline{x};c) \to \mlc_{n+2}(a,b, \overline{x};c)$ is a bijection for all $\overline{x}$ of length $1$ or $2$, the case $0$ being known.  For an inductive style argument, suppose it is true for $\overline{x}$ of length $i\leq 1$.  We should show that the bottom line below is a bijection
\begin{equation*}
\begin{tikzcd}[ampersand replacement=\&]
	{\mlc_{i+1}(ab,\overline{x};[y,c])} \&\& {\mlc_{i+2}(a,b,\overline{x};[y,c])} \\
	{\mlc_{i+2}(ab,\overline{x},y;c)} \&\& {\mlc_{i+3}(a,b,\overline{x},y;c])}
	\arrow["{- \circ_1 \theta_{a,b}}", from=1-1, to=1-3]
	\arrow["{e_{y,c} \circ_1 -}"', from=1-1, to=2-1]
	\arrow["{e_{y,c} \circ_1 -}", from=1-3, to=2-3]
	\arrow["{- \circ_1 \theta_{a,b}}"', from=2-1, to=2-3]
\end{tikzcd}
\end{equation*}
but this follows from the fact that the square commutes, by associativity axiom (\ref{eq:ass-line}), and the other three morphisms are bijections, by assumption. The case of the nullary map classifier is similar in form. 
\end{proof}  

\section{Skew Notions}
\label{sec:skew-notions}


Now that we have introduced short multicategories, we shall review some important skew notions. In particular, we will recall the definitions of \emph{skew monoidal category} \cite{Szlachanyi-skew} and \emph{skew multicategory} \cite{LackBourke:skew}. The first concept will be useful in Section~\ref{sec:short-vs-mult} when we will consider various correspondences between different flavours of monoidal category and multicategory. Then, in Section~\ref{sec:short-skew-mult} we will generalise the results of Section~\ref{sec:short-vs-mult} to skew multicategories. 

\subsection{Skew Monoidal Categories}
\label{sec:skew-mon}

We start reviewing the definition of skew monoidal categories and morphisms between them. A \textbf{(left) skew monoidal category} $(\catc,\otimes,i,\alpha,\lambda,\rho)$  \cite{Szlachanyi-skew} is a category $\catc$ together with a functor\footnote{For either $f=1_a$ or $g=1_b$ we will write $f\cdot 1_b=:f\cdot b$ and $1_a\cdot g=:a\cdot g$. }
\begin{align*}
\otimes\colon\catc \times \catc &\to \catc \\
(a,b) &\mapsto ab, \\
(f,g) &\mapsto fg=f\cdot g,
\end{align*}
a unit object $i \in \catc$, and natural families $\alpha_{a,b,c}\colon(ab) c \to a(bc)$, $\lambda_{a}\colon i a \to a$ and $\rho_{a}\colon a \to ai$ satisfying five axioms which are neatly labelled by the five words
\begin{center}
$abcd$\\
$iab$ $aib$ $abi$ \\
$ii$
\end{center}
of which the first refers to MacLane's pentagon axiom. More precisely, the axioms are the following 
\begin{equation}\label{ax:mon-pent}
\begin{tikzcd}
	&& {(ab)(cd)} \\
	{((ab)c)d} &&&& {a(b(cd))} \\
	& {(a(bc))d} && {a((bc)d)}
	\arrow["{\alpha_{a,b,c}d}"', from=2-1, to=3-2]
	\arrow["{\alpha_{a,bc,d}}"', from=3-2, to=3-4]
	\arrow["{a\alpha_{b,c,d}}"', from=3-4, to=2-5]
	\arrow["{\alpha_{ab,c,d}}", from=2-1, to=1-3]
	\arrow["{\alpha_{a,b,cd}}", from=1-3, to=2-5]
\end{tikzcd}
\end{equation}
\begin{minipage}{0.4\textwidth}
\begin{equation}
\begin{tikzcd}
	{(ia)b} & {i(ab)} \\
	& ab
	\arrow["{\alpha_{i,a,b}}", from=1-1, to=1-2]
	\arrow["{\lambda_{ab}}", from=1-2, to=2-2]
	\arrow["{\lambda_ab}"', from=1-1, to=2-2]
\end{tikzcd} \label{ax:mon-lambda}
\end{equation} 
    \end{minipage}%
    \begin{minipage}{0.2\textwidth}
    \hspace{0.2cm}
    \end{minipage}
    \begin{minipage}{0.4\textwidth}
\begin{equation}
\begin{tikzcd}
	ab & {(ab)i} \\
	& {a(bi)}
	\arrow["{\rho_{ab}}", from=1-1, to=1-2]
	\arrow["{\alpha_{a,b,i}}", from=1-2, to=2-2]
	\arrow["{a\rho_b}"', from=1-1, to=2-2]
\end{tikzcd} \label{ax:mon-rho}
\end{equation}
    \end{minipage}
\begin{minipage}{0.5\textwidth}
\begin{equation}
\label{ax:mon-ab}
\begin{tikzcd}
	ab & {(ai)b} & {a(ib)} \\
	&& ab
	\arrow["{\rho_ab}", from=1-1, to=1-2]
	\arrow["{\alpha_{a,i,b}}", from=1-2, to=1-3]
	\arrow["{a\lambda_b}", from=1-3, to=2-3]
	\arrow["{1_{ab}}"', from=1-1, to=2-3]
\end{tikzcd}
\end{equation} 
    \end{minipage}%
    \begin{minipage}{0.1\textwidth}
    \hspace{0.2cm}
    \end{minipage}
    \begin{minipage}{0.4\textwidth}
\begin{equation}
\label{ax:mon-ii}
\begin{tikzcd}
	i & ii \\
	& i. & {}
	\arrow["{\rho_i}", from=1-1, to=1-2]
	\arrow["{\lambda_i}", from=1-2, to=2-2]
	\arrow["{1_i}"', from=1-1, to=2-2]
\end{tikzcd}
\end{equation}
    \end{minipage}

\begin{defn}
Let $(\catc,\otimes,i,\alpha,\lambda,\rho)$ be a skew monoidal category. 
\begin{itemize}
\item $\catc$ is said to be \textbf{left normal} if $\lambda$ is invertible.

\item $\catc$ is said to be \textbf{(left) closed} if the endofunctor $-\otimes b$ has a right adjoint $[b,-]$ for any $b\in\catc$. We will sometimes refer to \emph{left closed} skew monoidal categories simply as \emph{closed} skew monoidal. 
\end{itemize}
\end{defn}

Let $(\catc,\otimes,i^C,\alpha^C,\lambda^C,\rho^C)$ and $(\catd,\otimes,i^D,\alpha^D,\lambda^D,\rho^D)$ be two skew monoidal categories. A \textbf{lax monoidal functor} $(F,f_0,f_2)$ \cite{Szlachanyi-skew} consists of a functor $F\colon\catc\to\catd$, a map
$$f_0\colon i^D\to Fi^C$$
and a family of maps
$$f_2\colon FaFb\to F(ab)$$
natural in $a$ and $b$ and satisfying the following axioms:
\begin{equation}
\label{eq:mon-moprh-alpha}
\begin{tikzcd}
	{(FaFb)Fc} && {F(ab)Fc} && {F(\,(ab)c\,)} \\
	{Fa(FbFc)} && {FaF(bc)} && {F(\,a(bc)\,)}
	\arrow["{f_2\cdot Fc}", from=1-1, to=1-3]
	\arrow["{f_2}", from=1-3, to=1-5]
	\arrow["\alpha^D"', from=1-1, to=2-1]
	\arrow["F\alpha^C", from=1-5, to=2-5]
	\arrow["{Fa\cdot f_2}"', from=2-1, to=2-3]
	\arrow["{f_2}"', from=2-3, to=2-5]
\end{tikzcd}
\end{equation}
\begin{minipage}{0.4\textwidth}
\begin{equation}
\label{eq:mon-morph-lambda}
\begin{tikzcd}
	iFa & Fa \\
	FiFa & {F(ia)}
	\arrow["\lambda^D", from=1-1, to=1-2]
	\arrow["{f_0Fa}"', from=1-1, to=2-1]
	\arrow["{f_2}"', from=2-1, to=2-2]
	\arrow["F\lambda^C"', from=2-2, to=1-2]
\end{tikzcd}
\end{equation} 
    \end{minipage}%
    \begin{minipage}{0.2\textwidth}
    \hspace{0.2cm}
    \end{minipage}
    \begin{minipage}{0.4\textwidth}
\begin{equation}
\label{eq:mon-morph-rho}
\begin{tikzcd}
	Fa && {F(ai)} \\
	{Fa.i} && FaFi.
	\arrow["F\rho^C", from=1-1, to=1-3]
	\arrow["\rho^D"', from=1-1, to=2-1]
	\arrow["{Fa\cdot f_0}"', from=2-1, to=2-3]
	\arrow["{f_2}"', from=2-3, to=1-3]
\end{tikzcd}
\end{equation}
    \end{minipage}

With an abuse of notation, we may write the associators, left/right unit maps as $\alpha,\lambda$ and $\rho$ both in $\catc$ and $\catd$, omitting the superscript. Skew monoidal categories and lax monoidal functors form a category $\skmon$\footnote{We notice that in \cite{LackBourke:skew} this category is denoted with $\skmon_l$ to underline that these are \emph{left} skew monoidal categories. In this paper, we drop the subscript $l$ since we only consider left skew monoidal categories and we rather use that space to specify other characteristics.}. We will denote with $\skmon_{ln}$, $\skmon^{cl}$ and $\skmon_{ln}^{cl}$ the full subcategories with objects left normal/closed/left normal and closed skew monoidal categories. 

\black There is a notion of \textbf{braiding} on a skew monoidal category \cite[Definition~2.2]{BouLack:skew-braid}, which specialises to braiding on monoidal categories, although looking slightly different. It consists of natural isomorphisms $s_{x,a,b}\colon (xa)b\to(xb)a$ satisfying the four axioms below. 
\begin{equation}\label{ax:sk-mon-brd-s}
\begin{tikzcd}[ampersand replacement=\&]
	\& {((xa)c)b} \& {((xc)a)b} \\
	{((xa)b)c} \&\&\& {((xc)b)a} \\
	\& {((xb)a)c} \& {((xb)c)a}
	\arrow["s", from=2-1, to=1-2]
	\arrow["{s\cdot b}", from=1-2, to=1-3]
	\arrow["s", from=1-3, to=2-4]
	\arrow["{s\cdot c}"', from=2-1, to=3-2]
	\arrow["s"', from=3-2, to=3-3]
	\arrow["{s\cdot a}"', from=3-3, to=2-4]
\end{tikzcd}
\end{equation}
\begin{equation}\label{ax:sk-mon-brd-s-a-1}
\begin{tikzcd}[ampersand replacement=\&]
	{((xa)b)c} \& {((xb)a)c} \& {((xb)c)a} \\
	{(xa)(bc)} \&\& {(x(bc))a}
	\arrow["{s\cdot c}", from=1-1, to=1-2]
	\arrow["s", from=1-2, to=1-3]
	\arrow["{\alpha\cdot a}", from=1-3, to=2-3]
	\arrow["\alpha"', from=1-1, to=2-1]
	\arrow["s"', from=2-1, to=2-3]
\end{tikzcd}
\end{equation}
\begin{equation}\label{ax:sk-mon-brd-s-a-2}
\begin{tikzcd}[ampersand replacement=\&]
	{((xa)b)c} \& {((xa)c)b} \& {((xc)a)b} \\
	{(x(ab))c} \&\& {(xc)(ab)}
	\arrow["s", from=1-1, to=1-2]
	\arrow["{s\cdot b}", from=1-2, to=1-3]
	\arrow["\alpha", from=1-3, to=2-3]
	\arrow["{\alpha\cdot c}"', from=1-1, to=2-1]
	\arrow["s"', from=2-1, to=2-3]
\end{tikzcd}
\end{equation}
\begin{equation}\label{ax:sk-mon-brd-a-s}
\begin{tikzcd}[ampersand replacement=\&]
	{((xa)b)c} \& {(x(ab))c} \& {x((ab)c)} \\
	{((xa)c)b} \& {(x(ac))b} \& {x((ac)b)}
	\arrow["{\alpha\cdot c}", from=1-1, to=1-2]
	\arrow["\alpha", from=1-2, to=1-3]
	\arrow["{x\cdot  s}", from=1-3, to=2-3]
	\arrow["s"', from=1-1, to=2-1]
	\arrow["{\alpha\cdot b}"', from=2-1, to=2-2]
	\arrow["\alpha"', from=2-2, to=2-3]
\end{tikzcd}
\end{equation}

We say that a braiding is a symmetry if, moreover, $s_{x,b,a}=(s_{x,a,b})^{-1}$. 
It can be shown that there is a category $\brdskmon$ with object braided skew monoidal categories and morphisms lax monoidal functors $(F,f_0,f_{a,b})\colon\cc\to\cd$ preserving the braiding, i.e. making the following diagram commutative. 
\begin{equation}\label{ax:brd-sk-mon-fct}
\begin{tikzcd}[ampersand replacement=\&]
	{(FxFa)Fb} \& {(FxFb)Fa} \\
	{F(xa)Fb} \& {F(xb)Fa} \\
	{F(\,(xa)b\,)} \& {F(\,(xb)a\,)}
	\arrow["{Fs^\cc}"', from=3-1, to=3-2]
	\arrow["{s^\cd}", from=1-1, to=1-2]
	\arrow["{f_{x,b}\cdot Fa}", from=1-2, to=2-2]
	\arrow["{f_{x,a}\cdot Fb}"', from=1-1, to=2-1]
	\arrow["{f_{xa,b}}"', from=2-1, to=3-1]
	\arrow["{f_{xb,a}}", from=2-2, to=3-2]
\end{tikzcd}
\end{equation}
Finally, we write $\symskmon$ and $\symmon$ for the subcategories of $\brdskmon$ and $\brdmon$ consisting of symmetric (skew) monoidal categories. 
\black 

\subsection{Skew Multicategories}
\label{sec:skew-mult}

In this section we will recall the definition of skew multicategory and some other important notions, all of which can be found in \cite{LackBourke:skew,BouLack:skew-braid}. 

\begin{defn}{\cite[Definition~4.2]{LackBourke:skew}}
A skew multicategory consists of
\begin{itemize}
\item a collection of objects $\catc_0$;
\item for each $a\in\catc_0$ a set $\mlc^l_0(\diamond;a)$ of \emph{nullary maps};
\item for each $n>0$, each $a_1,\ldots,a_n\in\catc_0$ and each $b\in\catc_0$ a set $\mlc_n^t(\overline{a};b)$ of \emph{tight $n$-ary maps} natural in all components and such that, when $n=1$, then $\mlc_1^t(a;b)=\catc(a,b)$;
\item for each $n>0$, each $a_1,\ldots,a_n\in\catc_0$ and each $b\in\catc_0$ a set $\mlc_n^l(\overline{a};b)$ of \emph{loose $n$-ary maps} natural in all components;
\item for each $n>0$, each $a_1,\ldots,a_n\in\catc_0$ and each $b\in\catc_0$ a function
$$j_{\overline{a},b}\colon\mlc_n^t(\overline{a};b)\to\mlc_n^l(\overline{a};b).$$
\end{itemize}
On top of this there is further structure:
\begin{itemize}
\item substitutions similar to a multicategory giving us multimaps $g(f_1,\ldots,f_n)$, which are tight just when $g$ and $f_1$ are. 
More precisely, for $x,y \in \{t,l\}$, we define 
\[
x \circ_i y := 
\begin{cases}
	t,& \text{if } x=y=t \text{ and } i=1\\
	t, & \text{if } x=t \text{ and }i \neq 1\\
	l, & \text{otherwise}
\end{cases}
\]
and require substitutions of the form
$$-\circ_i-\colon \mlc^x_n(\overline{b};c) \times \mlc^y_m(\overline{a};b_i) \longrightarrow \mlc^{x \circ_i y}_{n+m-1}(b_{<i},a,b_{>i};c),$$
which, moreover, commute with the comparisons viewing tight multimaps as loose. 
%
%
\end{itemize}
Finally the usual associativity and unit axioms must be satisfied. 

We call the category with objects $\catc_0$ and morphisms tight unary maps the underlying category of $\mlc$ and denote it with $\catc$.  
\end{defn}

\begin{rmk}
We can identify multicategories as skew multicategories in which all multimorphisms are tight (or loose), i.e. $j$ is the indentity.
\end{rmk}

Skew multicategories have a notion of morphism between them, which we call here \emph{skew multifunctor}. We recall that, given a functor $F\colon\catc\to\catd$, with $F\overline{a}$ we mean the list $Fa_1,\ldots,Fa_n$.

\begin{defn}\cite[Section~3~and~4]{LackBourke:skew}
Let $\mlc$ and $\md$ be skew multicategories.  A \textbf{skew multifunctor} is a functor $F\colon\catc\to\catd$ together with natural families
\begin{center}
$F^t_n\colon\mlc^t_n(\overline{a};b)\to\md^t_n(F\overline{a};Fb)$ \hspace{0.5cm} for \hspace{0.5cm} $1\geq n$ \\
$F^l_n\colon\mlc^l_n(\overline{a};b)\to\md^l_n(F\overline{a};Fb)$ \hspace{0.5cm} for \hspace{0.5cm} $0\geq n$
\end{center}
such that $F_1^t\equiv F$. These families must commute with all substitution operators and $j$. 
\end{defn}

Skew multicategories and skew multifunctors form a category $\smulti$. 

\subsubsection*{Left Representability and Closedness}

A skew multicategory $\mlc$ is \textbf{weakly representable} \cite[Section~4.4]{LackBourke:skew} if for each pair $x=t,l$ and $\overline{a}\in\catc^n$ there exists an object $m^x\overline{a}\in\catc$ and multimap
$$\theta^x_{\overline{a}}\in\mlc_n^x(\overline{a};m^x\overline{a})$$
with the property that the induced function 
$$-\circ_1\theta^x_{\overline{a}}\colon\mlc_1^t(m^x\overline{a};b)\to\mlc_n^x(\overline{a};b)$$
is a bijection for all $b\in\catc$. We call $\theta^t_{\overline{a}}$ a \textbf{tight $\mathbf{n}$-ary map classifier} and $\theta^l_{\overline{a}}$ a \textbf{loose $\mathbf{n}$-ary map classifier}. Moreover, we say that $\theta^x_{\overline{a}}$ is \textbf{left universal} if the induced function  
$$-\circ_1\theta^x_{\overline{a}}\colon\mlc_{1+r}^t(m^x\overline{a},\overline{x};b)\to\mlc_{n+r}^x(\overline{a},\overline{x};b)$$
is a bijection for each $r\geq 0$, $\overline{x}\in\catc^r$ and $b\in\catc$. 

\begin{defn}{\cite[Definition~4.5]{LackBourke:skew}}
A skew multicategory $\mlc$ is said to be \textbf{left representable} if it is weakly representable and all universal multimaps $\theta^x_{\overline{a}}$ are left universal. 
\end{defn}

We will denote with $\smulti_{lr}$ the full subcategory of $\smulti$ with objects left representable skew multicategories. 

\begin{rmk}
	One might wonder why we consider morphisms in $\smulti_{lr}$ to be only skew multifunctors and not require them to preserve the left representability structure, i.e. the left universal binary and nullary map classifiers, $\theta_{a,b}$ and $u$ respectively. 
	The idea is that, under the equivalence with skew monoidal categories described in \cite[Theorem~6.1]{LackBourke:skew},  skew multifunctors preserving $\theta_{a,b}$ and $u$ correspond to \emph{strong} monoidal functors and not \emph{lax}. 
	More precisely, a skew multifunctor preserving $\theta_{a,b}$ corresponds to a lax monoidal functor with $f_2$ invertible and, similarly, a skew multifunctor preserving $u$ corresponds to a lax monoidal functor with $f_0$ invertible. 
	This can be seen using, for instance, Lemma~\ref{lemma:sk-bij}. 
\end{rmk}
\black


\begin{defn}{\cite[Definition~4.7]{LackBourke:skew}}\label{def:closed-skew-multi}
A skew multicategory $\mlc$ is said to be \textbf{closed} if for all $b,c\in\catc$ there exists an object $[b,c]$ and tight multimap $e_{b,c}\in\mlc_2^t([b,c],b;c)$ with the universal property that the induced function
$$e_{b,c}\circ_1-\colon\mlc_n^x(\overline{a};[b,c])\to\mlc_{n+1}^x(\overline{a},b;c)$$
is a bijection for all $a_1,\ldots,a_n\in\catc$ and $x=t,l$. 
\end{defn}

We will denote with $\smulti^{cl}_{lr}$ the full subcategory of $\smulti$ with objects left representable closed skew multicategories. 

We conclude this subsection explaining briefly how to construct an equivalence 
$$T^{cl}_{lr}\colon\multilrcl\to\skmon_{ln}^{cl}$$
 between left representable closed multicategories and left normal skew monoidal closed categories. Even though this equivalence is not explicitly presented in \cite{LackBourke:skew}, it follows directly from some of their results. In particular let us recall three.

 \begin{rmk}
 	The various categories with objects skew monoidal categories or skew multicategories described in this section can actually be seen as 2-categories. 
 	Indeed, the theorems we are about to recall describe \emph{2-equivalences}. 
 	Since these 2-dimensional equivalences are strict, they also induce equivalences between the underlying categories. 
 	In this paper we will use only this part of these results, since our aim is to give different characterisations for (left representable/closed) skew multicategories. 
 	For this reason, we avoid giving precise definitions of the 2-cells in these 2-categories, and only briefly describe them in this remark. 
 	
 	The 2-cells in all of the 2-categories with objects skew monoidal categories are monoidal natural transformations, analogous to the ones for monoidal categories. 
 	On the other hand, a 2-cell between two skew multifunctors $F,G\colon\mlc\to\md$ will consists of a family, for any object $x\in\mc$, of unary maps $Fx\to Gx$ in $\md$ compatible with substitutions \cite[Section~3]{LackBourke:skew}. 
 \end{rmk}
 \black 

\begin{theorem}{\cite[Theorem~6.1]{LackBourke:skew}}
\label{thm:eq-skew-lr-multi}
There is a 2-equivalence between the 2-category $\skmon$ of skew monoidal categories and
the 2-category of left representable skew
multicategories.
\end{theorem}

From this theorem they then deduce the following result. 

\begin{theorem}{\cite[Theorem~6.3]{LackBourke:skew}}
\label{thm:eq-lnskew-lrmulti}
There is a 2-equivalence between the 2-categories of left normal skew monoidal categories and of left representable multicategories.
\end{theorem}

In the same way, one can prove that restricting the 2-equivalence given in the following result one get the wanted equivalence $T^{cl}_{lr}\colon\multilrcl\to\skmon_{ln}^{cl}$. 
\black

\begin{theorem}{\cite[Theorem~6.4]{LackBourke:skew}}
The 2-equivalence of Theorem~\ref{thm:eq-lnskew-lrmulti} restricts to a 2-equivalence between the 2-category $\skmon_{ln}^{cl}$ of closed skew monoidal categories and the 2-category $\multilrcl$ of left representable closed skew multicategories.
\end{theorem} 

\subsubsection*{Braidings and Symmetries}

\black For an ordinary multicategory $\mx$, a braiding on $\mx$ consists of an action of the braid groups $\brdgrp_n$ on $n$-ary multimaps compatible with substitution and identities. Similarly, a symmetry on $\mx$ involves an action of the symmetric group $\symgrp_n$ (again compatible with substitution and identities). 

Let us consider now a skew multicategory $\mlc$ and let $\brdgrp_n^1$ be the subgroup of $\brdgrp_n$ fixing the first variable. A \textbf{braiding} on $\mlc$ \cite[Section~5.3]{BouLack:skew-braid}
consists of, for any $r\in\brdgrp_n$ and $s\in\brdgrp^1_n$, actions
\begin{center}
$r^\ast\colon\mlc_n^l(a_1,\dots,a_n;b)\to\mlc_n^l(a_{r1},\dots,a_{rn};b)$ and \\
$s^\ast\colon\mlc_n^t(a_1,\dots,a_n;b)\to\mlc_n^t(a_1,a_{s2},\dots,a_{sn};b)$
\end{center}
compatible with substitution and $j_n$. A braiding is a symmetry if $r^\ast=s^\ast$ whenever $r$ and $s$ are sent to the same element of the symmetry group under the canonical map $\mid-\mid_n\colon\brdgrp_n\to\symgrp_n$. 

A skew multifunctor between braided skew multicategories is said to be \textbf{braided} if it respects all actions on $n$-ary multimaps, i.e. for any $r\in \brdgrp_n^x$ (where $\brdgrp_n^t=\brdgrp_n^1$ and $\brdgrp_n^l=\brdgrp_n$)
\[\begin{tikzcd}[ampersand replacement=\&]
	{\mlc_n^x(a_1,\cdots,a_n;c)} \& {\mlc_n^x(a_{r1},\cdots,a_{rn};c)} \\
	{\md_2^l(Fa_1,\cdots,Fa_n;Fc)} \& {\md_n^x(Fa_{r1},\cdots,Fa_{rn};Fc).}
	\arrow["{^\mlc r^\ast}", from=1-1, to=1-2]
	\arrow["{F_n^x}", from=1-2, to=2-2]
	\arrow["{F_n^x}"', from=1-1, to=2-1]
	\arrow["{^\md r^\ast}"', from=2-1, to=2-2]
\end{tikzcd}\]
We write $\brdsmulti$/$\symsmulti$ for the categories of braided/symmetric skew multicategories and braided skew multifunctors. Similarly, we will denote with $\brdsmulti_{lr}$ and $\symsmulti_{lr}$ the full subcategories of $\brdsmulti$ and $\symsmulti$ with objects left representable braided/symmetric skew multicategories. 
\black

\section{Short Multicategories vs Skew Monoidal Categories}
\label{sec:short-vs-mult}

In this section we will show that certain kinds of short multicategories are equivalent to certain kinds of multicategories. We will mostly consider kinds of representable multicategories because it will make proofs easier and many examples in the literature are of this kind. The strategy will be to use known equivalences between different flavours of monoidal category and multicategory \cite{Hermida2000Representable,LackBourke:skew}. For example, the left representable case gives us the following picture
\[\begin{tikzcd}
	{\multi_{lr}} \\
	& {\fmulti_{lr}} \\
	{\skmon_{ln}}.
	\arrow["T"', from=1-1, to=3-1]
	\arrow["U_{lr}", from=1-1, to=2-2]
	\arrow["K", dashed, from=2-2, to=3-1]
\end{tikzcd}\] 
We will start showing how to construct the functor $K$ and then prove it is an equivalence. The proof of the other cases will have the same structure. 

\subsection{The Left Representable and Representable Cases}
The first equivalence we will use is $T\colon\multilr\to\skmon_{ln}$ between left representable multicategories and left normal skew monoidal categories \cite[Theorem~6.3]{LackBourke:skew}.

\begin{notation*}
From now on, to increase readability of proofs, we will often mix algebraic parts and diagrams. For clarity, we shall explain what we mean with this. In a short multicategory any diagram has multiple interpretations, which are given by the order of the substitutions we apply. We presented some examples of this at the start of Section~\ref{sec:fin-mult}, when we explained how to interpret the associativity equations. For this reason, the formal proofs are always given by the algebraic expressions. However, the chain of equations can be quite long. We therefore add diagrams whenever the maps involved in the equations change and not only the bracketing. 
\end{notation*}

\begin{lemma}
\label{lemma:k-on-obj}
Given a left representable short multicategory $\mathbb C$ we can construct a left normal skew monoidal category $K \mathbb C$ in which:
\begin{itemize}
\item The tensor product $ab$ of two objects $a$ and $b$ is the binary map classifier;
\item The unit $i$ is the nullary map classifier;
\item Given $f\colon a \to b$ and $g\colon c \to d$ the tensor product $fg\colon ac \to bd$ is the unique morphism such that 
\begin{equation}
\label{eq:prod-maps}
\begin{gathered}
\begin{tikzpicture}[triangle/.style = {fill=yellow!50, regular polygon, regular polygon sides=3,rounded corners}]
\path (1,1) node [triangle,draw,shape border rotate=-90,label=135:$a$,label=230:$c$,inner sep=0pt] (c) {$\theta_{a,c}$} 
	(3,1) node [triangle,draw,shape border rotate=-90,inner sep=0pt] (ci) {$fg$}
	(9,0.5) node [triangle,draw,shape border rotate=-90,inner sep=1.5pt] (a') {$g$}
	(7,1.5) node [triangle,draw,shape border rotate=-90,inner sep=1pt] (b') {$f$} 
	(11,1) node [triangle,draw,shape border rotate=-90,label=135:$b$,label=230:$d$,inner sep=0pt] (c') {$\theta_{b,d}$};
	
\draw [-] (0,.6) to (c.220);
\draw [-] (0,1.35) to (c.140);
\draw [-] (c) to node [above] {$ac$} (ci);
\draw[-] (8,0.5) to node [below] {$c$} (a');
\draw[-] (6,1.5) to node [above] {$a$} (b');
\draw [-] (a') .. controls +(right:1cm) and +(left:1cm).. (c'.220);
\draw [-] (b') .. controls +(right:2cm) and +(left:1cm).. (c'.140);
	
\draw [-] (ci) to node [above] {$bd$} (4.25,1);

\node () at (5.25,1) {$=$};

	\draw [-] (c') to node [above] {$bd$} (12.5,1);
\end{tikzpicture}
\end{gathered}
\end{equation}
We will denote with $f\cdot c:=f1_c$ and similarly $a\cdot g:=1_ag$. 
\item The associator $\alpha\colon (ab)c\to a(bc)$ is defined as the unique map such that 
\begin{equation}
\label{eq:univ-alpha}
\begin{gathered}
\begin{tikzpicture}[triangle/.style = {fill=yellow!50, regular polygon, regular polygon sides=3,rounded corners}]
\path (0.25,0.5) node [triangle,draw,shape border rotate=-90,inner sep=0pt] (a) {$\theta_{b,c}$} 
	(2.5,1) node [triangle,draw,shape border rotate=-90,inner sep=-1.5pt,label=135:$a$,label=230:$bc$] (c) {$\theta_{a,bc}$} 
	(7,1.35) node [triangle,draw,shape border rotate=-90,inner sep=0pt] (b') {$\theta_{a,b}$} 
	(9,1) node [triangle,draw,shape border rotate=-90,inner sep=-1.5pt,label=135:$ab$,label=230:$c$] (c') {$\theta_{ab, c}$}
	(11.25,1) node [triangle,draw,shape border rotate=-90,inner sep=1pt] (f) {$\alpha$};
	
\draw [-] (a) to (c.230);
\draw [-] (1.25,1.45) to (c.135);
\draw [-] (7.9,0.65) to (c'.220);
\draw [-] (b') to (c'.140);
	
\draw [-] (c) to node [above] {$a(bc)$} (3.75,1);
\draw [-] (-.6,0.825) to node [above] {$b$} (a.140);
\draw [-] (-.6,0.175) to node [below] {$c$} (a.220);

\node () at (5,1) {$=$};

\draw [-] (c') to node [above] {$(ab)c$} (f);
\draw [-] (f) to node [above] {$a(bc)$} (13,1);
\draw [-] (6.15,1.7) to node [above] {$a$} (b'.140);
\draw [-] (6.15,1) to node [below] {$b$} (b'.220);
\end{tikzpicture}
\end{gathered}
\end{equation}

\item The left unit map $\lambda\colon ia\to a$ is defined as the unique map such that 
\begin{equation}
\label{eq:univ-lamda}
\begin{gathered}
\begin{tikzpicture}[triangle/.style = {fill=yellow!50, regular polygon, regular polygon sides=3,rounded corners}]
\path (2.5,1) node [triangle,draw,shape border rotate=-90] (c) {$1_a$} 
	(6.5,1.35) node [triangle,draw,shape border rotate=-90,inner sep=1pt] (b') {$u$} 
	(8,1) node [triangle,draw,shape border rotate=-90,label=135:$i$,label=230:$a$,inner sep=0pt] (c') {$\theta_{i,a}$}
	(10,1) node [triangle,draw,shape border rotate=-90,inner sep=0pt] (c'') {$\lambda_a$};
	
\draw [-] (c) to node [above] {$a$} (4,1);
\draw [-] (1.5,1) to node [above] {$a$} (c);

	\node () at (5,1) {$=$};

\draw [-] (7,0.65) to (c'.223);
\draw [-] (b') to (c'.137);
\draw [-] (c') to node [above] {$ia$} (c'');
\draw [-] (c'') to node [above] {$a$} (11,1);
\end{tikzpicture}
\end{gathered}
\end{equation}
(which is invertible by left representability).

\item The right unit map $\rho\colon a\to ai$ is defined as
\begin{equation}
\label{eq:univ-rho}
\begin{gathered}
\begin{tikzpicture}[triangle/.style = {fill=yellow!50, regular polygon, regular polygon sides=3,rounded corners}]
\path (0.5,0.55) node [triangle,draw,shape border rotate=-90,inner sep=1pt] (a) {$u$} 
	(2.5,1) node [triangle,draw,shape border rotate=-90,inner sep=0pt,label=135:$a$,label=230:$i$] (c) {$\theta_{a,i}$};

\draw [-] (a) to (c.230);
\draw [-] (1.2,1.4) to (c.135);
\draw [-] (c) to node [above] {$ai$} (4,1);

\end{tikzpicture}
\end{gathered}
\end{equation}

\end{itemize}
\end{lemma}
\begin{proof}  
Functoriality of $\catc^{2} \to \catc :(a,b) \mapsto ab$ follows from the universal property of the binary map classifier and profunctoriality of $\mlc_{2}(-;-)$.  It remains to verify the five axioms for a skew monoidal category.
\black All of them follow by checking the equalities using left representability. We will show explicitly how the axioms (\ref{eq:ass-line}.b) and (\ref{eq:ass-not-line}.c) prove the left unit axiom (\ref{ax:mon-lambda}) and refer the interested reader to \cite[Lemma~4.4.1]{LobbiaThesis} for the remaining axioms. We recall axiom (\ref{ax:mon-lambda}) below,\black
\[\begin{tikzcd}
	{(ia)b} & {i(ab)} \\
	& ab.
	\arrow["{\alpha_{i,a,b}}", from=1-1, to=1-2]
	\arrow["{\lambda_{ab}}", from=1-2, to=2-2]
	\arrow["{\lambda_ab}"', from=1-1, to=2-2]
\end{tikzcd}\] 
Using left representability it is enough to prove the equality precomposing with the universal nullary map $u$ and binary maps $\theta$. So, 
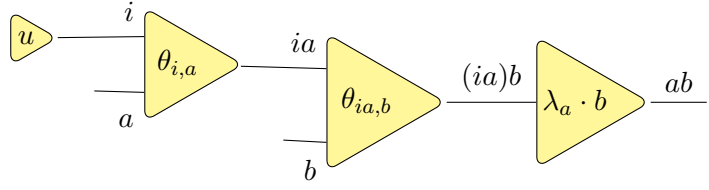
\begin{center}
\begin{tikzpicture}[triangle/.style = {fill=yellow!50, regular polygon, regular polygon sides=3,rounded corners}]
\path
	(1,1.35) node [triangle,draw,shape border rotate=-90,inner sep=1pt] (b') {$u$} 
	(3,1) node [triangle,draw,shape border rotate=-90,inner sep=1pt,label=135:$i$,label=230:$a$] (c') {$\theta_{i,a}$}
	(5.5,0.5) node [triangle,draw,shape border rotate=-90,inner sep=1pt,label=135:$ia$,label=230:$b$] (d) {$\theta_{ia, b}$}
	(8.25,0.5) node [triangle,draw,shape border rotate=-90,inner sep=-2pt] (f) {$\lambda_a\cdot b$};
	
\draw [-] (1.9,0.65) to (c'.220);
\draw [-] (b') to (c'.140);
\draw [-] (c') to (d.140);
\draw [-] (4.4,0) to (d.225);
\draw [-] (d) to node [above] {$(ia)b$} (f);
\draw [-] (f) to node [above] {$ab$} (10,0.5);
\end{tikzpicture}
\end{center}
\begin{align*}
&\lambda_a\cdot b\circ[\,(\,\theta_{ia,b}\circ_1\theta_{i,a}\,)\circ_1u\,]
&  
\\ 
&= [\,\lambda_a\cdot b\circ(\,\theta_{ia,b}\circ_1\theta_{i,a}\,)\,]\circ_1u
& (\textrm{nat. sub. null. into bin.}) 
\\ 
&= [\,(\,\lambda_a\cdot b\circ\theta_{ia,b}\,)\circ_1\theta_{i,a}\,]\circ_1u
& (\textrm{nat. sub. bin. into bin.}) 
\\ 
&= [\,(\,\theta_{a,b}\circ_1\lambda_a\,)\circ_1\theta_{i,a}\,]\circ_1u
& (\textrm{definition of}\,\lambda_a\cdot b) 
\end{align*}
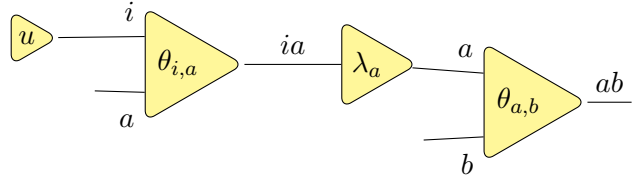
\begin{center}
\begin{tikzpicture}[triangle/.style = {fill=yellow!50, regular polygon, regular polygon sides=3,rounded corners}]
\path
	(1,1.35) node [triangle,draw,shape border rotate=-90,inner sep=1pt] (b') {$u$} 
	(3,1) node [triangle,draw,shape border rotate=-90,inner sep=1pt,label=135:$i$,label=230:$a$] (c') {$\theta_{i,a}$}
	(5.5,1) node [triangle,draw,shape border rotate=-90,inner sep=1pt] (f) {$\lambda_a$}
	(7.5,0.5) node [triangle,draw,shape border rotate=-90,inner sep=1pt,label=135:$a$,label=230:$b$] (d) {$\theta_{a, b}$};
	
\draw [-] (1.9,0.65) to (c'.220);
\draw [-] (b') to (c'.140);
\draw [-] (c') to node [above] {$ia$} (f);
\draw [-] (6.25,0) to (d.225);
\draw [-] (f) to (d.140);
\draw [-] (d) to node [above] {$ab$} (9,0.5);
\end{tikzpicture}
\end{center}
\begin{align*}
&= [\,\theta_{a,b}\circ_1(\,\lambda_a\circ_1\theta_{i,a}\,)\,]\circ_1u
& (\textrm{dinaturality sub. bin. into bin.}) 
\\ 
&= \theta_{a,b}\circ_1[\,(\,\lambda_a\circ_1\theta_{i,a}\,)\circ_1u\,]
& (\textrm{by axiom}\,(\ref{eq:ass-line}.b)) 
\\ 
&= \theta_{a,b}\circ_11_a
& (\textrm{by definition of}\,\lambda) 
\\
&=\theta_{a,b}
& (\textrm{by Remark~\ref{rmk:no-need-for-unary-circ_i}}) 
\end{align*}
On the other hand, 
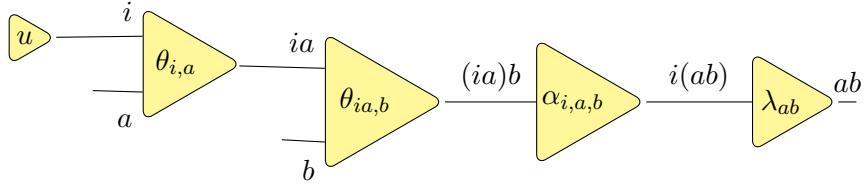
\begin{center}
\begin{tikzpicture}[triangle/.style = {fill=yellow!50, regular polygon, regular polygon sides=3,rounded corners}]
\path
	(1,1.35) node [triangle,draw,shape border rotate=-90,inner sep=1pt] (b') {$u$} 
	(3,1) node [triangle,draw,shape border rotate=-90,inner sep=1pt,label=135:$i$,label=230:$a$] (c') {$\theta_{i,a}$}
	(5.5,0.5) node [triangle,draw,shape border rotate=-90,inner sep=1pt,label=135:$ia$,label=230:$b$] (d) {$\theta_{ia, b}$}
	(8.25,0.5) node [triangle,draw,shape border rotate=-90,inner sep=-2pt] (f) {$\alpha_{i,a,b}$}
	(11,0.5) node [triangle,draw,shape border rotate=-90,inner sep=0pt] (f') {$\lambda_{ab}$};
	
\draw [-] (1.9,0.65) to (c'.220);
\draw [-] (b') to (c'.140);
\draw [-] (c') to (d.140);
\draw [-] (4.4,0) to (d.225);
\draw [-] (d) to node [above] {$(ia)b$} (f);
\draw [-] (f) to node [above] {$i(ab)$} (f');
\draw [-] (f') to node [above] {$ab$} (12,0.5);
\end{tikzpicture}
\end{center}
\begin{align*}
&[\;[\,(\lambda_{ab}\circ\alpha_{i,a,b})\circ\theta_{ia,b}\,]\circ_1\theta_{i,a}\;]\circ_1u
&  
\\
&=[\;[\,\lambda_{ab}\circ(\alpha_{i,a,b}\circ\theta_{ia,b})\,]\circ_1\theta_{i,a}\;]\circ_1u
& (\textrm{profunctoriality bin.}) 
\\
&=[\;\lambda_{ab}\circ[\,(\alpha_{i,a,b}\circ\theta_{ia,b})\circ_1\theta_{i,a}\,]\;]\circ_1u
& (\textrm{nat. sub. bin. into bin.}) 
\\
&=[\;\lambda_{ab}\circ(\,\theta_{i,ab}\circ_2\theta_{a,b}\,)\;]\circ_1u
& (\textrm{by definition of}\,\alpha) 
\end{align*}
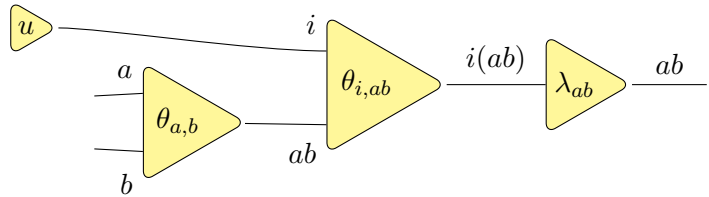
\begin{center}
\begin{tikzpicture}[triangle/.style = {fill=yellow!50, regular polygon, regular polygon sides=3,rounded corners}]
\path
	(1,1.25) node [triangle,draw,shape border rotate=-90,inner sep=1pt] (b') {$u$} 
	(3,0) node [triangle,draw,shape border rotate=-90,inner sep=1pt,label=135:$a$,label=230:$b$] (c') {$\theta_{a,b}$}
	(5.5,0.5) node [triangle,draw,shape border rotate=-90,inner sep=1pt,label=135:$i$,label=230:$ab$] (d) {$\theta_{i,ab}$}
	(8.25,0.5) node [triangle,draw,shape border rotate=-90,inner sep=0pt] (f') {$\lambda_{ab}$};
	
\draw [-] (1.9,-0.35) to (c'.220);
\draw [-] (1.9,0.35) to (c'.140);
\draw [-] (c') to (d.225);
\draw [-] (b') .. controls +(right:1cm) and +(left:1cm).. (d.140);
\draw [-] (d) to node [above] {$i(ab)$} (f');
\draw [-] (f') to node [above] {$ab$} (10,0.5);
\end{tikzpicture}
\end{center}
\begin{align*}
&=\lambda_{ab}\circ[\;(\,\theta_{i,ab}\circ_2\theta_{a,b}\,)\circ_1u\;]
&  (\textrm{nat. sub. null. into bin.})
\\
&=\lambda_{ab}\circ[\;(\,\theta_{i,ab}\circ_1u\,)\circ\theta_{a,b}\;]=
&  (\textrm{by axiom}\,(\ref{eq:ass-not-line}.c))
\end{align*}
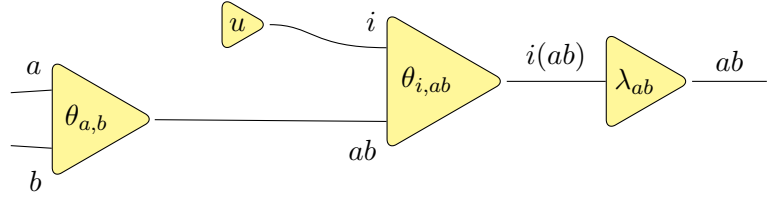
\begin{center}
\begin{tikzpicture}[triangle/.style = {fill=yellow!50, regular polygon, regular polygon sides=3,rounded corners}]
\path
	(3,1.25) node [triangle,draw,shape border rotate=-90,inner sep=1pt] (b') {$u$} 
	(1,0) node [triangle,draw,shape border rotate=-90,inner sep=1pt,label=135:$a$,label=230:$b$] (c') {$\theta_{a,b}$}
	(5.5,0.5) node [triangle,draw,shape border rotate=-90,inner sep=1pt,label=135:$i$,label=230:$ab$] (d) {$\theta_{i,ab}$}
	(8.25,0.5) node [triangle,draw,shape border rotate=-90,inner sep=0pt] (f') {$\lambda_{ab}$};
	
\draw [-] (0,-0.35) to (c'.220);
\draw [-] (0,0.35) to (c'.140);
\draw [-] (c') to (d.225);
\draw [-] (b') .. controls +(right:1cm) and +(left:1cm).. (d.140);
\draw [-] (d) to node [above] {$i(ab)$} (f');
\draw [-] (f') to node [above] {$ab$} (10,0.5);
\end{tikzpicture}
\end{center}
\begin{align*}
&=[\;\lambda_{ab}\circ(\,\theta_{i,ab}\circ_1u\,)\;]\circ\theta_{a,b}
&  (\textrm{profunctoriality bin.})
\\
&=1_{ab}\circ\theta_{a,b}
&  (\textrm{definition of}\,\lambda)
\\
&=\theta_{a,b}
& (\textrm{by Remark~\ref{rmk:no-need-for-unary-circ_i}})  
\end{align*}
\end{proof}

Before defining the functor $K\colon\fleftrep \to \skmon_{ln}$ on morphisms, we prove the following easy lemma.

\begin{lemma}\label{lem:bij}
Consider $\mlc, \md \in \fleftrep$ and a functor $F\colon \catc \to \catd$.  There is a bijection between natural families
$$F_{\overline{a},b}\colon\mlc_i(\overline{a};b)\to\md_i(F\overline{a};Fb)$$
and natural families\footnote{See Appendix~\ref{app:nat-of-f_a} for the explicit formulation of this naturality condition.}
$$f_{\overline{a}}\colon m(F\overline{a})\to F(m\overline{a})$$ 
where $m\overline{a}$ and $m(F\overline{a})$ are the $n$-ary map classifiers of the appropriate arity.
\end{lemma}
\begin{proof}
The bijection is governed by the following diagram
\begin{equation}\label{eq:yoneda}
\xymatrix{
{\mlc_i(\overline{a};-)}  \ar[rr]^{F_{\overline{a},-}} && {\md_i(F\overline{a};F-)}  \\
{\mlc_1(m\overline{a};-)}  \ar[u]_{\cong}^{- \circ_1 \theta_{\overline{a}}}  \ar[rr]_{F- \circ f_{\overline{a}}} && {\md_1(m(F\overline{a});F-)} \ar[u]^{\cong}_{- \circ_1 \theta_{F\overline{a}}}
}
\end{equation}
in which the vertical arrows are natural bijections and the lower horizontal arrow corresponds to the upper one using the Yoneda lemma. 
It is worth noticing that naturality in $\overline{a}$ follows by the fact that the classifiers $\theta$ are such. 
\end{proof}

\begin{rmk}
\label{rmk:def-K-on-lax-mon}
Given a morphism $F\colon \mlc \to \md \in \fleftrep$ we obtain, applying the above lemma, natural families $f_{2}\colon FaFb \to F(ab)$ and $f_0\colon i \to Fi$ defining the \emph{data} for a lax monoidal functor $KF\colon K\mlc \to K\md$. We will prove that this is a lax monoidal functor in Proposition~\ref{prop:full-faith}.

Explicitly, $f_{2}\colon FaFb \to F(ab)$ is the unique morphism such that $f_{2} \circ_{1} \theta_{Fa,Fb} = F_2(\theta_{a,b})$ whilst $f_{0}$ is the unique morphism such that $f_{0} \circ u = Fu$. 
\end{rmk}

\begin{notation*}
Let $\mlc$ be a short multicategory with a left universal nullary map classifier. Then we will use $(-)^\ast\colon\mlc_n(\overline{a};b)\to\mlc_{n+1}(i,\overline{a};b)$ for the inverse of $-\circ_1u$: in other words, for any $n$-multimap $f$, $f^\ast$ is the unique $(n+1)$-multimap such that $f^\ast\circ_1u=f$.  
\end{notation*}
\black
\begin{prop}\label{prop:full-faith}
With the definition on objects given in Lemma~\ref{lemma:k-on-obj} and on morphisms in Remark~\ref{rmk:def-K-on-lax-mon}, we obtain a fully faithful functor $K\colon\fleftrep \to \skmon_{ln}$.
\end{prop}
\begin{proof}
By Lemma~\ref{lemma:char-morph-left-repr} a morphism of $\fleftrep(\mlc,\md)$ is uniquely specified by a functor $F\colon C \to D$ and natural families $(F_2,F_0)$ satisfying three equations.  

By Lemma~\ref{lem:bij}, these natural families $(F_2,F_0)$ bijectively correspond to natural families $(f_2,f_0)$.  

Therefore, if we can prove that $(F_2,F_0)$ satisfy the equations of Lemma~\ref{lemma:char-morph-left-repr} if and only if $(f_2,f_0)$ satisfy the equations for a lax monoidal functor, then we will have described a bijection $K_{\mlc,\md}\colon\fleftrep(\mlc,\md) \to \skmon_{ln}(K\mlc,K\md)$.
Table~\ref{tab:axioms-lax-mon} describes the correspondence between axioms, \black we refer the interested reader to the proof of \cite[Proposition~4.4.4]{LobbiaThesis} for the details. \black 
\begin{table}[h!] 
\centering
\renewcommand\arraystretch{1.25}
\begin{tabular}{|c | c|}
\hline 
$(F_0,F_2)$ & $(f_0,f_2)$ \\
\hline
(\ref{eq:lemma-morph-bin}) & Associator axiom \\
(\ref{eq:lemma-morph-nullary}.a) & Left unit axiom \\
(\ref{eq:lemma-morph-nullary}.b) & Right unit axiom \\
\hline
\end{tabular}
\caption{}\label{tab:axioms-lax-mon}
\end{table}

Functoriality of $K$ follows routinely from the definition of $f_2$ and $f_0$.
\end{proof}

Let us recall that there is a forgetful functor $U_{lr}\colon\multi_{lr}\to\fmulti_{lr}$ and the authors of \cite{LackBourke:skew} construct an equivalence $T\colon\multi_{lr}\to\skmon_{ln}$. Moreover, comparing the construction of $K$ with that given in \cite[Section~6.2]{LackBourke:skew}, we see that the triangle below is commutative.
\[\begin{tikzcd}
	{\multi_{lr}} \\
	& {\fmulti_{lr}} \\
	{\skmon_{ln}}
	\arrow["T"', from=1-1, to=3-1]
	\arrow["U_{lr}", from=1-1, to=2-2]
	\arrow["K", from=2-2, to=3-1]
\end{tikzcd}\]

\begin{theorem}\label{thm:fin-equiv}
The functor $K\colon\fleftrep \to \skmon_{ln}$ is an equivalence of categories, as is the forgetful functor $U_{lr}\colon\multi_{lr}\to\fmulti_{lr}$.
\end{theorem}

\begin{proof}
Let us show that $K$ is an equivalence first.  Since $K$ is fully faithful by the preceding result, it remains to show that it is essentially surjective on objects.  Since $T=KU$ and the equivalence $T$ is essentially surjective, so is $K$, as required.  Finally, since $T=KU$ and both $T$ and $K$ are equivalences, so is $U_{lr}$.
\end{proof}

Then, if we consider the forgetful functor $U_{rep}\colon\multi_{rep}\to\fmulti_{rep}$ and the equivalence $T_{rep}\colon\multi_{rep}\to\mon$ given in \cite{Hermida2000Representable}, we get the following result.  

\begin{theorem}\label{thm:rep-mult}
The equivalence $K\colon\fleftrep \to \skmon_{ln}$ of Theorem~\ref{thm:fin-equiv} restricts to an equivalence $K_{rep}\colon\fmulti_{rep} \to \mon$ between representable short multicategories and monoidal categories, which fits in the commutative triangle of equivalences below.
 \[\begin{tikzcd}
	{\multi_{rep}} \\
	& {\fmulti_{rep}} \\
	{\mon}
	\arrow["T_{rep}"', from=1-1, to=3-1]
	\arrow["U_{rep}", from=1-1, to=2-2]
	\arrow["K_{rep}", from=2-2, to=3-1]
\end{tikzcd}\]
\end{theorem}

\begin{proof}
Let $\mlc\in\fmulti_{lr}$. If $\mlc$ is representable then $K\mlc$ has invertible left unit $\lambda$ since it is skew left normal. Therefore, we have left to prove that $\alpha$ and $\rho$ are isomorphisms as well. First, we can define the inverse of $\alpha$ through the chain of bijections
\begin{center}
$\mlc_3(a,b,c;(ab)c)\cong\mlc_2(a,bc;(ab)c)\cong\mlc_1(a(bc);(ab)c)$ \\ \vspace{0.1cm}
$\theta_{ab,c}\circ_1\theta_{a,b}$ \hspace{1.5cm} $\longmapsto$ \hspace{1.5cm} $\alpha\inv$ \hspace{0.5cm}
\end{center}
Using the universal properties of $(ab)c$ and $a(bc)$ we can show that $\alpha$ and $\alpha\inv$ are inverses of each other. Then, $\rho$ is defined as $\theta\circ_2u$, see (\ref{eq:univ-rho}). We define $\rho\inv$ as the map corresponding to $1_a$ through the following bijection
$$\mlc_1(ai;a)\cong\mlc_2(a,i;a)\cong\mlc_1(a;a),$$
i.e. $\rho\inv$ is the unique map such that 
\begin{equation}
\label{eq:rho-inv}
\begin{gathered}
\begin{tikzpicture}[triangle/.style = {fill=yellow!50, regular polygon, regular polygon sides=3,rounded corners}]
\path (2.5,1) node [triangle,draw,shape border rotate=-90] (c) {$1_a$} 
	(6.5,0.65) node [triangle,draw,shape border rotate=-90,inner sep=1pt] (b') {$u$} 
	(8,1) node [triangle,draw,shape border rotate=-90,label=135:$a$,label=230:$i$,inner sep=0pt] (c') {$\theta_{a,i}$}
	(10,1) node [triangle,draw,shape border rotate=-90,inner sep=0pt] (c'') {$\rho\inv$};
	
\draw [-] (c) to node [above] {$a$} (4,1);
\draw [-] (1.5,1) to node [above] {$a$} (c);

	\node () at (5,1) {$=$};

\draw [-] (7,1.35) to (c'.137);
\draw [-] (b') to (c'.223);
\draw [-] (c') to node [above] {$ai$} (c'');
\draw [-] (c'') to node [above] {$a$} (11,1);
\end{tikzpicture}
\end{gathered}
\end{equation}
which means that $\rho\inv\rho=1_a$. Using the universal property of $ai$ we can also prove that $\rho\rho\inv=1_{ai}$. 

On the other side, if $K\mlc$ is monoidal, then $\alpha$ and $\rho$ are invertible. Then 
$$\mlc_3(x,ab,y;z)\cong\mlc_2(x(ab),y;z)\cong\mlc_2((xa)b,y;z)\cong\mlc_3(xa,b,y)\cong\mlc_4(x,a,b,y;z)$$
where the second isomorphism is given by pre-composition with $\alpha_{x,a,b}$ and the rest by left representability. We can see how this isomorphism sends a map $f\colon x,ab,y\to z$ to
\begin{center}
$[\,(f'\circ_1\alpha)\circ_1\theta_{xa,b}\,]\circ_1\theta_{x,a}=$ \\
\begin{tikzpicture}[triangle/.style = {fill=yellow!50, regular polygon, regular polygon sides=3,rounded corners}]
\path
	(1,1.35) node [triangle,draw,shape border rotate=-90,inner sep=1.5pt,label=135:$x$,label=230:$a$] (b') {$\theta$} 
	(3,1) node [triangle,draw,shape border rotate=-90,inner sep=1.5pt,label=135:$xa$,label=230:$b$] (c') {$\theta$}
	(5.5,1) node [triangle,draw,shape border rotate=-90,inner sep=1pt] (f) {$\alpha$}
	(7.5,0.5) node [triangle,draw,shape border rotate=-90,inner sep=2pt,label=135:$x(ab)$,label=230:$y$] (d) {$f'$};
	
\draw [-] (0.25,1.6) to (b'.138);
\draw [-] (0.25,1) to (b'.228);
\draw [-] (2.25,0.7) to (c'.220);
\draw [-] (b') to (c'.140);
\draw [-] (c') to node [above] {$(xa)b$} (f);
\draw [-] (6.5,0.1) to (d.225);
\draw [-] (f) .. controls +(right:1cm) and +(left:1cm).. (d.140);
\draw [-] (d) to node [above] {$z$} (9,0.5);
\end{tikzpicture}

\end{center}
which can be proven to be equal to $f\circ_2\theta_{a,b}$ using the definition of $\alpha$ and associativity equations in $\mlc$.
\begin{align*}
&[\,(f'\circ_1\alpha)\circ_1\theta_{xa,b}\,]\circ_1\theta_{x,a}
& \\
&= [\,f'\circ_1(\alpha\circ\theta_{xa,b})\,]\circ_1\theta_{x,a}
& (\textrm{by dinaturality sub. binary into binary}) 
\\ 
&= f'\circ_1[\,(\alpha\circ\theta_{xa,b})\circ_1\theta_{x,a}\,]
& (\textrm{by axiom (\ref{eq:ass-line}.a)})
\\
&= f'\circ_1(\,\theta_{x,ab}\circ_2\theta_{a,b}\,)
& (\textrm{by definition of}\,\alpha)
\\
&= (\,f'\circ_1\theta_{x,ab}\,)\circ_2\theta_{a,b}
& (\textrm{by axiom (\ref{eq:ass-line}.a)})
\\
&= f\circ_2\theta_{a,b}
& (\textrm{by definition of}\,f').
\end{align*}
Then, the isomorphisms
\begin{center}
$\mlc_2(x,ab;z)\cong\mlc_3(x,a,b;z)$ \hspace{0.5cm} and \hspace{0.5cm} $\mlc_3(x,y,ab;z)\cong\mlc_4(x,y,a,b;z)$ 
\end{center} 
are constructed and shown to be induced by pre-composition with $\theta_{a,b}$ in a similar way. Finally, we show how $u$ induces the required isomorphisms for a representable short multicategory. By left representability, the map
$$-\circ u\colon\mlc_n(i,\overline{a};z)\to\mlc_{n-1}(\overline{a};z)$$
is an isomorphism (for $\overline{a}$ of length $n-1$). Then, since $\rho$ is invertible, we can define the following isomorphism
$$\mlc_3(a,i,b;z)\cong\mlc_2(ai,b;z)\cong\mlc_2(a,b;z)$$
where the last map is given by pre-composition with $\rho$ in the first variable and the first one by left representability. Thus, a ternary map $k\colon a,i,b\to z$ is sent to $k'\circ_1\rho$. The calculations below show that this is the same as precomposing with $u$ in the second variable. 
\begin{align*}
&k'\circ_1\rho_a 
& \\
&= k'\circ_1(\theta_{a,i}\circ_2u)
& (\textrm{by definition of}\,\rho)
\\ 
&= (k'\circ_1\theta_{a,i})\circ_2u
& (\textrm{by axiom (\ref{eq:ass-line}.b)})
\\ 
&= k\circ_2u
& (\textrm{by definition of}\,k').
\end{align*}
Similarly, we can construct the isomorphism $\mlc_2(a,i;z)\cong\mlc_1(a;z)$ and prove that it is induced by pre-composition with $u$.  

Since $\fmulti_{rep}$ and $\mon$ are full subcategories of $\fmulti_{lr}$ and $\skmon_{ln}$, the fully faithfulness of $K_{rep}$ follows from the one of $K$. Hence, $U_{rep}$ is an equivalence as well. 
\end{proof}

\subsection{The Closed Left Representable Case}

In this section we will consider the equivalence $T^{cl}_{lr}\colon\multilrcl\to\skmon_{ln}^{cl}$ between left representable closed multicategories and left normal skew monoidal closed categories.
The existence of this equivalence follows from \cite[Theorem~6.4]{LackBourke:skew} in the same way as \cite[Theorem~6.3]{LackBourke:skew} follows from \cite[Theorem~6.1]{LackBourke:skew}.

%
%

\begin{theorem}\label{thm:closed-lr-multi}
The equivalence $K\colon\fleftrep \to \skmon_{ln}$ restricts to an equivalence $K^{cl}_{lr}\colon\fleftrep^{cl} \to \skmon_{ln}^{cl}$ between left representable closed short multicategories and left normal skew closed monoidal categories, which fits in the commutative triangle of equivalences
 \[\begin{tikzcd}
	{\multi_{lr}^{cl}} \\
	& {\fmulti^{cl}_{lr}} \\
	{\skmon^{cl}_{ln}.}
	\arrow["T_{lr}^{cl}"', from=1-1, to=3-1]
	\arrow["U_{lr}^{cl}", from=1-1, to=2-2]
	\arrow["K_{lr}^{cl}", from=2-2, to=3-1]
\end{tikzcd}\]   
\end{theorem}

\begin{proof}
Let $\mlc$ be in $\fleftrep$. If $\mlc$ is closed then we have natural isomorphisms
\begin{equation*}
\catc(ab,c) =  \mlc_1(ab;c)\cong \mlc_2(a,b;c) \cong \mlc_1(a,[b,c]) = \catc(a,[b,c])
\end{equation*}
so that $K\mlc$ is monoidal skew closed, as required.  If $K\mlc$ is closed, then we have natural isomorphisms $\mlc_1(a;[b,c]) \cong \mlc_1(ab;c)$ for all $a,b,c$. By Yoneda, the composite
\begin{equation*}
\mlc_1(a;[b,c]) \cong  \mlc_1(ab;c)\cong \mlc_2(a,b;c)
\end{equation*}
 is of the form $e_{b,c} \circ_1 -$ for a binary map $e_{b,c}\colon [b,c],b \to c$, and to show that $\mlc$ is closed we must prove that the function on the bottom row below is a bijection for tuples $\overline{a}$ of length $0$ to $3$.  
 \begin{equation*}
 \xymatrix{
 \mlc_1(m(\overline{a});[b,c]) \ar[d]_{- \circ_1 \theta_{\overline{a}}} \ar[rr]^{e_{b,c} \circ{_1} -} && \mlc_2(m(\overline{a}),b;c) \ar[d]^{- \circ_1 \theta_{\overline{a}}} \\
\mlc_n(\overline{a};[b,c])  \ar[rr]_{e_{b,c} \circ{_1} -} && \mlc_{n+1}(\overline{a},b;c)} 
\end{equation*}
Since $\mlc$ underlies a left representable multicategory, there exists a left universal multimap $\theta_{\overline{a}}\colon\overline{a} \to m(\overline{a})$ and we have a commutative diagram as above in which the upper horizontal is invertible, as already established, and the two vertical functions by left universality, so that the lower horizontal is a bijection too.
%
%
%
%
\end{proof}

Putting together Theorem~\ref{thm:rep-mult} and Theorem~\ref{thm:closed-lr-multi} we get the following result. 

\begin{theorem}\label{thm:cl-rep-equiv}
The equivalence $K\colon\fleftrep \to \skmon_l$, defined in Theorem~\ref{thm:fin-equiv}, restricts to an equivalence $K_{rep}^{cl}\colon\fmulti_{rep}^{cl} \to \mon^{cl}$ between short representable closed multicategories and closed monoidal categories, which fits in the commutative triangle of equivalences
 \[\begin{tikzcd}
	{\multi_{rep}^{cl}} \\
	& {\fmulti^{cl}_{rep}} \\
	{\mon^{cl}.}
	\arrow["T_{rep}^{cl}"', from=1-1, to=3-1]
	\arrow["U_{rep}^{cl}", from=1-1, to=2-2]
	\arrow["K_{rep}^{cl}", from=2-2, to=3-1]
\end{tikzcd}\]   
\end{theorem}

\section{Short Skew Multicategories}
\label{sec:short-skew-mult}

In this section we will adapt the definitions in Section~\ref{sec:fin-mult} and the results in Section~\ref{sec:short-vs-mult} to the skew setting. Once again, we will follow the style of Proposition~\ref{prop:classical}.

A \textbf{short skew multicategory} consists, to begin with, of a category $\catc$ together with:
\begin{itemize}
\item For $1 \leq n \leq 4$ a functor $\mlc^{t}_n(-;-):(\catc^{n})^{op} \times \catc \to \Set$ such that, when $n=1$, we have $\mlc^t_1(-;-)=\catc(-,-)\colon\catc^{op} \times \catc \to \Set$.
\item For $n=0,1,2$ an additional functor $\mlc^{l}_n(-;-)\colon(\catc^{n})^{op} \times \catc \to \Set$ and natural transformation $j_n\colon\mlc^{t}_n(-;-) \to \mlc^{l}_n(-;-)$. 
\end{itemize}


\begin{rmk}
	\label{rmk:skw-no-need-for-unary-circ_i}
	The $l$-typed multimaps, i.e. the elements of $\mlc_n^l(\overline{a};b)$, are thought of as \emph{loose}, whereas the $t$-typed multimaps, i.e. the elements of $\mlc_n^t(\overline{a};b)$, are thought of as \emph{tight}. The function $j$ lets us view tights as loose.  Nullary maps are thought of as loose.  
	The tight unary maps are precisely the morphisms of $\catc$. 
	Similarly to the situation described in Remark~\ref{rmk:no-need-for-unary-circ_i}, both tight and loose $n$-ary multimaps $f$ admit compatible precomposition (in any position) and postcomposition by tight unary maps $p$ --- that is, by the morphisms of $\catc$ --- which we write as $f \circ_i p$ and $p \circ f$ respectively.
	Clearly, with these definitions, since functors preserve the identity, we get the following identity equations (for any $n$-ary loose or tight multimap $f$ and any $i=1,\ldots,n$).
	$$1\circ_1f=f=f\circ_i1$$
\end{rmk}

For $x,y \in \{t,l\}$, let us write
\[
    x \circ_i y = 
\begin{cases}
    t,& \text{if } x=y=t \text{ and } i=1\\
    t, & \text{if } x=t \text{ and }i \neq 1\\
    l, & \text{otherwise}
\end{cases}
\]

Then, in the cases listed below, we require functions
$$-\circ_i-\colon \mlc^x_n(\overline{b};c) \times \mlc^y_m(\overline{a};b_i) \longrightarrow \mlc^{x \circ_i y}_{n+m-1}(b_{<i},a,b_{>i};c)$$
for $i \in \{1,\ldots,n\}$ which are natural in $a_1,\ldots, a_m,b_1,\ldots,b_{i-1},b_{i+1},\ldots,b_n,c$ and dinatural in $b_i$.
Analogous to those from before, we require $\circ_i$ for the following cases:\footnote{The idea is that the rest of substitutions needed to form a skew multicategory, under left representability/closedness, will be derivable from these ones.}
\begin{itemize}
\item $m=2$, $n=2,3$ and $x=y=t$ (substitution of tight binary into tight binary/ternary);
\item $m=3$, $n=2$ and $x=y=t$ (substitution of tight ternary into tight binary);
\item $m=0$, $n=2,3$ and $x=t$ (substitution of nullary into tight binary/ternary).
\end{itemize}
but also:
\begin{itemize}
\item $m=0$, $n=1$ and $x=y=l$ (substitution of nullary into loose unary);
\item $m=1$, $y=l$, $n=2$ and $x=t$ (substitution of loose unary into tight binary); 
\item $m=2$, $y=t$, $n=1$ and $x=l$ (substitution of tight binary into loose unary). 
\end{itemize}

In the context of a binary multimap $f$, and multimaps $g$ and $h$ of arity $n$ and $p$ respectively (all tight except nullary ones), one can consider associativity equations of the following form:\footnote{We remark that, thanks to the equality $\mlc_1(-;-)=\catc(-,-)$, we can see identity maps $1$ in $\catc$ as tight unary maps.}
\begin{align}
f \circ_i(g\circ_j h) = (f\circ_i g)\circ h_{j+i-1}  \hspace{0.5cm} \textnormal{for} \hspace{0.5cm} 1 \leq i \leq 2, 1 \leq j \leq n , \label{eq:sk-ass-line}\\
(f\circ_1 g)\circ_{n+1}h=(f \circ_2 h) \circ_{1} g 
\hspace{0.5cm}  \hspace{1cm}
\label{eq:sk-ass-not-line} 
\black 
\end{align}

We require
equations (\ref{eq:sk-ass-line},\ref{eq:sk-ass-not-line}) in the following cases:
\begin{itemize}
\item[(a)] $n=p=2$;
\item[(b)] $n = 2$, $p=0$;
\item[(c)] only for (\ref{eq:sk-ass-not-line}), $n = 0$, $p=2$; 
\item[(d)] only for (\ref{eq:sk-ass-not-line}), $n=p = 0$. 
\end{itemize}

\begin{rmk}\label{rmk:jg-jp}
Let us unfold what naturality of $j_n$ means. Let $g$ be a tight binary map, $p$ and $q$ two tight unary maps and $v$ a nullary map. Naturality of $j_n$ means that the following equations hold:
\begin{equation}
\label{ax:nat-in-j}
\begin{split}
& g\circ_2jp=g\circ_2p \hspace{2cm} g\circ_1jp=j(g\circ_1p) \hspace{2cm} q\circ jp=j(q\circ p) \\ 
& jp\circ g=j(p\circ g) \hspace{2cm} jp\circ v=p\circ v.
\end{split} 
\end{equation}
Moreover, if we consider the second and third naturality equation with $p$ equals to the identity we get, using also Remark~\ref{rmk:skw-no-need-for-unary-circ_i}, the following description for $jg$ and $jq$:
\begin{center}
$jg=j(g\circ_11)=g\circ_1j1$ \\ \vspace{0.1cm}
$jq=j(q\circ_11)=q\circ_1j1$.
\end{center}
\end{rmk}

\begin{notation*}
We will denote tight $n$-ary multimaps as 
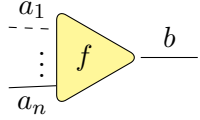
\begin{displaymath}
\begin{tikzpicture}[triangle/.style = {fill=yellow!50, regular polygon, regular polygon sides=3,rounded corners}]
\path (2,2) node [triangle,draw,shape border rotate=-90, inner sep=2.5pt,label=178:$\vdots$] (m) {$f$};

\draw [dashed] (1,2.4) to node [above] {$a_1$} (m.135);
\draw [-] (1,1.6) to node [below] {$a_n$} (m.225);
\draw [-] (m) to node [above] {$b$} (3.5,2);
\end{tikzpicture}
\end{displaymath}
and loose $n$-ary multimaps as 
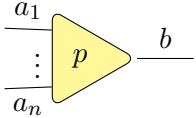
\begin{displaymath}
\begin{tikzpicture}[triangle/.style = {fill=yellow!50, regular polygon, regular polygon sides=3,rounded corners}]
\path (2,2) node [triangle,draw,shape border rotate=-90, inner sep=4pt,label=178:$\vdots$] (m) {$p$};

\draw [-] (1,2.4) to node [above] {$a_1$} (m.135);
\draw [-] (1,1.6) to node [below] {$a_n$} (m.225);
\draw [-] (m) to node [above] {$b$} (3.5,2);
\end{tikzpicture}.
\end{displaymath}
\end{notation*}

\begin{defn}
Let $\mlc$ and $\md$ two short skew multicategories. A \emph{morphism of short skew multicategories} is a functor $F\colon\mlc\to\md$ together with natural families
\begin{center}
$F^t_i\colon\mlc^t_i(\overline{a};b)\to\md^t_i(F\overline{a};Fb)$ \hspace{0.5cm} for \hspace{0.5cm} $1\leq i\leq 4$ \\
$F^l_i\colon\mlc^l_i(\overline{a};b)\to\md^l_i(F\overline{a};Fb)$ \hspace{0.5cm} for \hspace{0.5cm} $0\leq i\leq 2$
\end{center}
such that $F_1^t\equiv F$ (with $F\overline{a}$ we mean the list $Fa_1,\ldots,Fa_n$). These families must commute with all substitution operators $\circ_i$ and $j$. 
\end{defn}

Short skew multicategories and their morphisms form a category $\fsmulti$. Naturally, there is a forgetful functor $U^s\colon\smulti\to\fsmulti$. 

\subsection{The Left Representable Case}\label{subsec:left-repr-case}
A \textbf{tight binary map classifier} for $a$ and $b$ consists of a representation of $\catc_2(a,b;-)\colon\catc \to \Set$ -- in other words, a tight binary map $\theta_{a,b}\colon a,b \to ab$ for which the induced function 
$$-\circ \theta_{a,b}:\catc^{t}_1(ab;c) \to \catc^{t}_n(a,b;c)$$
is a bijection for all $c$. It is \textbf{left universal} if, moreover, the induced function 
$$-\circ_1 \theta_{a,b}\colon\catc^{t}_n(ab,\overline{x};d) \to \catc^{t}_{n+1}(a,b,\overline{x};d)$$
is a bijection for $n=2,3$ and $\overline{x}$ a tuple of the appropriate length.  A \textbf{nullary map classifier} is a representation of $\catc^{l}_2(-;-)\colon\mlc \to \Set$ --- thus, a certain nullary map $u \in \catc^{l}_2(-;i)$.  It is \textbf{left universal} if the induced function 
$$-\circ_1 u:\catc^{t}_{n+1}(i,\overline{x};d) \to \catc^{l}_{n}(\overline{x};d)$$
is a bijection for each $d$ and tuple $\overline{x}$ of length 1 and 2. 

\begin{defn}
A short skew multicategory $\catc$ is said to be \textbf{left representable} if it admits left universal nullary and tight binary map classifiers.
\end{defn}

We will denote by $\fsmulti_{lr}$ the full subcategory of $\fsmulti$ with objects left representable short multicategories. Naturally, the forgetful functor $U^s\colon\smulti\to\fsmulti$ restricts to a forgetful functor $U_{lr}^s\colon\smulti_{lr}\to\fsmulti_{lr}$. 

\begin{notation*}
Let $\mlc$ be a short skew multicategory with a left universal tight binary and nullary classifier. Then we will use $(-)'\colon\mlc^t_n(\overline{a};b)\to\mlc^t_{n-1}(a_1a_2,a_3,\ldots,a_n;b)$ for the inverse of $-\circ_1\theta{a_1,a_2}$ and $(-)^\ast\colon\mlc^l_n(\overline{a};b)\to\mlc^t_{n+1}(i,\overline{a};b)$ for the inverse of $-\circ_1u$. More precisely, for any tight $n$-multimap $f$, $f'$ is the unique tight $(n-1)$-multimap such that $f'\circ_1\theta=f$ and, for any loose $n$-multimap $q$, $q^\ast$ the unique tight $(n+1)$-multimap such that $q^\ast\circ_1u=q$. 
\end{notation*}

We start proving a characterisation of morphisms between left representable short skew multicategories (Lemma~\ref{lemma:unary-loose-morph}), 
 which will be useful in the proof of Lemma~\ref{lemma:char-sk-morph-left-repr}.

\begin{lemma}\label{lemma:unary-loose-morph}
Let $\mlc$ and $\md$ be left representable short skew multicategories and let us consider two natural families $F^l_0\colon\mlc^l_0(\diamond;a)\to\md^l_0(\diamond;Fa)$ and $F^t_2\colon\mlc^t_2(a,b;c)\to\md^t_2(Fa,Fb;Fc)$. If we define, for any loose unary map $q$, $F_1^lq:=F_2^tq^\ast\circ_1F_0^lu$ (where $u$ is the universal nullary map in $\mlc$ and $q^\ast$ is the unique binary map such $q^\ast\circ_1u=q$), then for any $v\in\mlc_0^l(\diamond;a)$ and $f\in\mlc_2^t(a,b;c)$, we have
$$F_1^l(f\circ_1v)=F_2^t(f)\circ_1F_0^l(v).$$
\end{lemma} 

\begin{proof}
Let us consider $v\in\mlc_0^l(\diamond;a)$ and $f\in\mlc_2^t(a,b;c)$. Then
\begin{align*}
&F_1^l(f\circ_1v)
& 
\\
&=F_2^t(\,(f\circ_1v)^\ast\,)\circ_1F_0^l(u)
&  (\textrm{by definition of}\,F_1^l)
\\
&=F_2^t(\,f\circ_1v^\ast\,)\circ_1F_0^l(u)
&  (\textrm{because}\,(f\circ_1v)^\ast=f\circ_1v^\ast)
\\
&=(\,F_2^t(f)\circ_1F(v^\ast)\,)\circ_1F_0^l(u)
&  (\textrm{by naturality of}\,F_2^t)
\\
&=F_2^t(f)\circ_1(\,F(v^\ast)\circ_1F_0^l(u)\,)
&  (\textrm{by dinaturality sub. nullary into tight binary})
\\
&=F_2^t(f)\circ_1F_0^l(v^\ast\circ_1u)=F_2^t(f)\circ_1F_0^l(v)
&  (\textrm{by naturality of}\,F_0^l).
\end{align*}
\end{proof}

\begin{lemma}\label{lemma:char-sk-morph-left-repr}
Let $\mlc$ and $\md$ be left representable short skew multicategories. A morphism $F\colon\mlc\to\md$ is uniquely specified by:
\begin{itemize}
\item a functor $F\colon\catc\to\catd$ between the undelying categories of $\mlc$ and $\md$;\footnote{Recall that, by definition of short skew multicategory, morphisms in $\catc$ and $\catd$ are the tight unary maps of $\mlc$ and $\md$ respectively.} 
\item natural families $F^l_0\colon\mlc^l_0(\diamond;a)\to\md^l_0(\diamond;Fa)$ and $F^t_2\colon\mlc^t_2(a,b;c)\to\md^t_2(Fa,Fb;Fc)$ such that $F$ commutes with   
\begin{equation}
\label{eq:lemma-sk-morph-nullary}
\begin{gathered}
%
%
%
%
\begin{tikzpicture}[triangle/.style = {fill=yellow!50, regular polygon, regular polygon sides=3,rounded corners}]
\path
	(2,.4) node [triangle,draw,shape border rotate=-90,inner sep=1pt] (b) {$v$} 
	(4,1) node [triangle,draw,shape border rotate=-90,label=135:$a$,label=230:$b$] (ab) {$f$};

\draw [dashed] (3.15,1.4) to (ab.140);
\draw [-] (b) .. controls +(right:1cm) and +(left:1cm).. (ab.227);
	
\draw [-] (ab) to node [above] {$c$} (5.5,1);

\end{tikzpicture} 
\end{gathered}
\end{equation}
%
%
%
%
and such that if we define, for any ternary tight map $h\in\mlc^t_3(\overline{a};b)$, $F_3^th:=F^t_2h'\circ_1F^t_2\theta$, then $F$ also commutes with 
\begin{equation}
\label{eq:lemma-sk-morph-bin}
\begin{gathered}
\begin{tikzpicture}[triangle/.style = {fill=yellow!50, regular polygon, regular polygon sides=3,rounded corners}]
\path 
	(2,0) node [triangle,draw,shape border rotate=-90,inner sep=0pt,label=135:$a$,label=230:$b$] (a) {$f$}
	(4,0.5) node [triangle,draw,shape border rotate=-90,label=135:$x$,label=230:$y$] (c) {$g$};

\draw [dashed] (3.1,0.8) to (c.139);
\draw [-] (a) .. controls +(right:1cm) and +(left:1cm).. (c.220);
\draw [dashed] (1.2,.25) to (a.140);
\draw [-] (c) to node [above] {$c$} (5,0.5);
\draw [-] (1.2,-0.25) to (a.225);
\end{tikzpicture} 
\end{gathered}
\end{equation}
and such that if we define, for any loose unary map $q$, $F_1^lq:=F_2^tq^\ast\circ_1F_0^lu$, then, for any tight unary map $p$,
\begin{equation}\label{eq:lemma-sk-morph-j}
F_1^lj(p)=jF_1^tp.
\end{equation}

\end{itemize}
\end{lemma}


\begin{proof} 
First of all, we need to define all of the natural families needed for a morphism in $\fsmulti_{lr}$. We start with $F_0^l$, $F_1^t\equiv F$ and $F_2^t$ given and we already defined $F_1^l$ and $F_3^t$. So, we have left to define $F_4^t$ and $F_2^l$:
\begin{itemize}
\item for $k\in\mlc_4^t(a,b,c,d;e)$ we define $F_4^tk:=F_3^tk'\circ_1F_2^t\theta$,
\item for $r\in\mlc_2^l(a,b;c)$ we define $F_2^lr:=F_3^tr^\ast\circ_1F_0^lu$. 
\end{itemize}
Then, we need to prove that these natural families commute with all substitutions, i.e. 
\begin{enumerate}[(i)]
	\item Tight binary into tight binary/ternary.
	\item Tight ternary into tight binary. 
	\item Nullary into tight binary/ternary.
	\item Nullary into loose unary.
	\item Loose unary into tight binary.
	\item Tight binary into loose unary. 
\end{enumerate}
Almost all of the first three can be proved in an analogous way as in Lemma~\ref{lemma:char-morph-left-repr}. For instance, to prove that $F$ preserves $g\circ_1f$ for $f$ and $g$ binary, we use (\ref{eq:lemma-sk-morph-bin}) and naturality of $F_2^t$ (in tight maps). The only exceptions are the substitution of a nullary map into the first component of a binary/tight ternary. Lemma~\ref{lemma:unary-loose-morph} proves the nullary into tight binary case. To prove the nullary into ternary case instead, we will assume substitution of loose unary in the first variable of a tight binary, which we will refer to as (v.1). Similarly, we will refer to  substitution of loose unary in the second variable of a tight binary with (v.2). \black We can see that (v.1) follows similarly to (i), using the left universal nullary map classifier, (v.2) from (i) and (iii) and (vi) from (i). For all the explicit calculations, we refer the reader to \cite[Lemma~4.5.6]{LobbiaThesis}. \black 

Finally, let us prove that $F$ commutes with $j_1$ and $j_2$. The assumption (\ref{eq:lemma-sk-morph-j}) is literally commutativity of $F$ with $j_1$. So, let $f\in\mlc_2^t(a,b;c)$, 
$$F_2^ljf=F_2^l(\,f\circ_1 j1_a\,)=F_2^tf\circ _1F_1^l(j1_a)=F_2^tf\circ_1 jF_1^t1_{a}=F_2^tf\circ_1 j1_{Fa}=j(F_2^tf).$$ 
Here we have used naturality of $j$ (in the style of Remark~\ref{rmk:jg-jp}), assumption (\ref{eq:lemma-sk-morph-j}) and $F_1^t1_{a}=1_{Fa}$ (since $F_1^t$ is a functor). 
\end{proof}

\begin{rmk}
Let us briefly see what happens to this characterisation when we consider $j=1$ both in $\mlc$ and $\md$, i.e. when they are left representable short \emph{multicategories}.  
First, condition \eqref{eq:lemma-sk-morph-j} implies that $F_1^t=F_1^l$
Hence, using Lemma~\ref{lemma:unary-loose-morph}, we see that $F$ preserves substitution $f\circ_1u$ for any binary map $f$ and any nullary map $u$, i.e. the first part of \eqref{eq:lemma-morph-nullary}. 
Furthermore, condition (\ref{eq:lemma-sk-morph-nullary}) corresponds to preserving subsitution of a nullary map in the second variable of a binary map, i.e. the second part of \eqref{eq:lemma-morph-nullary}, and (\ref{eq:lemma-sk-morph-bin}) corresponds to (\ref{eq:lemma-morph-bin}). 
This shows how Lemma~\ref{lemma:char-sk-morph-left-repr} corresponds to Lemma~\ref{lemma:char-morph-left-repr}. 
\black 
%
%
\end{rmk} 

Now, let us recall that there is an equivalence $T^s\colon\smulti_{lr}\to\skmon$ between left representable skew multicategories and skew monoidal categories \cite[Theorem~6.1]{LackBourke:skew}.  

\begin{lemma}
\label{lemma:skew-fin-mult-to-mult}
Given a left representable short skew multicategory $\mlc$ we can construct a skew monoidal category $K^s\mlc$ in which:
\begin{itemize}
\item The tensor product $ab$ of two objects $a$ and $b$ is the tight binary map classifier;
\item The unit $i$ is the nullary map classifier;
\item Given tight unary maps $f\colon a \to b$ and $g\colon c \to d$, the tensor product $fg\colon ac \to bd$ is the unique morphism such that 
\begin{equation}
\label{eq:skew-prod-maps}
\begin{gathered}
\begin{tikzpicture}[triangle/.style = {fill=yellow!50, regular polygon, regular polygon sides=3,rounded corners}]
\path (1,1) node [triangle,draw,shape border rotate=-90,label=135:$a$,label=230:$c$,inner sep=0pt] (c) {$\theta_{a,c}$} 
	(3,1) node [triangle,draw,shape border rotate=-90,inner sep=0pt] (ci) {$fg$}
	(9,0.5) node [triangle,draw,shape border rotate=-90,inner sep=1.5pt] (a') {$g$}
	(7,1.5) node [triangle,draw,shape border rotate=-90,inner sep=1pt] (b') {$f$} 
	(11,1) node [triangle,draw,shape border rotate=-90,label=135:$b$,label=230:$d$,inner sep=0pt] (c') {$\theta_{b,d}$};
	
\draw [-] (0,.6) to (c.220);
\draw [dashed] (0,1.35) to (c.140);
\draw [dashed] (c) to node [above] {$ac$} (ci);
\draw[dashed] (8,0.5) to node [below] {$c$} (a');
\draw[dashed] (6,1.5) to node [above] {$a$} (b');
\draw [-] (a') .. controls +(right:1cm) and +(left:1cm).. (c'.220);
\draw [dashed] (b') .. controls +(right:2cm) and +(left:1cm).. (c'.140);
	
\draw [-] (ci) to node [above] {$bd$} (4.25,1);

\node () at (5.25,1) {$=$};

	\draw [-] (c') to node [above] {$bd$} (12.5,1);
\end{tikzpicture}
\end{gathered}
\end{equation}
\item The associator $\alpha\colon (ab)c\to a(bc)$ is defined as the unique tight map such that 
\begin{equation}
\label{eq:univ-skew-alpha}
\begin{gathered}
\begin{tikzpicture}[triangle/.style = {fill=yellow!50, regular polygon, regular polygon sides=3,rounded corners}]
\path (0,0.5) node [triangle,draw,shape border rotate=-90,inner sep=0pt] (a) {$\theta_{b,c}$} 
	(2.5,1) node [triangle,draw,shape border rotate=-90,inner sep=-1.5pt,label=135:$a$,label=230:$bc$] (c) {$\theta_{a,bc}$} 
	(7,1.35) node [triangle,draw,shape border rotate=-90,inner sep=0pt] (b') {$\theta_{a,b}$} 
	(9,1) node [triangle,draw,shape border rotate=-90,inner sep=-1.5pt,label=135:$ab$,label=230:$c$] (c') {$\theta_{ab, c}$}
	(11.5,1) node [triangle,draw,shape border rotate=-90,inner sep=1pt] (f) {$\alpha$};
	
\draw [-] (a) to (c.230);
\draw [dashed] (1,1.45) to (c.135);
\draw [-] (7.9,0.65) to (c'.220);
\draw [dashed] (b') to (c'.140);
	
\draw [-] (c) to node [above] {$a(bc)$} (4,1);
\draw [dashed] (-.85,0.825) to node [above] {$b$} (a.140);
\draw [-] (-.85,0.175) to node [below] {$c$} (a.220);

\node () at (5,1) {$=$};

\draw [dashed] (c') to node [above] {$(ab)c$} (f);
\draw [-] (f) to node [above] {$a(bc)$} (13,1);
\draw [dashed] (6.15,1.7) to node [above] {$a$} (b'.140);
\draw [-] (6.15,1) to node [below] {$b$} (b'.220);
\end{tikzpicture}
\end{gathered}
\end{equation}
\item The left unit map $\lambda\colon ia\to a$ is defined as the unique tight unary map such that 
\begin{equation}
\label{eq:univ-skew-lamda}
\begin{gathered}
\begin{tikzpicture}[triangle/.style = {fill=yellow!50, regular polygon, regular polygon sides=3,rounded corners}]
\path (2.5,1) node [triangle,draw,shape border rotate=-90,inner sep=-0.5pt] (c) {$j(1_a)$} 
	(6.5,1.35) node [triangle,draw,shape border rotate=-90,inner sep=1pt] (b') {$u$} 
	(8.5,1) node [triangle,draw,shape border rotate=-90,label=135:$i$,label=230:$a$,inner sep=-0.5pt] (c') {$\theta_{i,a}$}
	(10.5,1) node [triangle,draw,shape border rotate=-90] (c'') {$\lambda$};
	
\draw [-] (c) to node [above] {$a$} (4,1);
\draw [-] (1.5,1) to node [above] {$a$} (c);

	\node () at (5,1) {$=$};

\draw [-] (7.5,0.65) to (c'.223);
\draw [dashed] (b') to (c'.137);
\draw [dashed] (c') to node [above] {$ia$} (c'');
\draw [-] (c'') to node [above] {$a$} (12,1);
\end{tikzpicture}
\end{gathered}
\end{equation}
\item The right unit map $\rho\colon a\to ai$ is the tight unary map defined as
\begin{equation}
\label{eq:univ-skew-rho}
\begin{gathered}
\begin{tikzpicture}[triangle/.style = {fill=yellow!50, regular polygon, regular polygon sides=3,rounded corners}]
\path (0.5,0.55) node [triangle,draw,shape border rotate=-90,inner sep=1pt] (a) {$u$} 
	(2.5,1) node [triangle,draw,shape border rotate=-90,inner sep=0pt,label=135:$a$,label=230:$i$] (c) {$\theta_{a,i}$};

\draw [-] (a) to (c.230);
\draw [dashed] (1.2,1.4) to (c.135);
\draw [-] (c) to node [above] {$ai$} (4,1);

\end{tikzpicture}
\end{gathered}
\end{equation}
\end{itemize}
\end{lemma}

\begin{proof}
Functoriality of $\catc^2\to\catc:(a,b)\mapsto ab$ follows from the universal property of the tight binary map classifier and profunctoriality of $\mlc_2^t(-;-)$. It remains to verify the five axioms for a skew monoidal category. 
Some of them have the same proof as short multicategories, we will give details only of the ones where we need to use new axioms. For instance, in the pentagon axiom (\ref{ax:mon-pent}) all the maps are tight, so the proof does not change (naturality and profunctoriality are defined using tight unary maps). Also the right unit axiom (\ref{ax:mon-rho}) has the same proof. 
\black For the rest of the axioms, the only thing that changes from the non-skew case is that we need to use also the naturality of $j$. This comes up since for a skew multicategory $\lambda_a\circ\theta_{i,a}\circ_1u=j1_a$ (by definition of $\lambda$). We refer the interested reader to \cite[Lemma~4.5.8]{LobbiaThesis} for the details. \black 

\end{proof}

Before defining the functor $K^s\colon\fsmulti_{lr}\to\skmon$, we prove the following easy lemma, which is the counterpart of Lemma~\ref{lem:bij}.  

\begin{lemma}\label{lemma:sk-bij}
Consider $\mlc, \md \in \fsmulti_{lr}$ and a functor $F\colon \cc \to \cd$.  There is a bijection between natural families, with $x\in\lbrace t,l\rbrace$,
$$F^x_{\overline{a},b}\colon\mlc^x_i(\overline{a};b)\to\md^x_i(F\overline{a};Fb)$$
and natural families
$$f^x_{\overline{a}}\colon m^x(F\overline{a})\to F(m^x\overline{a})$$ 
where $m^x\overline{a}$ and $m^x(F\overline{a})$ are the $n$-ary $x$-map classifiers of the appropriate arity.
\end{lemma}

\begin{proof}
The bijection is governed by the following diagram
\begin{equation}\label{eq:yoneda}
\begin{tikzcd}[ampersand replacement=\&]
	{\mlc^x_i(\overline{a};-)} \&\& {\md^x_i(F\overline{a};F-)} \\
	{\mlc^t_1(m^x\overline{a},-)} \&\& {\md^t_1(m^x(F\overline{a}),F-)}
	\arrow["{F^x_{\overline{a},-}}", from=1-1, to=1-3]
	\arrow["{- \circ_1 \theta^x_{\overline{a}}}","\cong"', from=2-1, to=1-1]
	\arrow["{F- \circ f^x_{\overline{a}}}"', from=2-1, to=2-3]
	\arrow["{- \circ_1 \theta^x_{F\overline{a}}}"',"\cong", from=2-3, to=1-3]
\end{tikzcd}
\end{equation}
in which the vertical arrows are natural bijections and the lower horizontal arrow corresponds to the upper one using the Yoneda lemma. \black 
\end{proof}

\begin{rmk}
\label{rmk:skew-K-on-map}
Given a morphism $F\colon \mlc \to \md \in \fsmulti_{lr}$ we obtain, on applying the above lemma, natural families of tight maps $f_{2}\colon FaFb \to F(ab)$ and $f_0\colon i \to Fi$ defining the \emph{data} for a lax monoidal functor $K^sF\colon K^s\mlc \to K^s\md$. That it is a lax monoidal functor follows directly from the following result.

Explicitly, $f_{2}\colon FaFb \to F(ab)$ is the unique morphism such that $f_{2} \circ_{1} \theta_{Fa,Fb} = F^t_2(\theta_{a,b})$ whilst $f_{0}$ is the unique morphism such that $f_{0} \circ u = F_0^lu$.
\end{rmk}

\begin{prop}\label{prop:sk-full-faith}
With the definition on objects given in Lemma~\ref{lemma:skew-fin-mult-to-mult} and on morphisms in Lemma~\ref{rmk:skew-K-on-map}, we obtain a fully faithful functor $K^s\colon\fsmulti_{lr} \to\skmon$. 
\end{prop}

\begin{proof}
The proof is quite similar to the proof of Proposition~\ref{prop:full-faith}. Using Lemma~\ref{lemma:char-sk-morph-left-repr} and~\ref{lemma:sk-bij}, it is enough to prove that $(F_0^l,F_2^t)$ satisfy the equations (\ref{eq:lemma-sk-morph-nullary}, \ref{eq:lemma-sk-morph-bin}, \ref{eq:lemma-sk-morph-j}) if and only if $(f_0,f_2)$ satisfy the equations for a lax monoidal functor (\ref{eq:mon-moprh-alpha}, \ref{eq:mon-morph-lambda}, \ref{eq:mon-morph-rho}). Equation ~(\ref{eq:lemma-sk-morph-bin}) corresponds to the associator axiom (\ref{eq:mon-moprh-alpha}) for a lax monoidal functor. Since all maps involved are tight, this follows by the same proof in Proposition~\ref{prop:full-faith}. In a similar way, we can prove how (\ref{eq:lemma-sk-morph-nullary}) corresponds to the right unit axiom (\ref{eq:mon-morph-rho}) of a lax monoidal functor. The only part that changes significantly is the one regarding the left unit axiom (\ref{eq:mon-morph-lambda}). We need to change the proof because the substitution of a nullary map into the first variable of a tight binary gives a \emph{loose unary} map. Let us start proving that if $(F_0^l,F_2^t)$ satisfy (\ref{eq:lemma-sk-morph-j}), then $(f_0,f_2)$ satisfy the left unit axiom (\ref{eq:mon-morph-lambda}). We will prove this axioms showing that $F\lambda\circ f_2\circ f_0Fa$ satisfy the defining property (\ref{eq:univ-skew-lamda}) of $\lambda$. 
\begin{center}
\begin{tikzpicture}[triangle/.style = {fill=yellow!50, regular polygon, regular polygon sides=3,rounded corners}]
\path
	(3,1) node [triangle,draw,shape border rotate=-90,inner sep=3.5pt] (c') {$u$}
	(5.5,0.5) node [triangle,draw,shape border rotate=-90,inner sep=3.5pt,label=135:$i$,label=230:$Fa$] (d) {$\theta$}
	(8,0.5) node [triangle,draw,shape border rotate=-90,inner sep=-2pt] (f) {$f_0Fa$}
	(10.75,0.5) node [triangle,draw,shape border rotate=-90,inner sep=2pt] (f') {$f_2$}
	(13,0.5) node [triangle,draw,shape border rotate=-90,inner sep=0pt] (f'') {$F\lambda$};
	
\draw [dashed] (c') .. controls +(right:1cm) and +(left:1cm).. (d.140);
\draw [-] (4.4,0.1) to (d.225);
\draw [dashed] (d) to node [above] {$iFa$} (f);%
\draw [dashed] (f) to node [above] {$FiFa$} (f');%
\draw [dashed] (f') to node [above] {$F(ia)$} (f'');%
\draw [-] (f'') to node [above] {$Fa$} (15,0.5);
\end{tikzpicture} \\
by definition of $f_0\cdot Fa$, $f_0$ and $f_2$ \\
\begin{tikzpicture}[triangle/.style = {fill=yellow!50, regular polygon, regular polygon sides=3,rounded corners}]
\path
	(3.5,0.75) node {$=$}
	(5,1) node [triangle,draw,shape border rotate=-90,inner sep=0pt] (f) {$Fu$}
	(7.5,0.5) node [triangle,draw,shape border rotate=-90,inner sep=0pt,label=135:$Fi$,label=230:$Fa$] (d) {$F\theta$}
	(10,0.5) node [triangle,draw,shape border rotate=-90,inner sep=0pt] (f'') {$F\lambda$};
	
\draw [dashed] (f) .. controls +(right:1cm) and +(left:1cm).. (d.140);
\draw [-] (6.25,0.1) to (d.225);
\draw [dashed] (d) to node [above] {$F(ia)$} (f'');%
\draw [-] (f'') to node [above] {$Fa$} (11.5,0.5);
\end{tikzpicture} 
\end{center}
\begin{align*}
&(\,F\lambda\circ F^t_2\theta\,)\circ_1F^l_0u 
& \\
&= F^t_2(\lambda\circ\theta)\circ_1F_0^lu
& (\textrm{by naturality of}\,F^t_2) 
\\ 
&= F_1^l(\,(\lambda\circ\theta_{i,a})\circ_1u\,)
& (\textrm{by Lemma~\ref{lemma:unary-loose-morph}})
 \\
&= F_1^l(j1_a)
& (\textrm{by defining property (\ref{eq:univ-skew-lamda}) of}\,\lambda)
\\ 
&= j1_{Fa}
& (\textrm{by assumption (\ref{eq:lemma-sk-morph-j})}).
\end{align*}
On the other hand, let us assume $(f_0,f_2)$ satisfy the left unit axiom (\ref{eq:mon-morph-lambda}). Then, by universal property of $\lambda$, $j1_{Fa}$ is equal to 
\begin{center}
\begin{tikzpicture}[triangle/.style = {fill=yellow!50, regular polygon, regular polygon sides=3,rounded corners}]
\path
	(5,1) node [triangle,draw,shape border rotate=-90,inner sep=0pt] (f) {$Fu$}
	(7.5,0.5) node [triangle,draw,shape border rotate=-90,inner sep=3.5pt,label=135:$i$,label=230:$Fa$] (d) {$\theta$}
	(10,0.5) node [triangle,draw,shape border rotate=-90,inner sep=3.5pt] (f'') {$\lambda$};
	
\draw [dashed] (f) .. controls +(right:1cm) and +(left:1cm).. (d.140);
\draw [-] (6.25,0.1) to (d.225);
\draw [dashed] (d) to node [above] {$iFa$} (f'');%
\draw [-] (f'') to node [above] {$Fa$} (11.5,0.5);
\end{tikzpicture} \\
by left unit axiom (\ref{eq:mon-morph-lambda}) \\
\begin{tikzpicture}[triangle/.style = {fill=yellow!50, regular polygon, regular polygon sides=3,rounded corners}]
\path
	(-1.5,0.75) node {$=$}
	(0,1) node [triangle,draw,shape border rotate=-90,inner sep=3.5pt] (u) {$u$}
	(4,0.5) node [triangle,draw,shape border rotate=-90,inner sep=-1.5pt] (f0) {$f_0Fa$}
	(2,0.5) node [triangle,draw,shape border rotate=-90,inner sep=3.5pt,label=135:$i$,label=230:$Fa$] (t) {$\theta$}
	(6.5,0.5) node [triangle,draw,shape border rotate=-90,inner sep=2pt] (v) {$f_2$}
	(9,0.5) node [triangle,draw,shape border rotate=-90,inner sep=1pt] (f2) {$F\lambda$};
\draw[-] 
	(0.75,0.1) to (t.225)
	(f2) to node [above] {$Fa$} (11,0.5);
	
\draw[dashed]
	(t) to node [above] {$iFa$} (f0)
	(f0) to node [above] {$FiFa$} (v)
	(u) .. controls +(right:1cm) and +(left:1cm).. (t.140)
	(v) to node [above] {$F(ia)$} (f2);
\end{tikzpicture} \\
by definition of $f_0Fb$ \\
\begin{tikzpicture}[triangle/.style = {fill=yellow!50, regular polygon, regular polygon sides=3,rounded corners}]
\path
	(-1.5,0.75) node {$=$}
	(0,1) node [triangle,draw,shape border rotate=-90,inner sep=3.5pt] (u) {$u$}
	(2,1) node [triangle,draw,shape border rotate=-90,inner sep=2pt] (f0) {$f_0$}
	(4,0.5) node [triangle,draw,shape border rotate=-90,inner sep=3.5pt,label=135:$Fi$,label=230:$Fa$] (t) {$\theta$}
	(6.5,0.5) node [triangle,draw,shape border rotate=-90,inner sep=2pt] (v) {$f_2$}
	(9,0.5) node [triangle,draw,shape border rotate=-90,inner sep=1pt] (f2) {$F\lambda$};
	
\draw[-] 
	(2.75,0.1) to (t.225)
	(f2) to node [above] {$Fa$} (11,0.5);
\draw[dashed]
	(u) to node [above] {$i$} (f0)
	(t) to node [above] {$FiFa$} (v)
	(f0) .. controls +(right:1cm) and +(left:1cm).. (t.140)
	(v) to node [above] {$F(ia)$} (f2);
\end{tikzpicture} \\
\end{center}
\begin{align*}
&= (\,F\lambda_a\circ F_2^t\theta_{i,a}\,)\circ_1F_0^lu
& (\textrm{by definition of}\,F_0^l\,\textrm{and}\,F_2^t) 
\\ 
&= F_2^t(\,\lambda\circ\theta_{i,a}\,)\circ_1F_0^lu
& (\textrm{by naturality of}\,F_2^t)
 \\
&= F_1^l(\,(\lambda\circ\theta_{i,a})\circ_1 u\,)
& (\textrm{by Lemma~\ref{lemma:unary-loose-morph}})
 \\
&= F_1^l(j1_a)
& (\textrm{by defining property (\ref{eq:univ-skew-lamda}) of}\,\lambda).
\end{align*}


Therefore, we get a correspondence analogous to the one in Table~\ref{tab:axioms-lax-mon}, which is described in Table~\ref{tab:sk-axioms-lax-mon}. 
\begin{table}[h]
\centering 
\renewcommand\arraystretch{1.25}
\begin{tabular}{|c | c|}
\hline 
$(F^l_0,F^t_2)$ & $(f_0,f_2)$ \\
\hline
(\ref{eq:lemma-sk-morph-bin}) & Associator axiom (\ref{eq:mon-moprh-alpha})\\
(\ref{eq:lemma-sk-morph-j}) & Left unit axiom (\ref{eq:mon-morph-lambda})\\
(\ref{eq:lemma-sk-morph-nullary}) & Right unit axiom (\ref{eq:mon-morph-rho})\\
\hline
\end{tabular}
\caption{}\label{tab:sk-axioms-lax-mon}
\end{table} \qedhere
\end{proof}

We recall that there is a forgetful functor $U_{lr}^s\colon\smulti_{lr}\to\fsmulti_{lr}$ and an equivalence $T^s\colon\smulti_{lr}\to\skmon$ between left representable skew multicategories and skew monoidal categories. Moreover, comparing the construction of $K^s$ with that given in \cite[Section~6.2]{LackBourke:skew}, we see that the triangle
\[\begin{tikzcd}
	{\smulti_{lr}} \\
	& {\fsmulti_{lr}} \\
	{\skmon.}
	\arrow["T^s"', from=1-1, to=3-1]
	\arrow["U_{lr}^s", from=1-1, to=2-2]
	\arrow["K^s", from=2-2, to=3-1]
\end{tikzcd}\] 
is commutative.

\begin{theorem}\label{them:sk-fin-equiv}
The functor $K^s\colon\fsmulti_{lr} \to \skmon$ is an equivalence of categories, as is the forgetful functor $U^s\colon\smulti_{lr}\to\fsmulti_{lr}$.
\end{theorem}

\begin{proof} 
Let us show that $K^s$ is an equivalence first. Since $K^s$ is fully faithful by Proposition~\ref{prop:sk-full-faith}, it remains to show that is essentially surjective on objects. Since $T^s=K^sU^s$ and the equivalence $T^s$ is essentially surjective, so is $K^s$, as required. Finally, since $T^s=K^sU^s$ and both $T^s$ and $K^s$ are equivalences, so is $U^s$. 
\end{proof}

\subsection{The Closed Left Representable Case}

\begin{defn}
A short skew multicategory is said to be \textbf{closed} if all $b,c\in\mlc$ there exists an object $[b,c]$ and a tight binary map $e_{b,c}\colon [b,c],b\to c$ for which the induced functions
\begin{center}
$e_{b,c}\circ_1-\colon\mlc^t_n(\overline{x};[b,c]) \to \mlc^t_{n+1}(\overline{x},b;c),$ \hspace{0.5cm} for $n=1,2,3$, \\ \vspace{0.1cm}
$e_{b,c}\circ_1-\colon\mlc^l_n(\overline{x};[b,c]) \to \mlc^l_{n+1}(\overline{x},b;c),$ \hspace{0.5cm} for $n=0,1$,
\end{center}
are isomorphisms. 
\end{defn}

Once again, let us notice that the restrictions on the arities $n$ are determined by the definition of a short skew multicategory. For instance, when dealing with loose maps we only consider $n=0,1$ because we do not have ternary loose maps in a short skew multicategory. 

We will denote with $\fsmulti^{cl}_{lr}$ the full subcategory of $\fsmulti$ with objects left representable closed short skew multicategories. Naturally, the forgetful functor $U^s_{lr}\colon\smulti_{lr}\to\fsmulti_{lr}$ restricts to a forgetful functor
\begin{center}
$U_{lr}^{s,cl}\colon\smulti_{lr}^{cl}\to\fsmulti_{lr}^{cl}$.
\end{center} Adapting Proposition~\ref{prop:left-iff-adj}, we get a characterisation of closed short skew multicategories which are also left representable. 

\begin{prop}
A closed short skew multicategory is left representable if and only if it has a nullary map classifier and each $[b,-]$ has a left adjoint. 
\end{prop}

\begin{proof}
If $\mlc$ is closed and left representable, then the natural bijections
$$\catc(ab,c)=\mlc_1^t(ab;c)\cong\mlc_2^t(a,b;c)\cong\mlc_1^t(a;[b,c])=\catc(a,[b,c])$$
show that $-b\dashv[b,-]$. Conversely, if $[b,-]$ has a left adjoint. Then we have natural isomorphisms
$$\mlc_1^t(ab;c)=\catc(ab,c)\cong\catc(a,[b,c])=\mlc_1^t(a;[b,c])\cong\mlc_2^t(a,b;c)$$
and, by Yoneda, the composite is of the form $-\circ_1\theta_{a,b}$ for a tight binary map classifier $\theta_{a,b}\colon a,b\to ab$. It remains to show that this and the nullary map classifier are left universal.  For the tight binary map classifier, we must show that
$- \circ \theta_{a,b}\colon\mlc^t_{n+1}(ab,\overline{x};c) \to \mlc^t_{n+2}(a,b, \overline{x};c)$ is a bijection for all $\overline{x}$ of length $1$ or $2$, the case $0$ being known. For an inductive style argument, suppose it is true for $\overline{x}$ of length $i\leq 1$.  We should show that the bottom line below is a bijection
\begin{equation*}
\begin{tikzcd}[ampersand replacement=\&]
	{\mlc^t_{i+1}(ab,\overline{x};[y,c])} \&\& {\mlc^t_{i+2}(a,b,\overline{x};[y,c])} \\
	{\mlc^t_{i+2}(ab,\overline{x},y;c)} \&\& {\mlc^t_{i+3}(a,b,\overline{x},y;c)}
	\arrow["{- \circ_1 \theta_{a,b}}", from=1-1, to=1-3]
	\arrow["{e_{y,c} \circ_1 -}"', from=1-1, to=2-1]
	\arrow["{e_{y,c} \circ_1 -}", from=1-3, to=2-3]
	\arrow["{- \circ_1 \theta_{a,b}}"', from=2-1, to=2-3]
\end{tikzcd}
\end{equation*}
but this follows from the fact that the square commutes, by associativity axiom (\ref{eq:sk-ass-line}.a), and the other three morphisms are bijections, by assumption.  The case of the nullary map classifier is similar in form but uses associativity axiom (\ref{eq:sk-ass-line}.b). 
\end{proof}

We recall that \cite[Theorem~6.4]{LackBourke:skew} gives an equivalence $T^s_{c}\colon\smulti_{lr}^{cl}\to\skmon^{cl}$ between left representable closed skew multicategories and skew closed monoidal categories. 

\begin{theorem}\label{thm:skew-left-closed-equiv}
The equivalence $K^s\colon\fsmulti_{lr} \to \skmon$ (from Theorem~\ref{them:sk-fin-equiv}) restricts to an equivalence $K^s_c\colon\fsmulti_{lr}^{cl} \to \skmon^{cl}$ between left representable closed short skew multicategories and skew closed monoidal categories, which fits in the commutative triangle of equivalences below. 
 \[\begin{tikzcd}
	{\smulti_{lr}^{cl}} \\
	& {\fsmulti^{cl}_{lr}} \\
	{\skmon^{cl}}
	\arrow["T^s_{c}"', from=1-1, to=3-1]
	\arrow["U_{lr}^{s,cl}", from=1-1, to=2-2]
	\arrow["K^s_c", from=2-2, to=3-1]
\end{tikzcd}\]
\end{theorem}

\begin{proof}
The strategy is to prove that a left representable short skew multicategory $\mlc$ is closed if and only if the skew monoidal category $K^s\mlc$ monoidal skew closed.

We start by considering when $\mlc\in~\fsmulti_{lr}$ is a closed short skew multicategory. By definition of closedness and left representability, we have natural isomorphisms 
\begin{equation*}
\catc(ab,c) =  \mlc^t_1(ab;c)\cong \mlc^t_2(a,b;c) \cong \mlc^t_1(a,[b,c]) = \catc(a,[b,c]),
\end{equation*}
therefore $K^s\mlc$ is monoidal skew closed, as required.

On the other hand, if $K^s\mlc$ is closed, then we have,  for all $a,b,c\in K^s\mlc$, natural isomorphisms $\mlc^t_1(a;[b,c]) \cong \mlc^t_1(ab;c)$. By Yoneda, the composite
\begin{equation*}
\mlc^t_1(a;[b,c]) \cong  \mlc^t_1(ab;c)\cong \mlc^t_2(a,b;c)
\end{equation*}
 is of the form $e_{b,c} \circ_1 -$ for a tight binary map $e_{b,c}\colon [b,c],b \to c$, and to show that $\mlc$ is closed we must prove that 
\begin{center}
$e_{b,c}\circ_1-\colon\mlc^t_n(\overline{a};[b,c]) \to \mlc^t_{n+1}(\overline{a},b;c),$ \hspace{0.5cm} for $n=2,3$, \\ \vspace{0.1cm}
$e_{b,c}\circ_1-\colon\mlc^l_n(\overline{a};[b,c]) \to \mlc^l_{n+1}(\overline{a},b;c),$ \hspace{0.5cm} for $n=0,1$,
\end{center} 
are bijections. For the tight maps case we can consider the diagram
 \begin{equation*}
 \xymatrix{
 \mlc^t_1(m^t(\overline{a});[b,c]) \ar[d]^{\cong}_{- \circ_1 \theta_{\overline{a}}} \ar[rr]^{e_{b,c} \circ{_1} -} && \mlc^t_2(m^t(\overline{a}),b;c) \ar[d]_{\cong}^{- \circ_1 \theta_{\overline{a}}} \\
\mlc^t_n(\overline{a};[b,c])  \ar[rr]_{e_{b,c} \circ{_1} -} && \mlc^t_{n+1}(\overline{a},b;c)} 
\end{equation*}
where $\theta_{\overline{a}}\colon\overline{a} \to m^t(\overline{a})$ is the left universal tight $n$-multimap. More precisely, 
\begin{center}
\begin{tabular}{llll}
for $n=2$,   & $m^t(a_1,a_2):=a_1a_2$, & and            & $\theta_{\overline{a}}:=\theta_{a_1,a_2}$,                   \vspace{0.1cm}      \\ 
for $n=3$,  & $m^t(a_1,a_2,a_3):=(a_1a_2)a_3$, & and       & $\theta_{\overline{a}}:=\theta_{a_1a_2,a_3}\circ_1\theta_{a_1,a_2}$.  
\end{tabular}
\end{center}
The commutativity of the diagram with $n=2$ follows from dinaturality of substitution of tight binary into tight ternary, whereas the one with $n=3$ follows from the associativity axiom (\ref{eq:sk-ass-line}.a). Thus, since the two vertical functions are invertible by left representability and the 
upper horizontal by construction, the lower horizontal is invertible as well. Similarly, for the loose case we consider the diagram
 \begin{equation*}
 \xymatrix{
 \mlc^t_1(m^l(\overline{a});[b,c]) \ar[d]_{- \circ_1 \theta^l_{\overline{a}}} \ar[rr]^{e_{b,c} \circ{_1} -} && \mlc^t_2(m^l(\overline{a}),b;c) \ar[d]^{- \circ_1 \theta^l_{\overline{a}}} \\
\mlc^l_n(\overline{a};[b,c])  \ar[rr]_{e_{b,c} \circ{_1} -} && \mlc^l_{n+1}(\overline{a},b;c)} 
\end{equation*}
where $\theta^l_{\overline{a}}\colon\overline{a} \to m(\overline{a})$ is the left universal loose $n$-multimap, i.e.
\begin{center}
\begin{tabular}{llll}
for $n=0$,   & $m^l(-):=i$, & and            & $\theta^l_{-}:=u$,                    \vspace{0.1cm}      \\ 
for $n=1$,  & $m^l(a):=a$, & and       & $\theta^l_{a}:=\theta_{i,a}\circ_1u$. 
\end{tabular}
\end{center}
This time, when $n=0$, the commutativity of the diagram follows from dinaturality of substitution of nullary into tight binary, whereas when $n=1$ follows from the associativity axiom (\ref{eq:sk-ass-line}.b). Then, since the other three maps are invertible, the lower horizontal map is an isomorphism. 
\end{proof}
\black 

\subsection{Short Braidings}

\black
Let $\mc$ be a short skew multicategory. A \textbf{short braiding} on $\mc$ consists of natural isomorphisms 
\begin{center}
$\beta^3_2\colon\mc_3^t(a_1,a_2,a_3;b)\to\mc_3^t(a_1,a_3,a_2;b)$ \\ \vspace{0.1cm}
$\beta^4_2\colon\mc_4^t(a_1,a_2,a_3,a_4;b)\to\mc_4^t(a_1,a_3,a_2,a_4;b)$ \\ \vspace{0.1cm}
$\beta^4_3\colon\mc_4^t(a_1,a_2,a_3,a_4;b)\to\mc_4^t(a_1,a_2,a_4,a_3;b)$ 
\end{center}

satisfying the following six axioms
\begin{itemize}
\item for any tight 4-ary map $h\in\mc_4^l(a_1,a_2,a_3,a_4;b)$, 
\begin{equation}\label{ax:sh-braid-4-ary}
\beta^4_2\beta^4_3\beta^4_2(h)=\beta^4_3\beta^4_2\beta^4_3(h)
\end{equation}

\item for any tight binary map $g\colon b_1,b_2\to c$ and tight ternary map $f\colon a_1,a_2,a_3\to b_i$,
\begin{minipage}{0.5\textwidth}
\begin{equation}\label{ax:sh-braid-3-in-2-first}
g\circ_1\beta^3_2(f)=\beta^4_2(g\circ_1f)
\end{equation}
\end{minipage} 
\begin{minipage}{0.5\textwidth}
\begin{equation}\label{ax:sh-braid-3-in-2-secd}
g\circ_2\beta^3_2(f)=\beta^4_3(g\circ_2f)
\end{equation}
\end{minipage}\vspace{0.2cm}

\item for any tight ternary map $g\colon b_1,b_2,b_3\to c$ and tight binary map $f\colon a_1,a_2\to$~$b_i$,
\begin{minipage}{0.5\textwidth}
\begin{equation}\label{ax:sh-braid-2-in-3-first}
\beta^4_3(g\circ_1f)=\beta^3_2(g)\circ_1f
\end{equation}
\end{minipage}
\begin{minipage}{0.5\textwidth}
\begin{equation}\label{ax:sh-braid-2-in-3-secd}
\beta^4_2\beta^4_3(g\circ_2f)=\beta^3_2(g)\circ_3f
\end{equation}
\end{minipage}
\begin{equation}\label{ax:sh-braid-2-in-3-thrd}
\beta^4_3\beta^4_2(g\circ_3f)=\beta^3_2(g)\circ_2f
\end{equation}
\end{itemize}
The short braiding is called a \textbf{short symmetry} if, moreover, 
\begin{equation}\label{ax:sh-symm}
\begin{tikzcd}[ampersand replacement=\&]
	\& {\mc_3^t(a_1,a_3,a_2;b)} \\
	{\mc_3^t(a_1,a_2,a_3;b)} \&\& {\mc_3^t(a_1,a_2,a_3;b)}
	\arrow["{\beta^3_2}", from=2-1, to=1-2]
	\arrow["{\beta^3_2}", from=1-2, to=2-3]
	\arrow[Rightarrow, no head, from=2-1, to=2-3]
\end{tikzcd}
\end{equation}
We call a short skew multicategory together with a short braiding a \textbf{braided short skew multicategory}. Moreover, we say that it is \textbf{symmetric} if the braiding is a symmetry.  

\begin{rmk}
	Axiom \eqref{ax:sh-braid-4-ary} can be visulised through a commutative diagram (see below), where the nodes represent the inputs of the 4-ary map. 
	\[
	\scalebox{0.9}
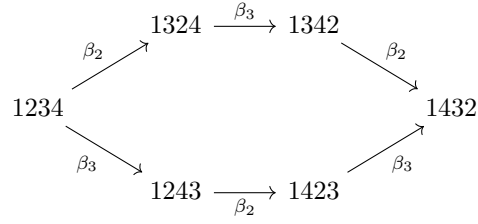
{\begin{tikzcd}[ampersand replacement=\&]
		\& 1324 \& 1342 \\
		1234 \&\&\& 1432 \\
		\& 1243 \& 1423
		\arrow["{\beta_3}", from=1-2, to=1-3]
		\arrow["{\beta_2}", from=1-3, to=2-4]
		\arrow["{\beta_2}", from=2-1, to=1-2]
		\arrow["{\beta_3}"', from=2-1, to=3-2]
		\arrow["{\beta_2}"', from=3-2, to=3-3]
		\arrow["{\beta_3}"', from=3-3, to=2-4]
	\end{tikzcd}}\]
\end{rmk}
\black

\begin{rmk}
In the definition above we do not consider any action on binary maps, even if in a short skew multicategory we could consider the action 
$$\beta^2\colon\mc_2^l(a,b;c)\to\mc_2^l(b,a;c).$$ 
The reason behind this choice is that when a short multicategory has a left universal nullary map classifier, then $\beta^2$ can be described using $\beta^3_2$ as below: 
$$\beta^2:=\quad\mc_2^l(a,b;c)\cong\mc_3^t(i,a,b;c)\xrightarrow{\beta^3_2}\mc_3^t(i,b,a;c)\cong\mc_2^l(b,a;c).$$
\end{rmk}

Given two braided short skew multicategories $\mc$ and $\md$, we say that a short skew multifunctor $F\colon\mc\to\md$ is \textbf{braided} if it respects the braiding isomorphisms, i.e. if for $r=\beta^3_2,\beta^4_2,\beta^4_3$ the following diagram commutes
\[\begin{tikzcd}[ampersand replacement=\&]
	{\mc_n^t(a_1,\ldots,a_n;c)} \& {\mc_n^t(a_{r1},\ldots,a_{rn};c)} \\
	{\md_n^t(Fa_1,\ldots,Fa_n;Fc)} \& {\md_n^t(Fa_{r1},\ldots,Fa_{rn};Fc)}
	\arrow["{^\mc r^\ast}", from=1-1, to=1-2]
	\arrow["{F_n^t}", from=1-2, to=2-2]
	\arrow["{F_n^t}"', from=1-1, to=2-1]
	\arrow["{^\md r^\ast}"', from=2-1, to=2-2]
\end{tikzcd}\]

There is a category $\brdfsmulti$ of braided short skew multicategories and braided short multifunctors. We call $\symfsmulti$ the full subcategory of of $\brdfsmulti$ with objects symmetric short skew multicategories. Naturally, the forgetful functor $U^s\colon\smulti\to\fsmulti$ restricts to forgetful functors 
\begin{center}
$U^{brd}_{lr}\colon\brdsmulti_{lr}\to\brdfsmulti_{lr}$ \\ \vspace{0.1cm}
$U^{sym}_{lr}\colon\symsmulti_{lr}\to \symfsmulti_{lr}$.
\end{center}

Now, we want to show that we can lift $K^s\colon\fsmulti_{lr}\to\skmon$ to the braided and symmetric setting. We start with objects in the proposition below. 

\begin{prop}\label{prop:brd-equ-obj}
Let $\mc$ be a left representable short skew multicategory and $K^s\mc$ the corresponding skew monoidal category. A short braiding on $\mc$ induces a braiding on $K^s\mc$ in which
the braiding isomorphism $s\colon (xa)b\to (xb)a$ is the unique map such that 
\begin{equation}\label{eq:def-s-brd}
s\circ \theta_{xa,b}\circ_1\theta_{x,a}=\beta^3_2(\theta_{xb,a}\circ_1\theta_{x,b}).
\end{equation}
Moreover, if the braiding on $\mc$ is a short symmetry, then the braiding on $K^s\mc$ is a symmetry. 
\end{prop}

\begin{proof}
Clearly $s$ defined in this way is natural, so we have left to prove that the axioms for a braided skew monoidal category hold. Other than associativity equations for multimaps, the axioms follow from the following table.
\begin{center}
\renewcommand\arraystretch{1.25} 
\begin{tabular}{|c:c|}
\hline
Axiom & Follows from \\
\hline
\eqref{ax:sk-mon-brd-s} & \eqref{ax:sh-braid-4-ary}, \eqref{ax:sh-braid-3-in-2-first}, \eqref{ax:sh-braid-2-in-3-first} \\

\eqref{ax:sk-mon-brd-s-a-1} 
& \eqref{ax:sh-braid-3-in-2-first},\eqref{ax:sh-braid-2-in-3-first},\eqref{ax:sh-braid-2-in-3-secd}  \\

\eqref{ax:sk-mon-brd-s-a-2} 
& \eqref{ax:sh-braid-3-in-2-first}, \eqref{ax:sh-braid-2-in-3-first}, \eqref{ax:sh-braid-2-in-3-secd} \\

\eqref{ax:sk-mon-brd-a-s} 
& \eqref{ax:sh-braid-3-in-2-secd},\eqref{ax:sh-braid-2-in-3-first} \\
\hline
\end{tabular} 
\end{center}
All the proofs are quite similar, so we will explicitly show only \eqref{ax:sk-mon-brd-s-a-1}. As always, we will show that the two sides of the diagram are the same when we precompose with the universal tight 4-ary map. We start from $s_{x,a,bc}\circ\alpha_{xa,b,c}$:
\begin{center}
\begin{tikzpicture}[triangle/.style = {fill=yellow!50, regular polygon, regular polygon sides=3,rounded corners}]
\path
	(1,1.35) node [triangle,draw,shape border rotate=-90,inner sep=1pt] (b') {$\theta$} 
	(2.5,1) node [triangle,draw,shape border rotate=-90,inner sep=1pt] (c') {$\theta$}
	(4,0.5) node [triangle,draw,shape border rotate=-90,inner sep=1pt] (d) {$\theta$}
	(6,0.5) node [triangle,draw,shape border rotate=-90,inner sep=1pt] (f) {$\alpha$}
	(8,0.5) node [triangle,draw,shape border rotate=-90,inner sep=1pt] (f') {$s$};
	
\draw [dashed] (0.2,1.6) .. controls +(right:0.5cm) and +(left:0.5cm).. node[above] {$x$} (b'.140);
\draw [-] (0.2,1.1) .. controls +(right:0.5cm) and +(left:0.5cm).. node[below] {$a$} (b'.220);
\draw [-] (1.7,0.75) .. controls +(right:0.5cm) and +(left:0.5cm).. node[below] {$b$} (c'.220); 
\draw [dashed] (b') .. controls +(right:0.5cm) and +(left:0.5cm).. node[above,xshift=0.1cm] {$xa$} (c'.140);
\draw [dashed] (c') .. controls +(right:0.5cm) and +(left:0.5cm).. node[scale=.9,above,xshift=0.25cm] {$(xa)b$} (d.140);
\draw [-] (3.2,0.25).. controls +(right:0.5cm) and +(left:0.5cm).. node[below] {$c$} (d.219);
\draw [dashed] (d) to node[scale=.8,above] {$((xa)b)c$} (f);
\draw [dashed] (f) to node[scale=.8,above] {$(xa)(bc)$} (f');
\draw [-] (f') to node[scale=.8,above] {$(x(bc))a$} (9.75,0.5);
\end{tikzpicture} \\
by definition of $\alpha$, \\
\begin{tikzpicture}[triangle/.style = {fill=yellow!50, regular polygon, regular polygon sides=3,rounded corners}]

\path
(-0.5,0.75) node {$=$}
	(1,1) node [triangle,draw,shape border rotate=-90,inner sep=1pt] (b') {$\theta$} 
	(2,0) node [triangle,draw,shape border rotate=-90,inner sep=1pt] (c') {$\theta$}
	(3.25,0.5) node [triangle,draw,shape border rotate=-90,inner sep=1pt] (d) {$\theta$}
	(4.25,0.5) node [triangle,draw,shape border rotate=-90,inner sep=1pt] (f) {$s$};
	
\draw [dashed] (0.2,1.25) .. controls +(right:0.5cm) and +(left:0.5cm).. node[above] {$x$} (b'.140);
\draw [-] (0.2,0.75) .. controls +(right:0.5cm) and +(left:0.5cm).. node[below] {$a$} (b'.220);
\draw[dashed] (b') .. controls +(right:0.5cm) and +(left:0.5cm).. (d.140);
\draw [dashed] (1.2,0.25) .. controls +(right:0.75cm) and +(left:0.25cm).. node[above] {$b$} (c'.140);
\draw [-] (1.2,-0.25) .. controls +(right:0.5cm) and +(left:0.5cm).. node[below] {$c$} (c'.220); 
\draw [-] (c') .. controls +(right:0.5cm) and +(left:0.25cm)..(d.219);
\draw [dashed] (d) to (f);
\draw [-] (f) to (5,0.5); 
\end{tikzpicture} 
\begin{tikzpicture}[triangle/.style = {fill=yellow!50, regular polygon, regular polygon sides=3,rounded corners}]
\path
(-0.5,0.75) node {$=$}
	(2,1) node [triangle,draw,shape border rotate=-90,inner sep=1pt] (b') {$\theta$} 
	(1,0) node [triangle,draw,shape border rotate=-90,inner sep=1pt] (c') {$\theta$}
	(3.25,0.5) node [triangle,draw,shape border rotate=-90,inner sep=1pt] (d) {$\theta$}
	(4.25,0.5) node [triangle,draw,shape border rotate=-90,inner sep=1pt] (f) {$s$};
	
\draw [dashed] (1.2,1.25) .. controls +(right:0.5cm) and +(left:0.5cm).. node[above] {$x$} (b'.140);
\draw [-] (1.2,0.75) .. controls +(right:0.5cm) and +(left:0.5cm).. node[below] {$a$} (b'.220);
\draw[dashed] (b') .. controls +(right:0.5cm) and +(left:0.5cm).. (d.140);
\draw [dashed] (0.2,0.25) .. controls +(right:0.5cm) and +(left:0.25cm).. node[above] {$b$} (c'.140);
\draw [-] (0.2,-0.25) .. controls +(right:0.5cm) and +(left:0.5cm).. node[below] {$c$} (c'.220); 
\draw [-] (c') .. controls +(right:0.75cm) and +(left:0.25cm).. (d.219);
\draw [dashed] (d) to (f);
\draw [-] (f) to (5,0.5); 
\end{tikzpicture}
\end{center}
\begin{align*}
&= \beta^3_2(\theta_{x(bc),a}\circ_1\theta_{x,bc})\circ_3\theta_{b,c}
& (\textrm{by definition of}\,s) 
\\ 
&= \beta_2^4\beta_3^4(\,(\theta_{x(bc),a}\circ_1\theta_{x,bc})\circ_2\theta_{b,c}\,)
& (\textrm{by axiom}\,\eqref{ax:sh-braid-2-in-3-secd}\,).
\end{align*}
On the other hand, let us consider $\alpha_{x,b,c}a\circ s_{xb,a,c}\circ s_{x,a,b}c$: 
\begin{center}
\begin{tikzpicture}[triangle/.style = {fill=yellow!50, regular polygon, regular polygon sides=3,rounded corners}]
\path
	(1,1.35) node [triangle,draw,shape border rotate=-90,inner sep=1pt] (b') {$\theta$} 
	(2.5,1) node [triangle,draw,shape border rotate=-90,inner sep=1pt] (c') {$\theta$}
	(4,0.5) node [triangle,draw,shape border rotate=-90,inner sep=1pt] (d) {$\theta$}
	(6,0.5) node [triangle,draw,shape border rotate=-90,inner sep=-1.5pt] (f) {$s\cdot c$}
	(8,0.5) node [triangle,draw,shape border rotate=-90,inner sep=1pt] (f') {$s$}
	(10,0.5) node [triangle,draw,shape border rotate=-90,inner sep=-2.5pt] (f'') {$\alpha\cdot a$};
	
\draw [dashed] (0.2,1.6) .. controls +(right:0.5cm) and +(left:0.5cm).. node[above] {$x$} (b'.140);
\draw [-] (0.2,1.1) .. controls +(right:0.5cm) and +(left:0.5cm).. node[below] {$a$} (b'.220);
\draw [-] (1.7,0.75) .. controls +(right:0.5cm) and +(left:0.5cm).. node[below] {$b$} (c'.220); 
\draw [dashed] (b') .. controls +(right:0.5cm) and +(left:0.5cm).. node[above,xshift=0.1cm] {$xa$} (c'.140);
\draw [dashed] (c') .. controls +(right:0.5cm) and +(left:0.5cm).. node[scale=.9,above,xshift=0.25cm] {$(xa)b$} (d.140);
\draw [-] (3.2,0.25).. controls +(right:0.5cm) and +(left:0.5cm).. node[below] {$c$} (d.219);
\draw [dashed] (d) to node[scale=.8,above,xshift=-0.1cm] {$((xa)b)c$} (f);
\draw [dashed] (f) to node[scale=.8,above,xshift=-0.1cm] {$((xb)a)c$} (f');
\draw [dashed] (f') to node[scale=.8,above] {$((xb)c)a$} (f'');
\draw [-] (f'') to node[scale=.8,above] {$(x(bc))a$} (12,0.5);
\end{tikzpicture} \\
by definition of $s\cdot c$, 
\begin{tikzpicture}[triangle/.style = {fill=yellow!50, regular polygon, regular polygon sides=3,rounded corners}]
\path
	(1,1.35) node [triangle,draw,shape border rotate=-90,inner sep=1pt] (b') {$\theta$} 
	(2.5,1) node [triangle,draw,shape border rotate=-90,inner sep=1pt] (c') {$\theta$}
		(4,1) node [triangle,draw,shape border rotate=-90,inner sep=1pt] (f) {$s$}
	(6,0.5) node [triangle,draw,shape border rotate=-90,inner sep=1pt] (d) {$\theta$}
	(8,0.5) node [triangle,draw,shape border rotate=-90,inner sep=1pt] (f') {$s$}
	(10,0.5) node [triangle,draw,shape border rotate=-90,inner sep=-2.5pt] (f'') {$\alpha\cdot a$};
	
\draw [dashed] (0.2,1.6) .. controls +(right:0.5cm) and +(left:0.5cm).. node[above] {$x$} (b'.140);
\draw [-] (0.2,1.1) .. controls +(right:0.5cm) and +(left:0.5cm).. node[below] {$a$} (b'.220);
\draw [-] (1.7,0.75) .. controls +(right:0.5cm) and +(left:0.5cm).. node[below] {$b$} (c'.220); 
\draw [dashed] (b') .. controls +(right:0.5cm) and +(left:0.5cm).. node[above,xshift=0.1cm] {$xa$} (c'.140);
\draw [dashed] (c') .. controls +(right:0.5cm) and +(left:0.5cm).. node[scale=.9,above] {$(xa)b$} (f);
\draw [-] (5.2,0.25).. controls +(right:0.5cm) and +(left:0.5cm).. node[below] {$c$} (d.219);
\draw [dashed] (d) to node[scale=.8,above] {$((xb)a)c$} (f');
\draw [dashed] (f) to node[scale=.8,above] {$(xb)a$} (d.140);
\draw [dashed] (f') to node[scale=.8,above] {$((xb)c)a$} (f'');
\draw [-] (f'') to node[scale=.8,above] {$(x(bc))a$} (12,0.5);
\end{tikzpicture} 
\end{center}

\begin{align*}
&= \alpha_{x,b,c}a\circ s_{xb,a,c}\circ[\theta_{(xb)a,c}\circ_1 \beta_2^3(\theta_{xb,a}\circ_1\theta_{x,b})]
& (\textrm{by definition of}\,s) 
\\ 
&=  \alpha_{x,b,c}a\circ s_{xb,a,c}\circ\beta_2^4(\,\theta_{(xb)a,c}\circ_1(\theta_{xb,a}\circ_1\theta_{x,b})\,)
& (\textrm{by axiom}\,\eqref{ax:sh-braid-3-in-2-first}\,)
\\
&=  \beta_2^4[\,\alpha_{x,b,c}a\circ s_{xb,a,c}\circ(\theta_{(xb)a,c}\circ_1(\theta_{xb,a}\circ_1\theta_{x,b}))\,]
& (\textrm{by naturality of}\,\beta_2^4).
\end{align*}
Therefore, to conclude the proof of this axiom, it suffices to prove 
$$\beta_3^4(\,(\theta_{x(bc),a}\circ_1\theta_{x,bc})\circ_2\theta_{b,c}\,)=\alpha_{x,b,c}a\circ s_{xb,a,c}\circ(\theta_{(xb)a,c}\circ_1(\theta_{xb,a}\circ_1\theta_{x,b})).$$
Let us consider the right hand side:
%
\begin{align*}
&=  \alpha_{x,b,c}a\circ (\,(s_{xb,a,c}\circ\theta_{(xb)a,c}\circ_1\theta_{xb,a})\circ_1\theta_{x,b}\,)
& (\textrm{by naturality of}\,s)
\\
&=  \alpha_{x,b,c}a\circ (\,\beta_2^3(\theta_{(xb)c,a}\circ_1\theta_{xb,c})\circ_1\theta_{x,b}\,)
& (\textrm{by definition of}\,s,\,\eqref{eq:def-s-brd}\,)
\\
&=  \alpha_{x,b,c}a\circ \beta_3^4(\,(\theta_{(xb)c,a}\circ_1\theta_{xb,c})\circ_1\theta_{x,b}\,)
& (\textrm{by axiom}\,\eqref{ax:sh-braid-2-in-3-first}\,)
\\
&= \beta_3^4[\alpha_{x,b,c}a\circ(\,(\theta_{(xb)c,a}\circ_1\theta_{xb,c})\circ_1\theta_{x,b}\,)]
& (\textrm{by naturality of}\,\beta_3^4)
\\
&= \beta_3^4(\,(\theta_{x(bc),a}\circ_1\theta_{x,bc})\circ_2\theta_{b,c}\,)
& (\textrm{by definition of}\,\alpha a\,\textrm{and}\,\alpha).
\end{align*}

Finally, if the short braiding is a symmetry, i.e. if \eqref{ax:sh-symm} holds, then:
\begin{align*}
\theta_{xb,a}\circ_1\theta_{x,b}
&=\beta_2^3\beta^3_2(\theta_{xb,a}\circ_1\theta_{x,b})
& (\textrm{by }\,\eqref{ax:sh-symm}\,)
&
\\
&=\beta_2^3(s\circ \theta_{xa,b}\circ_1\theta_{x,a})
& (\textrm{by definition of}\,s,\,\eqref{eq:def-s-brd}\,).
& \qedhere
\end{align*}
\end{proof}

In the following proposition we will consider a braided short skew multicategory which is left representable and show how we can rewrite the isomorphisms $\beta^3_2,\beta^4_2$ and $\beta^4_3$ using $s$ and the left universal tight 3-ary maps. 

We recall that we write $(-)'$ for the inverse of $-\circ_1\theta$. In particular, for a tight 3-ary map $f\colon a,b,c\to d$ in $\mc$, $f''\colon (ab)c\to d$ is the unique tight unary map such that
\begin{equation}\label{eq:3-ary-lr}
f=f''\circ(\theta_{ab,c}\circ_1\theta_{a,b}).
\end{equation}
Similarly, for a tight 4-ary map $g\colon a,b,c,d\to e$ in $\mc$, $g''\colon(ab)c,d\to e$ is the unique tight binary map such that
\begin{equation}\label{eq:4-ary-lr}
g=g''\circ_1(\theta_{ab,c}\circ_1\theta_{a,b}).
\end{equation}

In a similar way to \cite[Proposition~A.4]{BouLack:skew-braid}, the following proposition provides a description of $\beta_2^4$ and $\beta_3^4$ in terms of $\beta_2^3$ (in the left representable case). 

\begin{prop}\label{prop:char-beta-lr}
Let $\mc$ be a braided short skew multicategory which is left representable. Using the notation in Proposition~\ref{prop:brd-equ-obj}:
\begin{enumerate}[(i)]
\item For any tight 3-ary map $f\colon a,b,c\to d$ in $\mc$, 
\begin{equation}\label{eq:beta-3-ary-lr}
\beta^3_2(f)=f''\circ s\circ(\theta_{ac,b}\circ_1\theta_{a,c}).
\end{equation}
\item For any tight 4-ary map $g\colon a,b,c,d\to e$ in $\mc$, 
\begin{center}
\begin{minipage}{0.4\textwidth}
\begin{equation}\label{eq:beta-4-2-ary-lr}
\beta^4_2(g)= g''\circ_1\beta^3_2(\theta_{ab,c}\circ_1\theta_{a,b})
\end{equation}
\end{minipage} \hspace{0.15cm}\raisebox{-3.5pt}{and}\hspace{0.35cm}
\begin{minipage}{0.4\textwidth}
\begin{equation}\label{eq:beta-4-3-ary-lr}
\beta^4_3(g)= \beta^3_2(g')\circ_1\theta_{a,b}.
\end{equation}
\end{minipage}
\end{center}
\end{enumerate}
\end{prop}

\begin{proof}
\begin{enumerate}[(i)]
\item[]
\item Let $f\colon a,b,c\to d$ be a tight 3-ary map in $\mc$. Then: 
\begin{align*}
\beta^3_2(f)&= \beta^3_2(\,f''\circ(\theta_{ab,c}\circ_1\theta_{a,b})\,)
& (\textrm{by left representability}\,\eqref{eq:3-ary-lr}\,) 
\\ 
&= f''\circ\beta^3_2(\theta_{ab,c}\circ_1\theta_{a,b})
& (\textrm{by naturality of}\,\beta^3_2)
 \\
&= f''\circ s\circ(\theta_{ac,b}\circ_1\theta_{a,c})
& (\textrm{by definition of}\,s\,\eqref{eq:def-s-brd}\,).
\end{align*}

\item Let $g\colon a,b,c,d\to e$ be a tight 4-ary map in $\mc$. Then:
\begin{align*}
\beta^4_2(g)&= \beta^4_2(\,g''\circ_1(\theta_{ab,c}\circ_1\theta_{a,b})\,)
& (\textrm{by left representability}\,\eqref{eq:4-ary-lr}\,) 
\\ 
&= g''\circ_1\beta^3_2(\,\theta_{ab,c}\circ_1\theta_{a,b}\,)
& (\textrm{by}\,\eqref{ax:sh-braid-3-in-2-first}\,).
\end{align*}
\begin{align*}
\beta^4_3(g)&= \beta^4_3(\,g'\circ_1\theta_{a,b}\,)
& (\textrm{by left representability}) 
&
\\ 
&= \beta^3_2(g')\circ_1\theta_{a,b}
& (\textrm{by}\,\eqref{ax:sh-braid-2-in-3-first}\,).
& \qedhere
\end{align*}
\end{enumerate}

\end{proof}

A useful consequence of Proposition~\ref{prop:char-beta-lr} is the characterisation of braided short skew multifunctors in $\brdsmulti_{lr}$ which we present in the Lemma below. 

\begin{notation*}
	Following the same idea as in Proposition~\ref{prop:universal}, for any left representable short skew multicategory $\mlc$, we may use the notation $\theta_{a,b,c}:=\theta_{ab,c}\circ_1\theta_{a,b}$ to make some calculations more readable.
	
\end{notation*}

\begin{lemma}\label{lem:char-brd-lr-multifct}
Let $\mc$ and $\md$ be two braided left representable short skew multicategories and $F\colon\mc\to\md$ a short skew multifunctor between them. If $F$ respects $\beta^3_2$, then $F$ respects also $\beta^4_2$ and $\beta^4_3$. 
\end{lemma}

\begin{proof}
In Proposition~\ref{prop:char-beta-lr} we showed how to write $\beta^4_2$ and $\beta^4_3$ in terms of $\beta^3_2$. 
Then, since $F$ respects $\beta^3_2$ and substitutions, we have that $F$ respects also $\beta^4_2$ and $\beta^4_3$. For instance, using \eqref{eq:beta-4-2-ary-lr},
\begin{equation}\label{eq:f4-beta42}
F_4^t(\beta^4_2(g))= F_4^t(\,g''\circ_1\beta^3_2(\theta_{a,b,c})\,)
=F_2^tg''\circ_1F_3^t(\beta^3_2(\theta_{a,b,c}))
=F_2^tg''\circ_1\beta^3_2(F_3^t\theta_{a,b,c}).
\end{equation}

On the other hand, one can check that 
\begin{center}
\begin{tikzpicture}[triangle/.style = {fill=yellow!50, regular polygon, regular polygon sides=3,rounded corners}]
\path
	(-.5,1) node [triangle,draw,shape border rotate=-90,inner sep=-2.5pt] (c') {$f_2Fc$}
	(2,1) node [triangle,draw,shape border rotate=-90,inner sep=1pt] (f) {$f_2$}
	(4.5,0.5) node [triangle,draw,shape border rotate=-90,inner sep=-3pt] (d) {$F_2^tg''$};
	
\draw node at (-4,0.5) {$(F_4^tg)''=$};
\draw [dashed] (-2.75,1) to node[scale=.7,above] {$(FaFb)Fc$} (c');
\draw [dashed] (f) .. controls +(right:1cm) and +(left:0.5cm).. node[scale=.7,above] {$F((ab)c)$} (d.140);
\draw [-] (3.2,0) to node[scale=.7,below] {$Fd$} (d.225);
\draw [dashed] (c') to node[scale=.7,above] {$F(ab)Fc$} (f);%
\draw [-] (d) to node [scale=.7,above] {$Fe$} (6,0.5);%
\end{tikzpicture}
\end{center}
and thus
\begin{align*}
\beta_2^4(F_4^tg)&=(F_4^tg)''\circ_1\beta_2^3(\theta_{Fa,Fb,Fc})
& 
\\
& =(F_2^tg\circ_1f_2\circ f_2Fc)\circ_1\beta_2^3(\theta_{Fa,Fb,Fc})
& 
\\ 
&= F_2^tg\circ_1\beta_2^3((f_2\circ f_2Fc)\circ\theta_{Fa,Fb,Fc})
& (\textrm{by naturality of}\,\beta^3_2)
\\
&= F_2^tg\circ_1\beta_2^3(F_3^t\theta_{a,b,c})
& (\textrm{by definition of}\,f_2)
\\
&= F_4^t(\beta_2^4(g))
& (\textrm{by equation}\,\eqref{eq:f4-beta42}\,).
\end{align*}
Similarly, we can prove that 
	$F_4^t(\beta_3^4(g))=\beta_2^3(F_3^tg')\circ_1 F_2^t\theta_{a,b}=\beta_3^4(F_4^t(g))$.
\end{proof}

For a left representable braided short skew multicategory $\mc$, we write $K^{brd}\mc$ for the skew monoidal category defined in Proposition~\ref{prop:brd-equ-obj}. The next proposition shows how to construct braided skew monoidal functors starting from morphisms in $\brdfsmulti_{lr}$.

\begin{prop}\label{prop:brd-equ-maps}
Let $F\colon\mc\to\md$ be a braided short skew multifunctor between two braided left representable short skew multicategories. Then, $K^sF\colon K^{brd}\mc\to K^{brd}\md$ is a braided skew monoidal functor. 
\end{prop}

\begin{notation*}
With abuse of notation we will use the same notation for the data in $\mlc$ and $\md$. 
If the data appears inside $F$ it is referred to $\mlc$ and instead if it is outside it comes from $\md$. 
\end{notation*}

\begin{proof}
We need to prove axiom \eqref{ax:brd-sk-mon-fct} for $K^sF$. To prove this we just need to precompose with the universal tight ternary maps and check that the equality still holds. Then, it follows directly by the definition of $s$ and that, since $F$ is a braided short skew multifunctor, 
$$F_3^t(\beta^3_2(\theta_{xb,a}\circ_1\theta_{x,b})\,)={\beta^3_2}(\,F_3^t(\theta_{xb,a}\circ_1\theta_{x,b})\,).$$
More precisely, we will prove axiom \eqref{ax:brd-sk-mon-fct} by precomposing with the universal tight 3-ary map classifier. We start from $Fs\circ f_2\circ f_2Fb$:
\begin{align*}
& Fs\circ f_2\circ f_2Fb\circ(\theta_{FxFa,Fb}\circ_1\theta_{Fx,Fa})
& 
\\
& =Fs\circ (\,f_2\circ\theta_{F(xa),Fb}\circ_1(f_2\theta_{Fx,Fa})\,)
& (\textrm{by definition of}\,f_2Fb)
\\ 
&= Fs\circ F_2^t(\theta_{xa,b})\circ_1F_2^t(\theta_{x,a})
& (\textrm{by definition of}\,f_2)
\\
&= F_3^t(\,s\circ(\theta_{xa,b})\circ_1\theta_{x,a}\,)
& (\textrm{since}\,F\,\textrm{respects substitution})
\\
&= F_3^t(\,\beta_2^3(\theta_{xb,a})\circ_1\theta_{x,b}\,)
& (\textrm{by definition of}\,s)
\\
&= \beta_2^3F_3^t(\,(\theta_{xb,a})\circ_1\theta_{x,b}\,)
& (\textrm{since}\,F\,\textrm{respects braidings})
\\
&= (f_2\circ f_2Fa)\circ s\circ(\theta_{FxFa,Fb}\circ_1\theta_{Fx,Fa})
& (\textrm{by}\,\eqref{eq:beta-3-ary-lr}\,).
\end{align*}
It is worth mentioning that in the last line we also used the fact that 
\begin{align*}
	F_3^t(\,(\theta_{xb,a})\circ_1\theta_{x,b}\,)''=f_2\circ f_2Fa. & \qedhere
\end{align*}
\end{proof}

Propositions \ref{prop:brd-equ-obj} and \ref{prop:brd-equ-maps} lift the equivalence $K^s\colon\fsmulti_{lr}\to\skmon$ to two functors
\begin{center}
$K^{brd}\colon\brdfsmulti_{lr}\to\brdskmon$ \\
$K^{sym}\colon\symfsmulti_{lr}\to\symskmon$.
\end{center}
Moreover, precomposing with the forgetful functors $U^{brd}_{lr}$ and $U^{sym}_{lr}$ we get two functors $T^{brd}$ and $T^{sym}$ as shown below.
\begin{center}
\begin{tikzcd}[ampersand replacement=\&]
	{\brdsmulti_{lr}} \\
	\& {\brdfsmulti_{lr}} \\
	\brdskmon
	\arrow[""{name=0, anchor=center, inner sep=0}, "{T^{brd}}"', from=1-1, to=3-1]
	\arrow["{U_{lr}^{brd}}", from=1-1, to=2-2]
	\arrow["{K^{brd}}", from=2-2, to=3-1]
	\arrow["{:=}"{description, pos=0.3}, draw=none, from=0, to=2-2]
\end{tikzcd} \hspace{1cm}
\begin{tikzcd}[ampersand replacement=\&]
	{\symsmulti_{lr}} \\
	\& {\symfsmulti_{lr}} \\
	\symskmon
	\arrow[""{name=0, anchor=center, inner sep=0}, "{T^{sym}}"', from=1-1, to=3-1]
	\arrow["{U_{lr}^{sym}}", from=1-1, to=2-2]
	\arrow["{K^{sym}}", from=2-2, to=3-1]
	\arrow["{:=}"{description, pos=0.3}, draw=none, from=0, to=2-2]
\end{tikzcd}
\end{center}

We can see straight away that $T^{brd}$ has, on objects, the same description given in \cite[Theorem~A.1]{BouLack:skew-braid}. In \cite[Theorem~5.9]{BouLack:skew-braid} they show that there is a bijective correspondence between braidings on a left representable skew multicategory and braidings on the corresponding skew monoidal category, which restricts to symmetries. \black From this it follows that $T^{brd}$ and $T^{sym}$ are essentially surjective. In fact, consider a braided skew monoidal category $\cc$ and $\mc$ the left representable skew multicategory such that $\alpha\colon\cc\xrightarrow{\sim}T^s\mc$ (which exists since $T^s$ is essentially surjective). Since $\cc$ is braided, by transport of structure, there exists a unique braiding on $T^s\mc$ such that $\alpha$ is an isomorphism of braided skew monoidal categories. Then, by \cite[Theorem~5.9]{BouLack:skew-braid}, there exists a unique braiding on $\mc$ such that $T^{brd}\mc=T^s\mc$ (as braided skew monoidal categories). Hence, $\cc\cong T^{brd}\mc$. 

\black
\begin{theorem}\label{thm:brd-sk-equiv}
The functors 
\begin{center}
$K^{brd}\colon\brdfsmulti_{lr} \to \brdskmon$ \hspace{0.5cm} and \hspace{0.5cm} $K^{sym}\colon\symfsmulti_{lr} \to \symskmon$
\end{center} 
are equivalences of categories. 
\end{theorem}

\begin{proof}
As discussed above, we already know that $T^{brd}$ and $T^{sym}$ are essentially surjective. Since $T^{brd}=K^{brd}U^{brd}_{lr}$ and $T^{sym}=K^{sym}U^{sym}_{lr}$, then also $K^{brd}$ and $K^{sym}$ are essentially surjective. 

Thus, we only need to show that they are fully faithful. Since $\symfsmulti_{lr}$ is a full subcategory of $\brdfsmulti_{lr}$, we only need to show that $K^{brd}$ is fully faithful to also get the same result for $K^{sym}$.

Let $H\colon K^{brd}\mc\to K^{brd}\md$ a braided skew monoidal functor. We need to find a braided short skew multifunctor $F\colon\mc\to\md$ such that $K^{brd}F=H$. Since the action of $K^{brd}$ on maps is the same as $K^s$, we define $F\colon\mc\to\md$ to be the unique short skew multifunctor such that $K^sF=H$, seeing $H$ as a lax monoidal functor. Finally, it suffices to check that $F$ is \emph{braided} whenever $K^{brd}F$ is such. By Lemma~\ref{lem:char-brd-lr-multifct} we just need to show that $F$ respects $\beta^3_2$, i.e. for any tight ternary map $f\colon a,b,c\to d\in\mc$,
$$F_3^t(\beta^3_2f)={\beta^3_2}(F_3^tf).$$
Let us start considering the left hand side. 
\begin{align*}
F_3^t(\beta^3_2f)
&= F_3^t(\,f''\circ s\circ(\theta_{ac,b}\circ_1\theta_{a,c})\,)
& (\textrm{by Proposition~\ref{prop:char-beta-lr}}) 
\\ 
&= Ff''\circ Fs\circ (\,F_2^t(\theta_{ac,b})\circ_1F_2^t(\theta_{a,c})\,)
& (\textrm{since F short skew multifunctor}),
\end{align*}
\begin{center}
by definition of $f_2=h_2$, is equal to \\
\begin{tikzpicture}[triangle/.style = {fill=yellow!50, regular polygon, regular polygon sides=3,rounded corners}]
\path
	(0,1) node [triangle,draw,shape border rotate=-90,inner sep=0pt] (c') {$\theta^\md$}
	(2,1) node [triangle,draw,shape border rotate=-90,inner sep=1pt] (f) {$f_2$}
	(4,0.5) node [triangle,draw,shape border rotate=-90,inner sep=0pt] (d) {$\theta^\md$}
	(6.25,0.5) node [triangle,draw,shape border rotate=-90,inner sep=1pt] (f') {$f_2$}
	(8.5,0.5) node [triangle,draw,shape border rotate=-90,inner sep=-1pt] (f'') {$Fs^\mc$}
	(11,0.5) node [triangle,draw,shape border rotate=-90,inner sep=-1pt] (g) {$Ff''$};
	
\draw [-] (-0.8,0.75) to node[scale=.7,below] {$Fc$} (c'.220);
\draw [dashed] (-0.8,1.25) to node[scale=.7,above] {$Fa$} (c'.140);
\draw [dashed] (f) .. controls +(right:1cm) and +(left:0.5cm).. node[scale=.7,above] {$F(ac)$} (d.140);
\draw [-] (3.2,0.15) to node[scale=.7,below] {$Fb$} (d.225);
\draw [dashed] (c') to node[scale=.7,above] {$FaFc$} (f);%
\draw [dashed] (d) to node [scale=.7,above] {$F(ac)Fb$} (f');%
\draw [dashed] (f') to node [scale=.7,above] {$F((ac)b)$} (f'');%
\draw [dashed] (f'') to node [scale=.7,above] {$F((ab)c)$} (g);
\draw [-] (g) to node [scale=.7,above] {$Fd$} (12.75,0.5);
\end{tikzpicture}\\
by definition of $f_2\cdot Fb$, 
\begin{tikzpicture}[triangle/.style = {fill=yellow!50, regular polygon, regular polygon sides=3,rounded corners}]
\path
	(0,1) node [triangle,draw,shape border rotate=-90,inner sep=0pt] (c') {$\theta^\md$}
	(2,0.5) node [triangle,draw,shape border rotate=-90,inner sep=0pt] (d) {$\theta^\md$}
	(4.5,0.5) node [triangle,draw,shape border rotate=-90,inner sep=-2.5pt,scale=.9] (f) {$f_2Fb$}
	(7,0.5) node [triangle,draw,shape border rotate=-90,inner sep=1pt] (f') {$f_2$}
	(9.25,0.5) node [triangle,draw,shape border rotate=-90,inner sep=-1pt] (f'') {$Fs^\mc$}
	(11.75,0.5) node [triangle,draw,shape border rotate=-90,inner sep=-1pt] (g) {$Ff''$};
	
\draw [-] (-0.8,0.75) to node[scale=.7,below] {$Fc$} (c'.220);
\draw [dashed] (-0.8,1.25) to node[scale=.7,above] {$Fa$} (c'.140);
\draw [dashed] (c') .. controls +(right:1cm) and +(left:0.5cm).. node[scale=.7,above,yshift=0.1cm] {$FaFc$} (d.140);
\draw [-] (1.2,0.15) to node[scale=.7,below] {$Fb$} (d.225);
\draw [dashed] (d) to node[scale=.7,above] {$(FaFc)Fb$} (f);%
\draw [dashed] (f) to node [scale=.7,above] {$F(ac)Fb$} (f');%
\draw [dashed] (f') to node [scale=.7,above] {$F((ac)b)$} (f'');%
\draw [dashed] (f'') to node [scale=.7,above] {$F((ab)c)$} (g);
\draw [-] (g) to node [scale=.7,above] {$Fd$} (13.25,0.5);
\end{tikzpicture}\\
since $K^{brd}F$ is braided, see \eqref{ax:brd-sk-mon-fct},\\
\begin{tikzpicture}[triangle/.style = {fill=yellow!50, regular polygon, regular polygon sides=3,rounded corners}]
\path
	(0,1) node [triangle,draw,shape border rotate=-90,inner sep=0pt] (c') {$\theta^\md$}
	(2,0.5) node [triangle,draw,shape border rotate=-90,inner sep=0pt] (d) {$\theta^\md$}
	(4.5,0.5) node [triangle,draw,shape border rotate=-90,inner sep=1pt] (f) {$s^\md$}
	(7,0.5) node [triangle,draw,shape border rotate=-90,inner sep=-2.5pt,scale=.9] (f') {$f_2Fc$}
	(9.4,0.5) node [triangle,draw,shape border rotate=-90,inner sep=1pt] (f'') {$f_2$}
	(11.85,0.5) node [triangle,draw,shape border rotate=-90,inner sep=-1pt] (g) {$Ff''$};
	
\draw [-] (-0.8,0.75) to node[scale=.7,below] {$Fc$} (c'.220);
\draw [dashed] (-0.8,1.25) to node[scale=.7,above] {$Fa$} (c'.140);
\draw [dashed] (c') .. controls +(right:1cm) and +(left:0.5cm).. node[scale=.7,above,yshift=0.1cm] {$FaFc$} (d.140);
\draw [-] (1.2,0.15) to node[scale=.7,below] {$Fb$} (d.225);
\draw [dashed] (d) to node[scale=.7,above] {$(FaFc)Fb$} (f);%
\draw [dashed] (f) to node [scale=.7,above] {$(FaFb)Fc$} (f');%
\draw [dashed] (f') to node [scale=.7,above] {$F(ab)Fc$} (f'');%
\draw [dashed] (f'') to node [scale=.7,above] {$F((ab)c)$} (g);
\draw [-] (g) to node [scale=.7,above] {$Fd$} (13.35,0.5);
\end{tikzpicture}
\end{center}
Thus, putting together these equalities, we get
\begin{equation}\label{eq:lhs-brd-ff}
F_3^t(\beta^3_2f)=(Ff''\circ f_2\circ f_2\cdot Fc\circ s)\circ(\theta_{F(ac),Fb}\circ_1\theta_{Fa,Fc}\,).
\end{equation}
For the right hand side, we start noticing that $F_3^tf$ can be written in terms of the left universal 3-ary map and two $f_2$. More precisely,
\begin{align*}
F_3^tf
&= F_3^t(f''\circ \theta_{a,b,c})= F_3^t(f''\circ \theta_{ab,c}\circ_1\theta_{a,b})
& (\textrm{by}\,\eqref{eq:3-ary-lr}\,\textrm{and left representability}) 
\\ 
&= Ff''\circ F_3^t(\theta_{ab,c}\circ_1\theta_{a,b})
& (\textrm{by naturality of}\,F_3^t)
\\
&= Ff''\circ F_2^t(\theta_{ab,c})\circ_1F_2^t(\theta_{a,b})
& (\textrm{since}\,F\,\textrm{respects substitutions}) 
\\
&= Ff''\circ (f_2\circ\theta_{F(ab),Fc})\circ_1(f_2\circ\theta_{Fa,Fb})
& (\textrm{by definition of}\,f_2)
\\
&= (Ff''\circ f_2\circ f_2\cdot Fc)\circ(\theta_{F(ab),Fc}\circ_1\theta_{Fa,Fb})
& (\textrm{by definition of}\,\cdot\,\textrm{on maps}).
\end{align*}
Hence, 
\begin{align*}
{\beta^3_2}(F_3^tf)
&={\beta^3_2}(\,(Ff''\circ f_2\circ f_2\cdot Fc)\circ(\theta_{F(ab),Fc}\circ_1\theta_{Fa,Fb})\,)
& (\textrm{by part above})
\\
&=(Ff''\circ f_2\circ f_2\cdot Fc)\circ{\beta^3_2}(\,(\theta_{F(ab),Fc}\circ_1\theta_{Fa,Fb})\,)
& (\textrm{by naturality of}\,\beta^3_2)
\\
&=(Ff''\circ f_2\circ f_2\cdot Fc)\circ s\circ(\theta_{F(ac),Fb}\circ_1\theta_{Fa,Fc})\,)
& (\textrm{by definition of}\,s)
\\
&=F_3^t(\beta^3_2f)
& (\textrm{by}\,\eqref{eq:lhs-brd-ff}).
\end{align*}
\end{proof}

\begin{cor}\label{cor:brd-bij-corr}
Let $\mlc$ be a left representable skew multicategory and $\mlc_s$ the corresponding short one (by  Theorem~\ref{thm:skew-left-closed-equiv}). Then, there is a bijection between braidings on $\mlc$ and short braidings on $\mlc_s$. This correspondence restricts to symmetries. 
\end{cor}

\begin{proof}
The part regarding braidings follows directly from Theorem~\ref{thm:brd-sk-equiv} and \cite[Theorem~5.5]{BouLack:skew-braid}. Again by \cite[Theorem~5.5]{BouLack:skew-braid} we know that a braiding on $\mlc$ is a symmetry if and only if the corresponding braiding on the associated skew monoidal category $\cc$ is a symmetry. More precisely, this is true if and only if $s_{x,b,a}=s_{x,a,b}^{-1}$ for any $x,a,b\in\cc$. Looking at the action of the equivalence $K_b$ in Theorem~\ref{thm:brd-sk-equiv} we see that $s_{x,b,a}=s_{x,a,b}^{-1}$ if and only if \eqref{ax:sh-symm}, i.e. if the braiding on $\mlc_s$ is a symmetry. 
\end{proof}

\begin{cor}
A braiding on a skew multicategory $\mc$ is a symmetry if and only if the isomorphism $\beta^3_2\colon\mc_3^t(a_1,a_2,a_3;b)\cong\mc_3^t(a_1,a_3,a_2;b)$ satisfy $\beta^3_2=(\beta^3_2)^{-1}$. 
\end{cor}

Since a short multicategory $\mc$ can be seen as short skew multicategory with $\mc_n^t(\overline{a};b)=\mc_n^l(\overline{a};b)$, from the results above we can derive similar results for short multicategories. 

\begin{theorem}\label{thm:brd-equiv}
The equivalences $K^{brd}$ and $K^{sym}$ of Theorem~\ref{thm:brd-sk-equiv} induce equivalences
\begin{center}
$K^{brd}\colon\brdfmulti_{lr} \to \brdmon$ \hspace{0.5cm} and \hspace{0.5cm} $K^{sym}\colon\symfmulti_{lr} \to \symmon$.
\end{center} 
\end{theorem}

\begin{proof}
This follows from Theorem~\ref{thm:brd-sk-equiv} and the fact that left normal braided skew monoidal categories are actually monoidal \cite[Proposition~2.12]{BouLack:skew-braid}. 

Furthermore, one can check that skew braided functors between braided monoidal categories are exactly lax monoidal functors. 
\end{proof}

\black

\black
\subsection{Biclosedness}

We conclude by underlying an interesting phenomenon which appears when we have a biclosed structure. Precisely, whenever we have a biclosed structure, then the substitutions can be described using only the biclosed isomorphisms and the functoriality of $n$-ary maps. We start considering \emph{biclosed} multicategories.

What we described in Definition~\ref{def:multi-closed} is usually called a \emph{left closed} multicategory. Dually, one gets the definition of \emph{right closed} multicategory and therefore a \emph{biclosed} one. 

\begin{defn}
\label{def:fin-biclosed}
\begin{enumerate}
\item[]
\item A multicategory is said to be \emph{right closed} if for all $b,c$ there exists an object $r[b,c]$ and binary map $e^r_{b,c}\colon(b,r[b,c]) \to c$ for which the induced function, for any $n\in\mathbb{N}$, $e^r_{b,c}\circ_2-\colon\catc_n(\overline{x};r[b,c]) \to \catc_{n+1}(b,\overline{x};c)$ is an isomorphism.  

\item A multicategory is said to be \emph{biclosed} if it is both left and right closed. 
\end{enumerate}
\end{defn}

Clearly the definitions above can be given also in the short case by limiting $n=1,2,3$. 

\begin{notation*}
In a biclosed multicategory $\catc$ we will denote with $L_n$ the isomorphisms $\catc_n(\bar{a};b)\cong\catc_{n-1}(a_1,...,a_{n-1};l[a_n,b])$ and with $R_n$ the isomorphisms $\catc_n(\bar{a};b)\cong\catc_{n-1}(a_2,...,a_{n};r[a_1,b])$. Therefore, with this notation $L_n^{-1}=e^l_{a,b}\circ_1-$ and $R_n^{-1}=e^r_{a,b}\circ_2-$.

It is interesting to notice that, seeing biclosed multicategories as semantics for the simply-typed $\lambda$-calculus without products (see for instance \cite{Sav:clones-closd-cats-comb-logic,Lambek-multicategories}), these isomorphisms correspond to \emph{currying} (the $L$'s) and \emph{$\lambda$-abstraction} (the $R$'s).
\black 
\end{notation*}

In the next proposition we prove that substitution of binary maps in a biclosed multicategory can be described in terms of the isomorphisms $L$ and $R$, and the functoriality of $\mc_2(-;-)$. 

\begin{prop}
\label{prop:biclosed-subst}
Let $\catc$ be a biclosed multicategory. The following equations are true:
\begin{enumerate}[(i)]
\item for any pair of binary maps $f\colon a_1,a_2\to b_1$ and $g\colon b_1,b_2\to c$, 
$$g\circ_1f=L_3^{-1}(L_2g\circ f);$$
\item for any pair of binary maps $f\colon a_1,a_2\to b_2$ and $g\colon b_1,b_2\to c$, 
$$g\circ_2f=R_3^{-1}(R_2g\circ f).$$
\end{enumerate}
\end{prop}

\begin{proof}
We show only the first case (i), since the other one can be proved in an analogous way. The equation we need to prove is equivalent to proving $L_3(g\circ_1f)=L_2g\circ f$. By definition of $L_2$ we know that $L_2g$ is the unique unary map such that $e^l\circ_1L_2g=g$. 
Similarly, $L_3(g\circ_1f)$ is the unique binary map such that $e^l\circ_1L_3(g\circ_1f)=g\circ_1f$. 
Therefore, it suffices to prove that 
$$e^l\circ_1(L_2g\circ f)=g\circ_1f.$$
This can be proven using dinaturality
 of substitution of binary into binary and the definition of $L_2g$. 
\end{proof}
 
More generally, given the biclosed isomorphisms $L_n$ and $R_n$ one could check that substitution of $n$-ary multimaps into $m$-ary must be defined as
\[\begin{tikzcd}
	{\mc_m(\bar{b};c)\times\mc_n(\bar{a};b_i)} \\
	{\mc_1(b_i;r[b_{i-1},...r[b_1,l[b_{i+1},...l[b_m,c]...]]])\times\mc_n(\bar{a};b_i)} \\
	{\mc_n(\bar{a};r[b_{i-1},...r[b_1,l[b_{i+1},...l[b_m,c]...]]])} \\
	{\mc_{n+m-1}(b_1,...b_{i-1},\bar{a},b_{i+1},...,b_n;c)}
	\arrow["\Phi", from=1-1, to=2-1]
	\arrow["p", from=2-1, to=3-1]
	\arrow["\Psi", from=3-1, to=4-1]
\end{tikzcd}\]
where $\Phi$ is a combination of $L_j$ and $R_k$, $p$ is the profunctor action and $\Psi$ a combination of $L\inv_j$ and $R^{-1}_k$. We shall not indulge in these calculations here since it drifts away from the motivation of this subsection. In fact, we want to underline how one can construct potential substitutions starting from the biclosedness isomorphism. For instance, this strategy was used in \cite{BourkeLob:SkewApp} for the multicategory of loose multimaps between Gray-categories. 

\begin{rmk}
One could define a structure consisting of a category $\cc$ equipped with, for $0\leq n\leq 4$, $n$-ary functors $\mc_n(-;-)$, natural isomorphisms $L_n$ and $R_n$ subject to axioms involving $L_n$ and $R_n$, and prove that this structure is equivalent to a biclosed multicategory. We avoid doing this here because the axioms corresponding to the various instances of the associativity equations (\ref{eq:ass-line}, \ref{eq:ass-not-line}) become more complicated than the original axioms. For example, for binary maps $f\colon a_1,a_2\to x_1$, $g\colon b_1,b_2\to x_2$ and $h\colon x_1,x_2\to y$, the associativity equation $(h\circ_1f)\circ_3g=(h\circ_2g)\circ_1f$ becomes
$$R_4^{-1}R_3^{-1}[\,R_2R_3L_3^{-1}(L_2g\circ f_1)\circ f_2\,]=L_4^{-1}L_3^{-1}[\,L_2L_3R_3^{-1}(R_2g\circ f_2)\circ f_1\,].$$
For this reason, we believe that in the biclosed scenario, the best strategy is to use the isomorphisms $L_n$ and $R_n$ to define substitutions and then prove the axioms in their original form (\ref{eq:ass-line}, \ref{eq:ass-not-line}). 
\end{rmk}

\appendix

\black
\section{The Closed Case}\label{sec:closed-case}

Instead of considering left representability as starting point, we can also consider closedness. Indeed, there is an equivalence between closed skew multicategories with unit and skew closed categories \cite[Theorem~6.6]{LackBourke:skew}. 
In this Section, we will show that similar calculations to the left representable case, but using the evaluations morphisms instead of the multimaps classifiers, lead to an equivalence 
$$K_{cl}^s\colon\fsclosed\to\skcl,$$ 
between closed short skew multicategories with units and skew closed categories. Moreover, this equivalence restricts to one $\fclosed\to\closed$ between closed short multicategories and closed categories. Since most of the example in the literature are also left representable, we left the treatment of this particular case to the appendix and, in the proofs, we give only the details of some axioms (even though we believe they are enough to give the idea for the remaining axioms). It is worth mentioning that both in \cite{LackBourke:skew} and here only \emph{left} closedness is considered. Analogous results would follow considering right closedness.

\subsection{Skew Closed Categories}
We now recall the definition of the category $\skcl$ of skew closed categories and skew closed functors \cite[Section~2]{Street:skew-closed}. A \textbf{(left) skew-closed category} $(\cc,[-,-],i,I,J,L)$ consists of a category $\catc$ equipped with a functor 
$$[-,-]\colon\catc^\op\times\catc\to\catc$$ 
and object $i$, together with natural transformations $I\colon[i,a]\to a$ (right unit), $J\colon i\to[a,a]$ (left unit) and $L\colon [b,c]\to[\,[a,b],[a,c]\,]$ (associator) subject to five axioms \cite[Axioms~2.1~--~2.5]{Street:skew-closed}. Then, a functor $F\colon\catc\to\catd$ is defined to be \textbf{closed} when it is equipped with a morphism $f_0\colon i^\cd\to Fi^\cc$ and a natural transformation $f_{a,b}\colon F[a,b]\to[Fa,Fb]$ satisfying three axioms \cite[Axioms~2.6~--~2.8]{Street:skew-closed}. 

We denote with $\smulticlosed$ the full subcategory of $\smulti$ with objects closed skew multicategories (see Definition~\ref{def:closed-skew-multi}) with a nullary map classifier, also called \emph{closed skew multicategories with units}. In \cite[Theorem~6.6]{LackBourke:skew} it is proven that there is an equivalence $T_{cl}^s\colon\smulticlosed\to\skcl$, which extends the equivalence  $\multicl\to\closed$ given in \cite{Manzyuk2012Closed}.

\subsection{Closed Short Skew Multicategories}

We denote with $\fsclosed$ the full subcategory of $\fsmulti$ with objects closed short skew multicategories (Definition~\ref{def:fin-closed}) with a nullary map classifier. Naturally, the forgetful functor $U\colon\smulti\to\fsclosed$ restricts to a forgetful functor
\begin{center}
$U_{cl}^s\colon\smulticlosed\to\fsclosed$.
\end{center} 

We underline that closed (short) skew multicategories with units do not require the nullary maps classifier to be left universal. Nevertheless, in any closed (short) skew multicategory with units, the nullary map classifier is always left universal. This is a consequence of closedness, as we show in the following proposition, which essentially follows from the proof of \cite[Proposition~4.8]{LackBourke:skew}.

\begin{prop}\label{prop:closed-units-then-left-uni}
Let $\mc$ be a closed short skew multicategory. If $\mc$ admits a nullary map classifier $u\colon\diamond\to i$, then $i$ is left universal. 
\end{prop}

\begin{proof}
We proceed inductively. By definition, $\mc_1^t(i;y)\cong\mc_0^l(\diamond;y)$. Then, for $n=2,3$,  
$$\mc_n^t(i,\overline{a},x;y)\cong\mc_{n-1}^t(i,\overline{a};[x,y])\cong\mc_{n-2}^l(\overline{a};[x,y])\cong\mc_{n-1}^l(\overline{a},x;y),$$
where the first isomorphism is given by closedness, the second by induction (and it must be $-\circ_1u$) and the last one by closedness again. One can check using the definition of the closedness isomorphism that the resulting isomorphism is always given by $-\circ_1u$.
\end{proof}

\begin{notation*}
Let $\mc$ be a closed (short) skew multicategory. For any multimap $f\in\mc_{n+1}^x(\overline{a},b;c)$ we denote with $f^\sharp\in\mc_{n}^x(\overline{a};[b,c])$ the unique map corresponding to $f$ via the closedness isomorphism $\mc_{n+1}^x(\overline{a},b;c)\cong\mc_{n}^x(\overline{a};[b,c])$. 

We also recall that, for any nullary map $v\colon\diamond\to a$ in $\mc$, we write $v^\ast\colon i\to a$ for the tight unary map determined by the universal map classifier. 
\end{notation*}

\begin{rmk}\label{rmk:closed-iso-props}
Similarly to what happens in the left representable case, also the operation $(-)^\sharp$ \emph{respects substitution}. More precisely, we can see that the following equations are true by postcomposing with the universal maps $e_{b,c}$ given by the closed structure. 
\begin{itemize}
\item For any binary map $f\colon a,b\to c$ and nullary map $v\colon\diamond\to a$, then $f^\sharp\circ v=(f\circ_1v)^\sharp.$
\item For any binary map $f\colon a,b\to x$ and multimap $g\colon x,\overline{c}\to y$, then 
$(g\circ_1f)^\sharp=g^\sharp\circ_1 f.$
\item For any binary map $f\colon a,b\to x$ and any unary map $q\colon a'\to a$, then 
$(f^\sharp\circ q)^\sharp=[1,f^\sharp]\circ q^\sharp.$
\end{itemize}
It is interesting to notice that these equaitions correspond to standard equalities regarding currying in the $\lambda$-calculus (see for instance \cite[Sections~1.2,~1.4]{hottbook} or \cite{Sav:clones-closd-cats-comb-logic,Lambek-multicategories}). 
\black  
\end{rmk}

As in Section~\ref{subsec:left-repr-case}, we start by proving a characterisation of morphisms between closed short skew multicategories in the following lemma. This is analogous to Lemma~\ref{lemma:char-sk-morph-left-repr}.

\begin{lemma}\label{lemma:char-sk-morph-closed}
Let $\mc$ and $\md$ be two closed short skew multicategories. A morphism $F\colon\mc\to\md$ is uniquely specified by:
\begin{itemize}
\item A functor $F\colon\cc\to\cd$ (where $\cc(x,y):=\mc_1^t(x;y)$ and $\cd(a,b):=\md_1^t(a;b)$);

\item Natural families $F^l_0\colon\mlc^l_0(\diamond;a)\to\md^l_0(\diamond;Fa)$ and $F^t_2\colon\mlc^t_2(a,b;c)\to\md^t_2(Fa,Fb;Fc)$ such that $F$ commutes with   
\begin{equation}
\label{eq:lemma-sk-morph-nullary-cl}
\begin{gathered}
%
%
%
%
\begin{tikzpicture}[triangle/.style = {fill=yellow!50, regular polygon, regular polygon sides=3,rounded corners}]
\path
	(2,.4) node [triangle,draw,shape border rotate=-90,inner sep=1pt] (b) {$v$} 
	(4,1) node [triangle,draw,shape border rotate=-90,label=135:$a$,label=230:$b$] (ab) {$f$};

\draw [dashed] (3.15,1.4) to (ab.140);
\draw [-] (b) .. controls +(right:1cm) and +(left:1cm).. (ab.227);
	
\draw [-] (ab) to node [above] {$c$} (5.5,1);

\end{tikzpicture} 
\end{gathered}
\end{equation}
%
%
%
%
and such that if we define, for any ternary tight map $h\in\mlc^t_3(a_1,a_2,a_3;b)$, $F_3^th:=F^t_2e_{a_3,b}\circ_1F^t_2h^\sharp$, then $F$ also commutes with 
\begin{equation}
\label{eq:lemma-sk-morph-bin-cl}
\begin{gathered}
\begin{tikzpicture}[triangle/.style = {fill=yellow!50, regular polygon, regular polygon sides=3,rounded corners}]
\path 
	(2,0) node [triangle,draw,shape border rotate=-90,inner sep=0pt,label=135:$a$,label=230:$b$] (a) {$f$}
	(4,0.5) node [triangle,draw,shape border rotate=-90,label=135:$x$,label=230:$y$] (c) {$g$};

\draw [dashed] (3.1,0.8) to (c.139);
\draw [-] (a) .. controls +(right:1cm) and +(left:1cm).. (c.220);
\draw [dashed] (1.2,.25) to (a.140);
\draw [-] (c) to node [above] {$c$} (5,0.5);
\draw [-] (1.2,-0.25) to (a.225);
\end{tikzpicture} 
\end{gathered}
\end{equation}
and such that if we define, for any loose unary map $q$, $F_1^lq:=F_2^te_{Fa,Fa'}\circ_1F_0^l\overline{q}$, then, for any tight unary map $p$,
\begin{equation}\label{eq:lemma-sk-morph-j-cl}
F_1^lj(p)=jF_1^tp.
\end{equation}
\end{itemize}
\end{lemma}

\begin{proof}
The proof is analogous to the one of Lemma~\ref{lemma:char-sk-morph-left-repr} using the properties of the closed universal maps $e_{b,c}$. The only difference lies in the definition of the natural families $F_4^t$ and $F_2^l$ (and also $F_3^t$ and $F_1^l$ which are defined in the statement of the Lemma). 
\begin{itemize}
\item For $k\in\mlc_4^t(a,b,c,d;e)$ we define $F_4^tk:=F_2^te_{d,e}\circ_1F_3^tk^\sharp$,
\item For $r\in\mlc_2^l(a,b;c)$ we define $F_2^lr:=F_2^te_{b,c}\circ_1F_1^lr^\sharp$. 
\end{itemize}
Exactly as in Lemma~\ref{lemma:char-sk-morph-left-repr}, one can prove directly from the definitions and naturality of the families $F_i^x$, and the properties of closedness (including Remark~\ref{rmk:closed-iso-props}), that these families commute with all substitutions. Below we describe what we need for each substitution, on top of the associativity equations in $\mc$ and $\md$. 
\begin{enumerate}[(i)]
	\item\label{tb-tb} Tight binary into tight binary: naturality of $F_2^t$ and \eqref{eq:lemma-sk-morph-bin-cl}.
	
	\item\label{tb-tt} Tight binary into ternary: \ref{tb-tb} and Remark~\ref{rmk:closed-iso-props}.
	
	\item\label{tt-tb} Tight ternary into tight binary: naturality of $F_2^t$, Remark~\ref{rmk:closed-iso-props}, \ref{tb-tb}, \ref{tb-tt} and~\ref{eq:lemma-sk-morph-bin-cl}.
	
	\item\label{n-tb} Nullary into tight binary: Remark~\ref{rmk:closed-iso-props}, naturality of $F_2^t$ and $F_0^l$.
		
	\item\label{n-tt} Nullary into tight ternary ($2^{nd}$ and $3^{rd}$ variable): \ref{tb-tb}, \ref{n-tb},  and naturality $F_2^t$.
	
		\item\label{lu-tb} Loose unary into tight binary: Remark~\ref{rmk:closed-iso-props}, naturality of $F_2^t$ and $F_0^l$, \eqref{eq:lemma-sk-morph-nullary-cl} and~\ref{n-tt}.
	
	\item\label{n-1-tt} Nullary into tight ternary ($1^{st}$ variable): \eqref{eq:lemma-sk-morph-nullary-cl}, \ref{tb-tb}, \ref{n-tb} and \ref{lu-tb}.
	
	\item\label{n-lu} Nullary into loose unary: \ref{n-tb} and naturality of $F_0^l$.
	
	\item\label{tb-lu} Tight binary into loose unary: \ref{tb-tb} and \ref{n-tt}. \qedhere
\end{enumerate}
\end{proof}

Following the same strategy we used in Section~\ref{sec:short-skew-mult}, the aim is once again to get a picture as below, where $T_{cl}^s$ is the equivalence given in \cite[Theorem~6.6]{LackBourke:skew} and $U_{cl}^s$ the natural forgetful functor.
\begin{equation}\label{diag-eqv-closed}
\begin{tikzcd}[ampersand replacement=\&]
	\smulticlosed \\
	\& {\fsclosed} \\
	\skcl
	\arrow["{U_{cl}^s}", from=1-1, to=2-2]
	\arrow["{T_{cl}^s}"', from=1-1, to=3-1]
	\arrow["{K_{cl}^s}", dashed, from=2-2, to=3-1]
\end{tikzcd}
\end{equation}

We start assigning to any closed short skew multicategory with units $\mc$ a skew closed category $K_{cl}^s\mc$.  

\begin{lemma}
\label{lemma:closed-skew-fin-mult-to-skew-closed}
Given a closed short skew multicategory $\mlc$ with units we can construct a skew closed category $K_{cl}^s\mlc$   in which:
\begin{itemize}
\item The hom object $[b,c]$ for a pair of objects $b$ and $c$ is defined by the hom object of the closed short skew multicategory;

\item Given a tight unary map $f\colon b\to b'$, the tight map $[f,c]$  is defined as the unique map such that 
$$e_{b,c}\circ_1[f,c]=e_{b',c}\circ_2f;$$

\item Given a tight unary map $g\colon c\to c'$, the tight map $[b,g]$ is defined as the unique map such that 
$$e_{b,c'}\circ_1[b,g]=g\circ e_{b,c};$$

\item The unit $i$ is given by the nullary map classifier;

\item The right unit $I\colon [i,a]\to a$ is defined as
$$I:=e_{i,a}\circ_2u;$$

\item The left unit $J\colon i\to [a,a]$ is defined as the unique map such that \footnote{Here we use the left representability of $i$ (see Proposition~\ref{prop:closed-units-then-left-uni}) and the resulting chain of isomorphisms $\mc_1^t(i;[a,a])\cong\mc_2^t(i,a;a)\cong\mc_1^l(a;a)$.}
$$(e_{a,a}\circ_1J)\circ_1 u= j(1_a);$$

\item The associator $L\colon[b,c]\to[\,[a,b],[a,c]\,]$ is defined as the unique map such that 

\begin{center}
\begin{tikzpicture}[triangle/.style = {fill=yellow!50, regular polygon, regular polygon sides=3,rounded corners},baseline={([yshift=-.5ex]current bounding box.center)}]
\path
	(1.5,1.5) node [triangle,draw,shape border rotate=-90,inner sep=2pt] (b') {$L$} 
	(4.5,1) node [triangle,draw,shape border rotate=-90,inner sep=2pt] (c') {$e$}
	(6.5,0.5) node [triangle,draw,shape border rotate=-90,inner sep=2pt,label=230:$b$] (d) {$e$};
	%
	\draw [dashed] (0.25,1.5) to node [above] {$[b,c]$} (b');
\draw [-] (3.25,0.75) to node [below] {$[a,b]$} (c'.220);
\draw [dashed] (b') to node [above] {$[\,[a,b],[a,c]\,]$} (c'.140);
\draw [dashed] (c') to node [above] {$[b,c]$} (d.140);
\draw [-] (5.5,0.25) to (d.225);
\draw [-] (d) to node [above] {$c$} (7.5,0.5);
\end{tikzpicture} \hspace{0.2cm} $=$ \hspace{0.2cm}
\begin{tikzpicture}[triangle/.style = {fill=yellow!50, regular polygon, regular polygon sides=3,rounded corners},baseline={([yshift=-.5ex]current bounding box.center)}]
\path 
	(2,0) node [triangle,draw,shape border rotate=-90,inner sep=2pt,label=230:$a$] (a) {$e$}
	(4,0.5) node [triangle,draw,shape border rotate=-90,inner sep=2pt,label=230:$b$] (c) {$e$};

\draw [dashed] (3.1,0.8) to node [above] {$[b,c]$} (c.139);
\draw [-] (a) .. controls +(right:1cm) and +(left:1cm).. (c.220);
\draw [dashed] (1,.25) to node [above] {$[a,b]$} (a.140);
\draw [-] (c) to node [above] {$c$} (5,0.5);
\draw [-] (1,-0.25) to (a.225);
\end{tikzpicture}
\end{center}

\end{itemize}
\end{lemma}

\begin{proof}
The functoriality of $[-,-]\colon\cc^{op}\times\cc\to\cc$ follows from the universal property of the hom $[b,c]$ and profunctoriality of $\mc_2^t(-;-)$. It remains to verify the five axioms for a skew closed category. All of them follow checking that the equalities hold after postcomposing with the universal maps $e$'s given by the closed structure on $\mc$. We will show explicitly the calculations for the axiom involving $J$ and $L$ shown below (see \cite[Axioms~2.3]{Street:skew-closed}). 
\[\begin{tikzcd}[ampersand replacement=\&]
	i \& {[b,b]} \\
	\& {[\,[a,b],[a,b]\,]}
	\arrow["{J_b}", from=1-1, to=1-2]
	\arrow["{J_{[a,b]}}"', from=1-1, to=2-2]
	\arrow["{L^a_{b.b}}", from=1-2, to=2-2]
\end{tikzcd}\]
Using closedness and left representability of $i$, it is enough to prove the equality postcomposing with $e_{[a,b],[a,b]}$ and $e_{a,b}$, and precomposing with $u$. So, 
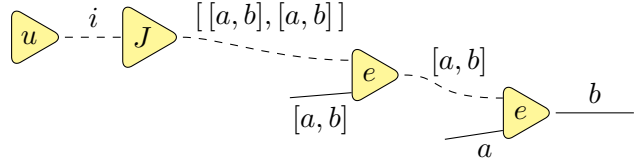
\begin{center}
\begin{tikzpicture}[triangle/.style = {fill=yellow!50, regular polygon, regular polygon sides=3,rounded corners}]
\path
	(1,1.5) node [triangle,draw,shape border rotate=-90,inner sep=1.5pt] (b') {$u$} 
	(2.5,1.5) node [triangle,draw,shape border rotate=-90,inner sep=1.5pt] (c') {$J$}
	(5.5,1) node [triangle,draw,shape border rotate=-90,inner sep=2pt] (f) {$e$}
	(7.5,0.5) node [triangle,draw,shape border rotate=-90,inner sep=2pt,label=230:$a$] (d) {$e$};
	
\draw [dashed] (b') to node [above] {$i$} (c');
\draw [dashed] (c') .. controls +(right:1cm) and +(left:1cm).. node [above,xshift=0.25cm,yshift=0.1cm] {$[\,[a,b],[a,b]\,]$} (f.140);
\draw [-] (4.45,0.7) to node [below] {$[a,b]$} (f.225);
\draw [-] (6.5,0.15) to (d.225);
\draw [dashed] (f) .. controls +(right:1cm) and +(left:1cm).. node [above,xshift=0.25cm] {$[a,b]$} (d.140);
\draw [-] (d) to node [above] {$b$} (9,0.5);
\end{tikzpicture}
\begin{align*}
&[\,e_{a,b}\circ_1(e_{[a,b],[a,b]}\circ_1J_{[a,b]})\,]\circ_1u
& \\
&= e_{a,b}\circ_1[\,(e_{[a,b],[a,b]}\circ_1J_{[a,b]})\circ_1u\,]
& (\textrm{by axiom (\ref{eq:ass-line}.b)}) 
\\ 
&= e_{a,b}\circ_1 j(1_{[a,b]})
& (\textrm{by definition of}\,J).
\\
&= j(e_{a,b}\circ_1 1_{[a,b]})=j(e_{a,b})
& (\textrm{by naturality of}\;j\;\textrm{and unit axiom}).
\end{align*}
\end{center}
Similarly one proves that also 
$$[\,e_{a,b}\circ_1(\,e_{[a,b],[a,b]}\circ_1(L_{b,b}^a\circ J_b)\,)\,]\circ_1u=e_{a,b}.$$
\end{proof}

Before defining the functor $K_{cl}^s\colon\fsclosed\to\skcl$ on morphisms, we prove the following useful lemma, which corresponds to Lemma~\ref{lemma:sk-bij} for the left representable case.

\begin{lemma}\label{lemma:sk-bij-cl}
Consider $\mc, \md \in \fsclosed$ and a functor $F\colon \cc \to \cd$.  There is a bijection between natural families 
\begin{center}
$F_0^l\colon\mc_0^l(\diamond;a)\to\md_0^l(\diamond;Fa)$ and $F_2^t\colon\mc_0^l(a,b;c)\to\md_0^l(Fa,Fb;Fc)$,
\end{center}
and natural families of morphisms
\begin{center}
$f_0\colon i^\md\to Fi^\mc$ and $f_{a,b}\colon F[a,b]\to[Fa,Fb]$.
\end{center}

\end{lemma}

\begin{proof}
The correspondence between $F_0^l$ and $f_0$ is the same as in Lemma~\ref{lemma:sk-bij}. 

Similarly to Lemma~\ref{lemma:sk-bij}, the correspondence between $F_2^t$ and $f_{a,b}$ is governed by the following diagram 
\[\begin{tikzcd}[ampersand replacement=\&]
	{\mc_1^t(x;[a,b])} \&\& {\md_1^t(Fx;[Fa,Fb])} \\
	{\mc_2^t(x,a;b)} \&\& {\md_2^t(Fx,Fa;Fb)}
	\arrow["{F(f_{a,b}\circ-)}", from=1-1, to=1-3]
	\arrow["\cong","{e_{a,b}\circ_1-}"', from=1-1, to=2-1]
	\arrow["{e_{Fa,Fb}\circ_1-}","\cong"', from=1-3, to=2-3]
	\arrow["{F_2^t}"', from=2-1, to=2-3]
\end{tikzcd}\]
in which the vertical arrows are natural bijections by closedness and the lower horizontal arrow corresponds to the upper one using the Yoneda lemma. 
\end{proof}

\begin{rmk}
\label{rmk:def-K-cl-on-funct}
Given a morphism $F\colon \mlc \to \md \in \fsclosed$ we obtain, applying the above lemma, natural families $f_{a,b}\colon F[a,b] \to [Fa,Fb]$ and $f_0\colon i \to Fi$ defining the \emph{data} for a closed functor $K_{cl}^sF\colon K_{cl}^s\mlc \to K_{cl}^s\md$. We will prove that it is a closed functor in Proposition~\ref{prop:full-faith-cl}.

Explicitly, $f_{a,b}\colon F[a,b] \to [Fa,Fb]$ is the unique morphism such that $e_{Fa,Fb}\circ_1f_{2}=F_2^t(e_{a,b})$ whilst $f_{0}$ is the unique morphism such that $f_{0} \circ u^\md = F_0^lu^\mc$. 
\end{rmk}

In a similar way to Proposition~\ref{prop:sk-full-faith}, now we use Lemma~\ref{lemma:char-sk-morph-closed} and Lemma~\ref{lemma:sk-bij-cl} to prove the following proposition. 

\begin{prop}\label{prop:full-faith-cl}
With the definition on objects given in Lemma~\ref{lemma:closed-skew-fin-mult-to-skew-closed} and on morphisms in Remark~\ref{rmk:def-K-cl-on-funct} we obtain a fully faithful functor $K_{cl}^s\colon\fsmulti_{cl}^s\to\skcl$. 
\end{prop}

\begin{proof}
Just like in Proposition~\ref{prop:sk-full-faith} it is enough to show that $(F^l_0,F^t_2)$ satisfy the conditions of Lemma~\ref{lemma:char-sk-morph-closed} if and only if $(f_0,f_{a,b})$ satisfy the axioms of closed functor. We leave the correspondence in Tabel~\ref{tab:sk-axioms-closed-funct} below. The calculations are completely analogous to the ones in the left representable case, using the closedness techniques of this section. 

\begin{table}[h]
\centering 
\renewcommand\arraystretch{1.5}
\begin{tabular}{|c | c|}
\hline 
$(F^l_0,F^t_2)$ & $(f_0,f_{a,b})$ \\
\hline
 \eqref{eq:lemma-sk-morph-nullary-cl} & $I$ axiom \cite[Axiom~2.6]{Street:skew-closed}\\
\eqref{eq:lemma-sk-morph-bin-cl} & $L$ axiom \cite[Axiom~2.8]{Street:skew-closed} \\
\eqref{eq:lemma-sk-morph-j-cl} & $J$ axiom \cite[Axiom~2.7]{Street:skew-closed}\\
\hline
\end{tabular}
\caption{}\label{tab:sk-axioms-closed-funct}
\end{table}
\end{proof}

We now have defined a (fully faithful) functor $K_{cl}^s$ which fits in diagram \eqref{diag-eqv-closed}. Moreover, comparing the construction of $K_{cl}^s$ with that given in the equivalence $T_{cl}^s$ of \cite[Theorem~6.6]{LackBourke:skew}, we see that the diagram \eqref{diag-eqv-closed} is commutative. 

\begin{theorem}\label{them:sk-fin-equiv-closed}
The functor $K_{cl}^s\colon\fsclosed \to \skcl$ is an equivalence of categories, as is the forgetful functor $U_{cl}^s\colon\fsclosed\to\smulticlosed$.
\end{theorem}

\begin{proof} 
Let us show that $K_{cl}^s$ is an equivalence first. Since $K_{cl}^s$ is fully faithful by Proposition~\ref{prop:full-faith-cl}, it remains to show that is essentially surjective on objects. Since $T_{cl}^s=K_{cl}^sU_{cl}^s$ and the equivalence $T_{cl}^s$ is essentially surjective, so is $K_{cl}^s$, as required. Finally, since $T_{cl}^s=K_{cl}^sU_{cl}^s$ and both $T_{cl}^s$ and $K_{cl}^s$ are equivalences, so is $U_{cl}^s$. 
\end{proof}

Considering a short multicategory $\mc$ as a short skew multicategory with $\mc_n^t(\overline{a};b)=\mc_n^l(\overline{a};b)$, from Theorem~\ref{them:sk-fin-equiv-closed} we can derive a similar result for short multicategories. 

\begin{cor}
The equivalences $K_{cl}^s$ and $U_{cl}^s$ restricts to multicategories, forming the following commutative triangle of equivalences.
\[\begin{tikzcd}[ampersand replacement=\&]
	{\multicl} \\
	\& {\fclosed} \\
	\closed
	\arrow["{U_{cl}}", from=1-1, to=2-2]
	\arrow["{T_{cl}}"', from=1-1, to=3-1]
	\arrow["{K_{cl}}", from=2-2, to=3-1]
\end{tikzcd}\]
\end{cor}

\begin{proof}
It suffices to notice that the equivalence $T^s_{cl}$ restricts to the equivalence $T_{cl}\colon\multicl\to\closed$ between closed multicategories with units and closed categories, which was proven in \cite{Manzyuk2012Closed}. 
\end{proof}

\black

\section{Some Naturality Conditions}
\label{app:nat-conditions}

\subsection{Naturality and Dinaturality Axioms for Substitutions}
\label{app:nat-dinat-axioms}

Here we write explicitly the naturality and dinaturality requirements in Proposition~\ref{prop:classical}.
\begin{itemize}
	\item Naturality in $a_1,\ldots, a_n$ means that, for any $\overline{f}=(f_1,\ldots,f_m)\colon\overline{a}\to\overline{a'}\in\catc^m$, the following diagram is commutative. 
	\[\begin{tikzcd}[ampersand replacement=\&]
		{\mlc_n(\overline{b};c) \times \mlc_m(\overline{a};b_i)} \& {\mlc_{n+m-1}(\overline{b}_{<i},\overline{a},\overline{b}_{>i};c)} \\
		{\mlc_n(\overline{b};c) \times \mlc_m(\overline{a'};b_i)} \& {\mlc_{n+m-1}(\overline{b}_{<i},\overline{a'},\overline{b}_{>i};c)}
		\arrow["{\circ_i}", from=1-1, to=1-2]
		\arrow["{1\times\mlc_m(\overline{f};b_i)}"', from=1-1, to=2-1]
		\arrow["{\mlc_{n+m-1}(\overline{b}_{<i},\overline{f},\overline{b}_{>i};c)}", from=1-2, to=2-2]
		\arrow["{\circ_i}"', from=2-1, to=2-2]
	\end{tikzcd}\]
	\item Similarly, naturality in $c$ means that, for any $h\colon c\to c'\in\catc$, the following diagram is commutative. 
	\[\begin{tikzcd}[ampersand replacement=\&]
		{\mlc_n(\overline{b};c) \times \mlc_m(\overline{a};b_i)} \& {\mlc_{n+m-1}(\overline{b}_{<i},\overline{a},\overline{b}_{>i};c)} \\
		{\mlc_n(\overline{b};c') \times \mlc_m(\overline{a'};b_i)} \& {\mlc_{n+m-1}(\overline{b}_{<i},\overline{a'},\overline{b}_{>i};c')}
		\arrow["{\circ_i}", from=1-1, to=1-2]
		\arrow["{\mlc_n(\overline{b};h)\times1}"', from=1-1, to=2-1]
		\arrow["{\mlc_{n+m-1}(\overline{b}_{<i},\overline{a},\overline{b}_{>i};h)}", from=1-2, to=2-2]
		\arrow["{\circ_i}"', from=2-1, to=2-2]
	\end{tikzcd}\]
	\item Naturality in $b_j$ for $j\neq i$, means that for any $g_j\colon b_j\to b_j'\in\catc$, the diagram below commutes. 
	In order to capture both the case when $j<i$ and $j>i$, in the diagram below we write $-g_j-$ to mean applying $g_j$ to the component relative to $b_j$ and denote with $\overline{x}_{\overline{b},\overline{a},i,j}$ the lists which is the same as $\overline{b}_{<i},\overline{a},\overline{b}_{>i}$ except for $b_j$ which is replaced by $b_j'$. 
\[\begin{tikzcd}[ampersand replacement=\&]
	{\mlc_n(\overline{b};c) \times \mlc_m(\overline{a};b_i)} \& {\mlc_{n+m-1}(\overline{b}_{<i},\overline{a},\overline{b}_{>i};c)} \\
	{\mlc_n(\overline{b}_{<j},b_j',\overline{b}_{>j};c) \times \mlc_m(\overline{a};b_i)} \& {\mlc_{n+m-1}(\overline{x}_{\overline{b},\overline{a},i,j};c)}
	\arrow["{\circ_i}", from=1-1, to=1-2]
	\arrow["{\mlc_n(\overline{b}_{<j},g_j,\overline{b}_{>j};c)\times1}"', from=1-1, to=2-1]
	\arrow["{\mlc_{n+m-1}(-g_j-;c)}", from=1-2, to=2-2]
	\arrow["{\circ_i}"', from=2-1, to=2-2]
\end{tikzcd}\]
	\item Finally, dinaturality in $b_i$ means that, for any $g_i\colon b_i\to b_i'\in\catc$, the following diagram commutes.
	\[\begin{tikzcd}[ampersand replacement=\&]
		{\mlc_n(\overline{b}_{<i},b_i',\overline{b}_{>i};c) \times \mlc_m(\overline{a};b_i)} \&\& {\mlc_n(\overline{b}_{<i},b_i',\overline{b}_{>i};c) \times \mlc_m(\overline{a};b_i')} \\
		{\mlc_n(\overline{b};c) \times \mlc_m(\overline{a};b_i)} \&\& {\mlc_{n+m-1}(\overline{b}_{<i},\overline{a},\overline{b}_{>i};c)}
		\arrow["{1 \times \mlc_m(\overline{a};g_i)}", from=1-1, to=1-3]
		\arrow["{\mlc_n(\overline{b}_{<i},g_i,\overline{b}_{>i};c)\times1}"', from=1-1, to=2-1]
		\arrow["{\circ_i}", from=1-3, to=2-3]
		\arrow["{\circ_i}"', from=2-1, to=2-3]
	\end{tikzcd}\]
\end{itemize}

\subsection{Naturality in Lemma~\ref{lem:bij}}
\label{app:nat-of-f_a}

In this subsection we explicitly write down the naturality condition for $f_{\overline{a}}\colon m(F\overline{a})\to F(m\overline{a})$ in Lemma~\ref{lem:bij}. 

First, we notice that given any family of morphisms $\overline{p}=(p_1,\ldots,p_n)\colon\overline{a}\to\overline{b}\in\catc^n$, we get a corresponding morphism $m\overline{p}\colon m\overline{a}\to m\overline{b}$ given by the unversal property of $m\overline{a}$. 
More precisely, $m\overline{p}$ is the morphism corresponding to $\theta_{\overline{b}}\circ (p_1,\ldots,p_n)$ through the isomorphism
$$\mlc_n(\overline{a};m\overline{b})\cong\mlc_1(m\overline{a};m\overline{b}),$$
i.e. $m\overline{p}$ is the unique morphism such that $m\overline{p}\circ\theta_{\overline{a}}=\theta_{\overline{b}}\circ (p_1,\ldots,p_n)$. 

Now we are ready to formulate the naturality condition for $f_{\overline{a}}$: for any $\overline{p}\colon\overline{a}\to\overline{b}\in\catc^n$, we require the following diagram to commute. 

\[\begin{tikzcd}[ampersand replacement=\&]
	{m(F\overline{a})} \& {F(m\overline{a})} \\
	{m(F\overline{b})} \& {F(m\overline{b})}
	\arrow["{f_{\overline{a}}}", from=1-1, to=1-2]
	\arrow["{m(F\overline{p})}"', from=1-1, to=2-1]
	\arrow["{F(m\overline{p})}", from=1-2, to=2-2]
	\arrow["{f_{\overline{b}}}"', from=2-1, to=2-2]
\end{tikzcd}\]

\bibliography{References}
\bibliographystyle{plain}

\end{document}